\documentclass[a4paper,11pt]{article}

\usepackage[utf8x]{inputenc}
\usepackage{amssymb} 
\usepackage{amsmath}
\usepackage{amsthm}
\usepackage{mathtools}
\usepackage{stmaryrd}
\usepackage{graphicx}
\usepackage{enumitem}
\usepackage{subfigure}

\theoremstyle{plain} 
\newtheorem{theorem}{Theorem}
\newtheorem{proposition}[theorem]{Proposition}
\newtheorem{corollary}[theorem]{Corollary}
\newtheorem{lemma}[theorem]{Lemma}
\theoremstyle{definition}
\newtheorem{remark}[theorem]{Remark}
\newtheorem{definition}[theorem]{Definition}

\newtheorem{example}[theorem]{Example}

\def\keywords#1{\small{\textbf{Keywords:} #1}}
\newenvironment{acknowledgements}{\subsection*{Acknowledgements}}{}

\newcommand{\Cal}[1]{\mathcal{#1}}
\newcommand{\C}{\mathbb{C}}
\newcommand{\CP}{\mathbb{CP}}
\newcommand{\R}{\mathbb{R}}
\newcommand{\Z}{\mathbb{Z}}
\newcommand{\N}{\mathbb{N}}

\newcommand{\res}{\mathrm{Res}}
\newcommand{\widebar}[1]{\mkern 2mu\overline #1}

\newcommand{\Spec}{\mathrm{Spec\,}}
\newcommand{\sminus}{\smallsetminus}

\makeatletter
\newcommand{\@fatnormbar}{\vrule\@width 2\p@}
\newcommand{\fatnorm}[1]{\left.|\kern-2\nulldelimiterspace \@fatnormbar \kern.15em  #1  \kern.15em \@fatnormbar  \right.\kern-\nulldelimiterspace}
\newcommand{\boxnorm}[1]{\left.\kern-2\nulldelimiterspace \talloblong #1 \talloblong \right.\kern-2.5\nulldelimiterspace}
\makeatother

\begin{document}

\author{Martin Klime\v s}
\title{Confluence of singularities of non-linear differential equations via Borel--Laplace transformations}

\maketitle

\begin{abstract}
Borel summable divergent series usually appear when studying solutions of analytic ODE near a multiple singular point. 
Their sum, uniquely defined in certain sectors of the complex plane, is obtained via the Borel--Laplace transformation.
This article shows how to generalize the Borel--Laplace transformation in order to investigate bounded solutions of parameter dependent 
non-linear differential systems with two simple (regular) singular points unfolding a double (irregular) singularity.
We construct parametric solutions on domains attached to both singularities, that converge locally uniformly to the sectoral Borel sums.
Our approach provides a unified treatment for all values of the complex parameter.

\medskip\noindent
\end{abstract}

\section{Introduction}

When studying formal solutions of complex analytic ODE near a multiple singular point, it is the general rule to find divergent series. 
However, one can always construct true analytic solutions, defined on certain sectors attached to the singularity, 
which are asymptotic to the formal solution, and which are in some sense unique.
In general, the solutions on different sectors do not coincide, and if extended to larger sectors, they may drastically change their asymptotic behavior due to the presence of hidden exponentially small terms, known as the (non-linear) Stokes phenomenon.
In case where the singularity is a generic double point such sectoral solutions are obtained from the formal one using  Borel-Laplace summation procedure.
It is now understood, that the divergence of the formal asymptotic series is caused by singularities of its Borel transform, which also encode information on the geometry of the singularity.
Another way how to understand the Stokes phenomena is by considering generic parameter depending deformations which split the multiple (irregular) singular point into several simple (regular) singularities:
it turns out that the local analytic solutions at each singular point of the deformed equation in general do not match, 
thus explaining why solutions with nice asymptotic behavior at the limit when the singular points coalesce only exist in sectors.

When investigating families of analytic systems of ODEs depending on a complex parameter, that unfold a multiple singularity,
one is faced with the problem that the Borel method of summation of formal series does not allow to deal with several singularities and their confluence.
One of our goals here is to show how one can generalize (unfold) the Borel and Laplace operators in case of a generic singularity of multiplicity 2.

In this article we are investigating parametric families of first order non-linear differential systems unfolding 
a double singularity
\begin{equation}\label{eq:BL-system}
(x^2\!-\epsilon)\frac{dy}{dx}=My+f(x,\epsilon,y),\qquad (x,\epsilon,y)\in\C\times\C\times\C^m,
\end{equation}
with $M$ an invertible $m\!\times\! m$-matrix, $f(x,\epsilon,y)$ an analytic germ, $\,D_yf(0,0,0)=0$, and where $\epsilon\in\C$ is a small parameter. 
We study bounded parametric solutions of \eqref{eq:BL-system} near the singular points $x=\pm\sqrt\epsilon$ and their limits when $\epsilon\to 0$.  
Such solutions correspond to \emph{ramified} center manifolds of an unfolded codimension 1 saddle-node singularity in a family of complex vector fields
$$\dot x=x^2\!-\epsilon,\qquad \dot y=My+f(x,\epsilon,y).$$ 

For $\epsilon=0$, the divergence of the formal power series solution of \eqref{eq:BL-system} means that an analytic center manifold does not exist. Instead there are ``sectoral center manifolds'' corresponding to the Borel sums of the divergent series.
 
For $\epsilon\neq 0$, it is well known that for non-resonant values of the parameter 
there exists a local analytic solution on a neighborhood of each singularity $x=\pm\sqrt\epsilon$.
Previous studies of the confluence phenomenon \cite{SS}, \cite{G} have focused at the limit behavior of these local solutions when $\epsilon\to 0$. 
Because the resonant values of $\epsilon$ accumulate at $0$ in a finite number of directions, these directions of resonance in the parameter space could not be covered in those studies, except if the spectrum of $M$ was of Poincaré type.
Here we make no assumption other that $M$ is invertible. 

We will construct a new kind of parametric solutions of systems \eqref{eq:BL-system} which are defined and bounded on certain ramified domains attached to both singularities
$x=\pm\sqrt\epsilon$ (at which they possess a limit) in a spiraling manner. 
They depend analytically on  $\sqrt\epsilon$ from a sector of opening $>\pi$, thus covering a full neighborhood of the origin in the parameter $\epsilon$ space (including those parameters values for which the unfolded system is resonant), and they converge uniformly when $\epsilon\to 0$ to a pair of the classical sectoral solutions: 
Borel sums of the formal power series solution of the limit system, defined on two sectors covering a full neighborhood of the double singularity at the origin. 
In fact, each such pair of the sectoral Borel sums for $\epsilon=0$ unfolds to a unique above mentioned parametric solution.

We provide three different and complementary interpretations of these unfolded sectoral solutions: 
\begin{itemize}\parskip=3pt
\item[i)] Using unfolded Borel and Laplace transformations:
This is the principal approach of this article, with an advantage that it provides a unified treatment for all values of the parameter $\epsilon$ and explains the form of natural domains on which the solutions exist and are bounded.
Most importantly, it gives an insight to intrinsic properties of the singularity and to the source of the divergence similar to that provided by the classical Borel--Laplace approach.
\item[ii)] Using the Hadamard-Perron theorem for $\epsilon\neq 0$. 
\item[iii)] Interpreting them as certain Borel sums of the unique formal power series in $(x,\epsilon)$ 
solving \eqref{eq:BL-system}, which in turn is their asymptotic expansion. 
An important consequence of this correspondence is that the formal and the unfolded sectoral solutions satisfy the same $\{\partial_x,\partial_\epsilon\}$-partial differential relations.
\end{itemize}

These solutions were previously constructed by other methods in the special cases of dimension $m=1$ and a general multiplicity of the singular point \cite{RT},
and in the case of Riccati systems corresponding to normalizing transformations for families of linear differential systems unfolding a non-resonant irregular singularity \cite{LR}, \cite{HLR}, which motivated our present study. 
All our results translate directly to this situation, $y$ playing role of such normalizing transformation (Section~\ref{subsec:BL-1.3} below).

\section{Statement of results}

Notation:
Throughout the text $]a,b[$ (resp. $[a,b]$) denotes the open (resp. closed) oriented segment between two points $a,b\in\C$; 
$\,e^{i\alpha}\R^+=[0,+\infty e^{i\alpha}[$ is an oriented ray,
and $c+e^{i\alpha}\R=]c-\infty e^{i\alpha},c+\infty e^{i\alpha}[$, with $\alpha\in\R$, $c\in\C$, is an oriented line.

\subsection{Borel--Laplace transformations and their unfolding}

The \emph{Borel method of summation of (1-summable) divergent series} is used to construct their sectoral Borel sums: 
unique analytic functions that are \emph{asymptotic} to the series in certain sectors of opening $>\pi$ at the singular point and 
satisfy the same differential relations.

Let $\,\hat y(x)=\sum_{k=1}^{+\infty} y_k\, x^k\,$ be a formal power series.
Using the Euler formula for the $\Gamma$-function: $\Gamma(k)=\int_0^{+\infty} z^{k-1} e^{-z}\,dz$, equal to $(k-1)!$ if $k\in\N_{>0}$,
one can write $x^k=\int_0^{+\infty e^{i\alpha}} \!\!\!\frac{\xi^{k-1}}{(k-1)!}\,e^{-\frac{\xi}{x}}\,d\xi$,
for $x$ in the half-plane $\Re e^{i\alpha}x>0$.   
Hence 
$$\hat y(x)=\sum_{k=1}^{+\infty} y_k\, x^k=\sum_{k=1}^{+\infty}\int_0^{+\infty e^{i\alpha}}\!\!\! \frac{y_k}{(k-1)!}\,\xi^{k-1}\cdot e^{-\frac{\xi}{x}}\,d\xi.$$
The \emph{formal Borel transform} of $\hat y$ is the series
\begin{equation}\label{eq:BL-formalborel}
 \widehat B[\hat y](\xi)=\sum_{k=1}^{+\infty}\frac{y_k}{(k-1)!}\,\xi^{k-1}.
\end{equation}
If the coefficients of $\hat y(x)$ have at most factorial growth ($\,|y_k|\leq c^k k!\,$ for some $c>0$), 
then the series $\widehat B[\hat y](\xi)$ is convergent on a neighborhood of 0 with a sum $\phi(\xi)$.
If moreover $\phi$ has an analytic  extension  to a half-line $e^{i\alpha}\R^+$ and has at most exponential growth there
($|\phi(x)|\leq K\,e^{\Lambda|\xi|}$, $\xi\in e^{i\alpha}\R^+$, for some $K,\Lambda > 0$),
then its \emph{Laplace transform in the direction $\alpha$}
\begin{equation}\label{eq:BL-laplace}
 L_\alpha[\phi](x)=\int_0^{+\infty e^{i\alpha}} \!\!\!\phi(\xi)\cdot e^{-\frac{\xi}{x}}\,d\xi
\end{equation}
is convergent for $x$ in a small open disc of diameter $\frac{1}{\Lambda}$ centered at $\frac{e^{i\alpha}}{2\Lambda}$ and extends to 0 (which lies on the boundary of the disc), 
defining there the \emph{Borel sum} of $\hat y(x)$ in direction $\alpha$.
A series $\hat y[x]$ is \emph{Borel summable} (or \emph{1-summable}) 
if its Borel sums $L_\alpha[\widehat B[\hat y]](x)$ exist in all but finitely many directions $0\leq\alpha<2\pi$.
When varying continuously the direction in which the series is summable, the corresponding Borel sums are analytic extensions one of the other, yielding a function defined on a sector of opening $>\pi$.

Let us remark that $\hat y[x]$ is convergent if and only if it is Borel summable in all directions.
This means that the Borel sums of divergent series can only exist on sectors. This is also known as the Stokes phenomenon.

The Borel sums of  $\hat y(x)$ are asymptotic of Gevrey order 1 to the formal series $\hat y(x)$  at the origin, and 
most importantly, they satisfy the same analytic differential equations as $\hat y(x)$.
More detailed information on the Borel summability can be found, for example, in \cite{MR2}, \cite{M} or \cite{Bal}.

\smallskip
A typical source of Borel summable power series are formal solutions of generic ODEs at an irregular singular point of multiplicity 2. 

\begin{example}\label{example:BL-1}
 A non-homogeneous linear analytic ODE with a double singularity at the origin
  \begin{equation}\label{eq:BL-cm0}
  x^{2}\frac{dy}{dx}=y+ f(x), \quad (x,y)\in\C\times\C,
  \end{equation}
where $\,f(x)\in x\,\C\{x\}\,$ is a convergent power series, possesses a unique formal solution $\hat y(x)$. 
Generically, this series is divergent (for instance if $f(x)=-x$ then $\hat y(x)=\sum_{n=1}^{+\infty}(n-1)!x^{n}$ is the Euler series). 
The formal Borel transform of the equation \eqref{eq:BL-cm0} is
$$\xi\cdot\widehat B[\hat y](\xi)=\widehat B[\hat y](\xi)+\widehat B[f](\xi),$$
hence the reason for the divergence of $\hat y(x)$ is materialized by the singularity of  
$\widehat B[\hat y](\xi)=\frac{\widehat B[f](\xi)}{\xi-1}$ at $\xi=1$.
The Borel sum $y(x)=L_\alpha[\widehat B[\hat y]](x)$ of $\hat y(x)$, $\alpha\in\,]0,2\pi[$, is a solution to \eqref{eq:BL-cm0}, well defined in a ramified sector $\arg x\in\,]-\frac{\pi}{2}+\eta,\frac{5\pi}{2}-\eta[$ for any $\eta>0$.
The set $(x,y(x))$ is a \emph{center manifold of a saddle-node} singularity of the vector field
$$\dot x = x^{2},\quad  \dot y = y+ f(x).$$
This example shows that in general an analytic center manifold does not exist, but instead there are ``sectoral center manifolds''.
\end{example}

\smallskip

The inverse to the Laplace transformation in direction $\alpha$ is the analytic Borel transformation:
If  $y(x)$ is a germ of function analytic on a closed sector of opening $\geq\pi$ bisected by $e^{i\alpha}\R^+$ that vanishes at 0 as $O(x^\lambda)$ uniformly in the sector for some $\lambda>0$, then its \emph{analytic Borel transform in direction $\alpha$} is defined as the ``Cauchy principal value'' ($V.P.$) of the integral
\begin{equation}\label{eq:BL-1}
B_\alpha[y](\xi) = \tfrac{1}{2\pi i} \, V.P.\!\int_{\gamma} \!y(x)\,e^{\frac{\xi}{x}}\,\tfrac{dx}{x^2}, \qquad\text{for }\ \xi\in e^{i\alpha}\R, 
\end{equation}
over a circle $\gamma=\{\Re\big(\tfrac{e^{i\alpha}}{x}\big)=C\}$, $C>0$, inside the sector.
The plane of $\xi$ is also called the \emph{Borel plane}.

The formal Borel transform \eqref{eq:BL-formalborel} of an analytic germ $y$ vanishing at 0 is related to the analytic one by
\begin{equation}\label{eq:BL-formalavsnalyticborel}
B_\alpha[y](\xi)=\chi_\alpha^+(\xi)\cdot\widehat  B[y](\xi),\qquad\text{for }\ \xi\in e^{i\alpha}\R, 
\end{equation}
where
\begin{equation*}
\chi_\alpha^+(\xi)=\left\{ \begin{array}{ll}     
     1, & \text{if $\ \xi\in\,]0,+\infty e^{i\alpha}[$,}\\[3pt]
     0, & \text{if $\ \xi\in\,]\!-\!\infty e^{i\alpha},0[$.}
   \end{array}\right. 
\end{equation*}

\smallskip
The idea of unfolding the Borel--Laplace operators in order to generalize the methods of Borel summability and resurgent analysis to systems with several confluent singularities
was initially brought up by Sternin and Shatalov in \cite{SS}.
The key lies in appropriate unfolding of the ``kernels'' $\,e^{\frac{\xi}{x}}\,\tfrac{dx}{x^2}$ and $\,e^{-\frac{\xi}{x}}\,d\xi$ 
of the transformations \eqref{eq:BL-laplace} and \eqref{eq:BL-1}, and in right determination of the paths of integration.
The Borel transformation is designed so that it converts the derivation $x^2\frac{d}{dx}$ to multiplication by $\xi$, and we will want to preserve this property.

The complex vector field  $\, x^2\tfrac{\partial}{\partial x} \,$
with a double singularity at the origin is naturally (and universally) unfolded as
\begin{equation}\label{eq:BL-vectfield}
  (x^2\!-\epsilon)\tfrac{\partial}{\partial x},\quad \epsilon\in\C.
\end{equation}
We will associate to it the \emph{unfolded Borel} and \emph{Laplace transformations}
\begin{equation}\label{eq:BL-unfoldedBL}
\begin{aligned}
\Cal B_\alpha^+[y](\xi,\sqrt\epsilon) &= \tfrac{1}{2\pi i} \, V.P.\!\int_{\Re e^{i\alpha}t(x,\epsilon)=C} \!y(x)\,e^{t(x,\epsilon)\xi}\,dt(x),\quad 
0<C<\Re\big( e^{i\alpha}\tfrac{\pi i}{\sqrt\epsilon} \big), \\
\Cal L_\alpha[\phi](x,\sqrt\epsilon) &=\int_{-\infty e^{i\alpha}}^{+\infty e^{i\alpha}} \phi(\xi)\,e^{-t(x,\epsilon)\xi}\,d\xi,
\end{aligned}
\end{equation}
where 
\begin{equation}\label{eq:BL-t}
\hskip6pt 
t(x,\epsilon)=-\int\frac{dx}{x^2\!-\!\epsilon}:=
\left\{  \begin{array}{ll}
     -\frac{1}{2\sqrt\epsilon}\log\frac{x-\sqrt\epsilon}{x+\sqrt\epsilon}, &  \text{if $\ \epsilon\neq 0$,}\\[6pt]
     \frac{1}{x}, & \text{if $\ \epsilon=0$,}\\
   \end{array}\right.
\end{equation} 
is the negative complex time of the vector field \eqref{eq:BL-vectfield}.
Let us remark that the (unilateral) Laplace transformation $L_\alpha[\phi]\,$  \eqref{eq:BL-laplace} is equal to the (bilateral) Laplace transformation 
$\Cal L_\alpha[\phi]$ with $\epsilon\!=\!0$ and $t(x,0)=\frac{1}{x}$, 
if one extends the integrand by 0
for $\xi\in\,]\!-\!\infty e^{i\alpha},0[\,$:
$$L_\alpha[\phi](x)=\Cal L_\alpha[\chi_\alpha^+\phi](x,0).$$

In Sections~\ref{sec:BL-2} and \ref{sec:BL-3} we will establish some general properties of these transformations based on the classical theory of Fourier and Laplace integrals, and in Section~\ref{sec:BL-4} we will apply them to study solutions of \eqref{eq:BL-system} in the vicinity of the singular points.

\subsection{Center manifold of an unfolded codimension 1 saddle--node type singularity}\label{subsec:BL-1.2}

An isolated singular point of a holomorphic vector field in $\C^{m+1}$ is of \emph{saddle--node} type if its linearization matrix has exactly one zero eigenvalue, and 
it is of \emph{codimension 1} if the multiplicity of the singular point is 2. 

We consider an analytic family of vector fields in $\C^{m+1}$ unfolding such saddle node singularity  in form
\begin{equation}\label{eq:BL-vectsyst}
\dot x=x^2\!-\epsilon,\quad \dot y=My+ f(x,\epsilon,y), \qquad (x,\epsilon,y)\in\C\times\C\times\C^m,
\end{equation}
with $\epsilon\in\C$  a small parameter, $M$ \emph{invertible}, and 
$f$ a germ of analytic vector function at the origin of $\C\times\C\times\C^m$ with
\begin{equation}\label{eq:BL-f}
D_{y}f(0,0,0)=0,\quad\text{and}\quad f(x,\epsilon,0)=O(x^2\!-\epsilon). 
\footnote{If instead $f(x,\epsilon,0)$ was only $O(|x|+|\epsilon|)$, and $u_{\pm\sqrt\epsilon}\in\C^m$ were the unique solutions of 
$0=M u_{\pm\sqrt\epsilon}+f(\pm\sqrt\epsilon,u_{\pm\sqrt\epsilon},\epsilon)$, with $\,u_{\pm 0}=0$, then the change of variable $y\,\mapsto\, y-\tfrac{1}{2\sqrt\epsilon}\Big(u_{+\sqrt\epsilon}(x\!+\!\sqrt\epsilon)-u_{-\sqrt\epsilon}(x\!-\!\sqrt\epsilon)\Big),$ analytic in $(x,\epsilon)$, would bring the system \eqref{eq:BL-vectsyst} to a one with $f(x,\epsilon,0)=O(x^2\!-\epsilon)$.}
\end{equation} 
For $\epsilon=0$ the vector field \eqref{eq:BL-vectsyst} possesses a ramified 1-dimensional ``center manifold'' consisting of several sectoral pieces tangent to the $x$-axis.
Here we study its parametric unfolding in the family: It is 
given as a graph of a function $y=y(x,\sqrt\epsilon)$, ramified at $x=\pm\sqrt\epsilon$,
satisfying the singular non-linear system of $m$ ordinary differential equations
\begin{equation}\label{eq:BL-cm}
(x^2\!-\epsilon)\frac{dy}{dx}=My+ f(x,\epsilon,y), \qquad (x,\epsilon,y)\in \C\times\C\times\C^m.
\end{equation}


\begin{remark}
If $f(x,\epsilon,0)=0$, then \eqref{eq:BL-cm} has a unique analytic solution given by $y=0$. Being trivial, this case bears no interest in this article. Reciprocally, if \eqref{eq:BL-cm}  has an analytic solution $y=\phi(x,\epsilon)$, then the change of variable $y\mapsto y-\phi(x,\epsilon)$ brings \eqref{eq:BL-vectsyst} to a form with $f(x,\epsilon,0)=0$.
\end{remark}

\begin{remark}
In dimension $m=1$, families of vector fields unfolding a saddle-node of codimension $k$  were thoroughly studied in \cite{RT}.
\end{remark} 

\begin{remark}
A general analytic family of vector fields in $\C^{m+1}$ unfolding a saddle-node singularity of codimension 1 to two simple singularities is locally orbitally analytically equivalent to
\begin{equation}\label{eq:BL-vectsyst1}
\dot x=(x^2\!-\epsilon) +G(x,\epsilon,y),\quad \dot y=My+ F(x,\epsilon,y), 
\end{equation}
with $G(x,\epsilon, y)=o(|y|)$, $F(x,\epsilon,y)=O(x^2\!-\epsilon)+o(|y|)$.
The singular transformation (blow-up) $y=(x^2\!-\epsilon)u$ brings \eqref{eq:BL-vectsyst1} to 
\begin{equation*}
\dot x=(x^2\!-\epsilon)\big(1+\tfrac{G(x,\epsilon,(x^2\!-\epsilon)u)}{(x^2\!-\epsilon)}\big) ,\quad 
\dot u=Mu+\tfrac{F(x,\epsilon,(x^2\!-\epsilon)u)}{(x^2\!-\epsilon)}-2x\big(1+\tfrac{G(x,\epsilon,(x^2\!-\epsilon)u)}{(x^2\!-\epsilon)}\big)u,
\end{equation*}
which by the assumption is analytic near $0\in\C\times\C\times\C^m$. A transformation sending its two singularities to the points $(x,u)=(\pm\sqrt\epsilon, 0)$ and a division by a non-vanishing germ reduces it to the from \eqref{eq:BL-vectsyst}. 
\footnote{For $m=1$, it's been shown in \cite[Proposition 3.1]{RT}, cf. also \cite[Lemma 1]{G}, 
that the the family \eqref{eq:BL-vectsyst1} is in fact locally orbitally analytically equivalent to a family \eqref{eq:BL-vectsyst}.} 
\end{remark}

\subsubsection{Formal solution.}

\begin{proposition}[Formal solution]\label{prop:BL-formalsolution}
The system \eqref{eq:BL-cm} possesses a unique solution in terms of a formal power series in $(x,\epsilon)$:
\begin{equation}\label{eq:BL-formalsolution}
 \hat y(x,\epsilon)=\sum_{k,j=0}^{+\infty}y_{kj}x^k\epsilon^j,\qquad y_{kj}\in\C^m.
\end{equation}
This series is divisible by $(x^2\!-\epsilon)$, and its coefficients satisfy 
$\|y_{kj}\|\leq L^{k+2j}(k+2j-1)!$ for some $L>0$.
\end{proposition}

\begin{proof}
Write $\hat y(x,\epsilon)=(x^2\!-\epsilon)\sum_{k,j}u_{kj} x^k\epsilon^j$, $u_{kj}\in\C^m$, 
$$\,f(x,\epsilon,y)=\sum_{|l|\geq 0}\sum_{k,j}f_{l,k,j}x^k\epsilon^j y^l, \qquad  f_{l,kj}\in\C^m,$$  
where $y^l:=y_1^{l_1}\cdot\ldots\cdot y_m^{l_m}$ for each multi-index $l=(l_1,\ldots,l_m)\in\N^m$, and $|l|=l_1+\ldots+l_m$.
Substituting $\hat y(x,\epsilon)$ for $y$ in $f$, dividing the equation \eqref{eq:BL-cm} by $(x^2\!-\epsilon)$,
and comparing the coefficients of $x^k\epsilon^j$, one obtains a set of equations
$$Mu_{kj}=(k\!+\!1)(u_{k-1,j}-u_{k+1,j-1}) +P_{kj},$$
where $P_{kj}$ is a polynomial in 
$\{u_{k'j'}\mid k'\leq k,\ j'\leq j,\ k'\!+\!2j'< k\!+\!2j\}$.
By recursion with respect to the order $k+2j$ of the indices $(k,j)$ this uniquely determines all the coefficient vectors $u_{kj}$, and therefore also the coefficient vectors $y_{kj}=u_{k-2,j}-u_{k,j-1}$. 

Similarly, $y_{kj}$ (note that $k+2j\geq 2$) satisfy recursive equations
$$My_{kj}=(k\!+\!1)y_{k+1,j-1}-(k\!-\!1)y_{k-1,j} +Q_{kj},$$
where $Q_{kj}$ is a polynomial in 
$\{y_{k'j'}\mid k'\leq k,\ j'\leq j,\ k'+2j'< k+2j\}$, given by
$$Q_{kj}=-f_{0,k,j}\,-\sum_{\!|l|=1}\!\!\!\!\!\!\! \sum_{\substack{\kappa\leq k \\ \iota\leq j\\ \ 2\leq\kappa+2\iota<k+2j\!\!\!\! }}\!\!\!\!\!\!\! f_{l,k-\kappa,j-\iota}a_{l,\kappa,\iota}
\,-\sum_{\!|l|\geq 2}\!\!  \sum_{\substack{\ \kappa\leq k\\ \ \iota\leq j\\ 2|l|\leq\kappa+2\iota }}\!\! f_{l,k-\kappa,j-\iota} a_{l,\kappa,\iota},$$
where $a_{l,\kappa,\iota}$ are polynomials in 
$\{y_{k'j'}\mid k'\leq \kappa,\ j'\leq \iota,\ k'+2j'\leq \kappa+2\iota-2|l|\}$ determined by
$\left(\hat y(x,\epsilon)\right)^l=\sum_{\kappa+2\iota \geq 2|l|}a_{l,\kappa,\iota}x^\kappa\epsilon^\iota$.
We want to show by induction on the order $k+2j$ that 
$\|y_{kj}\|\leq L^{k+2j}(k\!+\!2j\!-\!1)!$, for some $L>0$ independent of $(k,j)$, for
$\|\cdot\|$ the maximum norm on $\C^m$.

Let us first show that if $\|y_{k'j'}\|\leq L^{k'+2j'}(k'\!+\!2j'\!-\!1)!$ for each
$k'\!\leq\!\kappa,\ j'\!\leq\!\iota,\ k'\!+\!2j'\leq\kappa\!+\!2\iota\!-\!2|l|$, then 
$$\|a_{l,\kappa,\iota}\|\leq L^{\kappa+2\iota}2^{|l|-1}\tfrac{(\kappa+2\iota-|l|)!}{|l|!},\quad 2|l|\leq\kappa+2\iota.$$
This is certainly true if $|l|=1$. If $|l|\geq 2$, then
$\|a_{l,\kappa,\iota}\|\leq\sum_{k',j'}\|a_{l',\kappa-k',\iota-j'}\|\cdot\|y_{k',j'}\|$,
where $l'$ is a smaller multi-index, $|l'|=|l|-1$, and the sum is taken over the set of indices
$\{(k',j')\mid k'\!\leq\!\kappa,\ j'\!\leq\!\iota,\ 2\leq k'\!+\!2j'\leq \kappa\!+\!2\iota\!-\!2|l|\!+\!2\}$
whose cardinality is equal $(\kappa\!+\!1)(\iota\!+\!1)-|l|(|l|\!+\!1)-2<(\kappa\!+\!2\iota\!-\!|l|\!-\!1)(\kappa\!+\!2\iota\!-\!|l|)$.
We can write $(\kappa\!+\!2\iota\!-\!k'\!-\!2j'\!-\!|l|\!+\!1)!\cdot(k'\!+\!2j'\!-\!1)!\leq \frac{2}{|l|}(\kappa\!+\!2\iota\!-\!|l|\!-\!2)!$ since the first factorial is at least $(|l|\!-\!1)!$. Hence the estimate.

Now if $\|f_{l,k,j}\|\leq L_1^{k+2j}L_2^{|l|}$ for some $L_1,L_2>0$, then
\begin{align*}
\|Q_{kj}\|\tfrac{ L^{-k-2j}}{(k+2j-1)!} &\leq \ 
\left(\tfrac{L_1}{L}\right)^{\!k+2j}
+\tfrac{mL_2}{k+2j-1}\sum_{\kappa,\iota} \tfrac{(\kappa+2\iota-1)!}{(k+2j-2)!}\left(\tfrac{L_1}{L}\right)^{\!k+2j-\kappa-2\iota}+\\
&+\!\!\sum_{p=2}^{\lfloor\frac{k+2j}{2}\rfloor}\!\!\tfrac{(2mL_2)^p}{p!(k+2j-p+1)\ldots(k+2j-1)}
\sum_{\kappa,\iota} \tfrac{(\kappa+2\iota-p)!}{(k+2j-p)!}\left(\tfrac{L_1}{L}\right)^{\!k+2j-\kappa-2\iota}.
\end{align*}
For $L>L_1$, each of the two sums over $(\kappa,\iota)$ can be estimated from above by 3, hence, taking into account that $k+2j-p\geq\lceil\frac{k+2j}{2}\rceil$,
$$\|Q_{kj}\|\tfrac{ L^{-k-2j}}{(k+2j-1)!}\leq
\left(\tfrac{L_1}{L}\right)^{\!k+2j}\!\!+\tfrac{3\,mL_2}{k+2j-1}+
\tfrac{3\,(2mL_2)^2}{k+2j-1}\!\!\sum_{p=2}^{\lfloor\frac{k+2j}{2}\rfloor}\!\!\! \tfrac{1}{p!}\big(\!\tfrac{2mL_2}{\lceil\frac{k+2j}{2}\rceil+1}\big)^{\!p-2\!},$$
which can be made less than $\frac{1}{2\|M^{-1}\|}$ supposing that $L$ and $k+2j$ are large enough. 
And the same is true for $\|(k+1)y_{k+1,j-1}-(k-1)y_{k-1,j}\| \leq 2L^{-k-2j-1}\cdot
(k+2j-1)!$.
\qed\end{proof}



\subsubsection{Sectoral center manifold and its unfolding.}

For $\epsilon=0$ it is known that the equation \eqref{eq:BL-cm} has a unique solution in terms of a \emph{Borel summable} formal power series 
$\hat y_0(x)=\hat y(x,0)$ \eqref{eq:BL-formalsolution}. 
This a very classical theorem going essentially back to J. Horn, M. Hukuhara or 
J. Malmquist \cite{M1} among many, whose modern versions can be found for example in works of Braaksma \cite{BB}, Écalle, or Ramis and Sibuya \cite{RS}, (see also \cite{MR1} for the case $m=1$): 

\begin{theorem}\label{theorem:BL-BERS}
For $\epsilon=0$ the formal solution $\hat y_0(x)=\sum_{k=0}^{+\infty}y_{k0}x^k$ of the system
\eqref{eq:BL-vectsyst} is Borel summable in each direction $\alpha$ with $e^{i\alpha}\R^+$ disjoint to $\Spec M$.
Hence to each connected component $\Omega$ of $\,\C\sminus\bigcup_{\lambda\in\Spec M}\lambda\R^+$ in the Borel plane (Figure~\ref{figure:BL-4d1}) corresponds a unique
Borel sum of $\hat y_0(x)$, a solution of the system, defined on a sector in the $x$-plane of opening $>\pi$ and asymptotic to $\hat y_0(x)$.

More generally, for each $j\in\N$, the formal component $\hat y_j(x):=\sum_{k=0}^{+\infty}y_{kj}x^k$ of \eqref{eq:BL-formalsolution} is Borel summable in the same directions.
\begin{figure}[ht]
\centering
\includegraphics[width=0.5\textwidth]{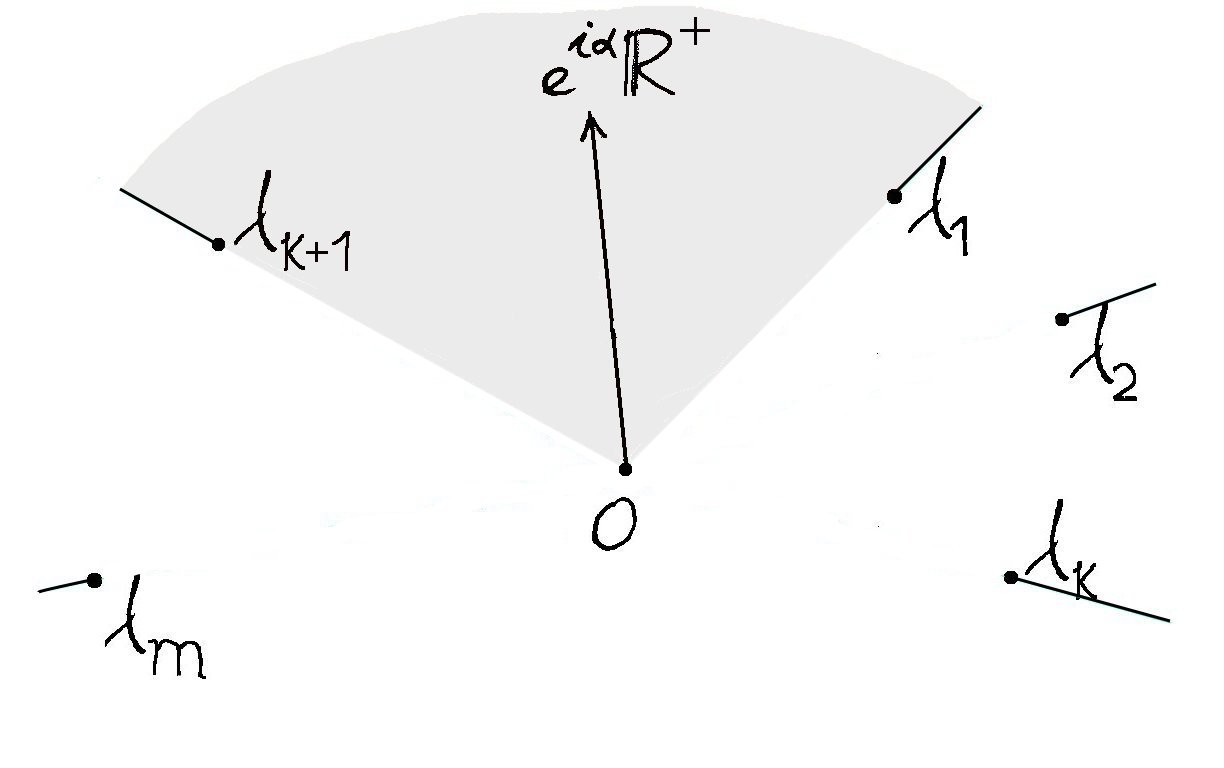}
\caption{The rays $\lambda\R^+$, $\lambda\in\Spec M$ divide the Borel plane in sectors.
The integration path $e^{i\alpha}\R^+$ of the Laplace transform $L_\alpha[\widehat B[\hat y_0]]$ varies in such sectors.
}
\vskip-12pt
\label{figure:BL-4d1}
\end{figure}
\end{theorem}

\begin{proof}
The Borel summability of $\hat y_j(x)$ is obtained by recursion on $j$ using
Theorem 4 in \cite{BB}: each $\hat y_j(x)$ is a formal solution to a system
of differential equations
$$x^2\frac{dy_j}{dx}=My_j+h_j(x)+g_j(x,y_j),\quad g_j(x,y_j)=O(x\|y_j\|),$$
with $h_j, g_j$ depending polynomially on $\hat y_0(x),\ldots,\hat y_{j-1}(x)$, thus Borel summ\-able in $x$ in
directions disjoint to $\Spec M$.
\qed
\end{proof}


This means that for each \emph{two opposite components} $\Omega^+, \Omega^-$ in the Borel plane (i.e. such that $\Omega^+\cup\Omega^-\cup\{0\}$ contains some straight line $e^{i\alpha}\R$),
the two corresponding sectors $X_0^{\pm}$ of summability form a covering of a neighborhood of the origin in the $x$-plane.
Theorem~\ref{theorem:BL-2} below shows that each such covering pair of sectors $\{ X_0^+,X_0^-\}$ unfolds for $\epsilon\neq 0$
to a single ramified domain $X(\!\sqrt\epsilon)$, adherent to both singular points $x=\pm\sqrt\epsilon$ (see Figure~\ref{figure:BL-4a}), 
on which there exists a unique bounded solution $y(x,\sqrt\epsilon)$ of \eqref{eq:BL-cm},
depending analytically on $\sqrt\epsilon$ taken from a sector $S$ of opening  $>\pi$,
that converge uniformly to the two respective Borel sums of $\hat y_0(x)$ on $X_0^+,X_0^-$, when $\sqrt\epsilon\to 0$.  
First we construct these domains.


\begin{definition}[Family of domains $X(\!\sqrt\epsilon)$, $\sqrt\epsilon\in S$]\label{definition:BL-X}
Let $\{\Omega^+,\Omega^-\}$ be a pair of opposite sectoral components of $\C\sminus\bigcup_{\lambda\in\Spec M}\lambda\R^+$, 
and let $\beta_1<\beta_2$ be such that $\bigcup_{\alpha\in\,]\beta_1,\beta_2[}e^{i\alpha}\R\subseteq\Omega^+\cup\Omega^-\cup\{0\}$.
For some $0<\eta<\frac{\beta_2-\beta_1}{4}$, $\rho>0$ and $\Lambda>0$, let
\begin{equation}\label{eq:BL-S}
S=\{\sqrt\epsilon\in\C \mid \arg\sqrt\epsilon\in\,]\beta_1-\pi+2\eta,\beta_2-2\eta[,\ |\sqrt\epsilon|<\rho\}\cup\{0\},
\end{equation}
and for each $\sqrt\epsilon\in S$ let 
\begin{equation}\label{eq:BL-TalphaLambda}
\begin{aligned}
\mathbf{T}_\alpha^+(\Lambda,\sqrt\epsilon):=&\{\Lambda<\Re(e^{i\alpha} t) < -\Re(\tfrac{e^{i\alpha}\pi i}{\sqrt\epsilon})-\Lambda \}, \\
\mathbf{T}_\alpha^-(\Lambda,\sqrt\epsilon):=&\{-\Lambda>\Re(e^{i\alpha} t) > \Re(\tfrac{e^{i\alpha}\pi i}{\sqrt\epsilon})+\Lambda \}, 
\end{aligned}
\end{equation}
be slanted strips in the time $t$-plane in direction $-\alpha+\frac{\pi}{2}$ that pass in between closed discs of radius $\Lambda$ centered at the points $0$ and $\mp\frac{\pi i}{\sqrt\epsilon}$. 
Define
$$T^\pm(\!\sqrt\epsilon)=\bigcup_\alpha \mathbf{T}_\alpha^\pm(\Lambda,\sqrt\epsilon) $$ 
(see Figure~\ref{figure:BL-6}) as their union with varying $\alpha$ \footnote{These $\alpha$ will later correspond to the direction of the unfolded Laplace integrals \eqref{eq:BL-unfoldedBL}, and $\mathbf{T}_\alpha^\pm(\Lambda,\sqrt\epsilon)$ to their strips of convergence.} 
\begin{equation}\label{eq:BL-alpha}
\max\{\arg\sqrt\epsilon,\beta_1\}+\eta<\alpha<\min\{\beta_2,\arg\sqrt\epsilon+\pi\}-\eta.
\end{equation}

\begin{figure}[p]
\centering
\includegraphics[width=\textwidth]{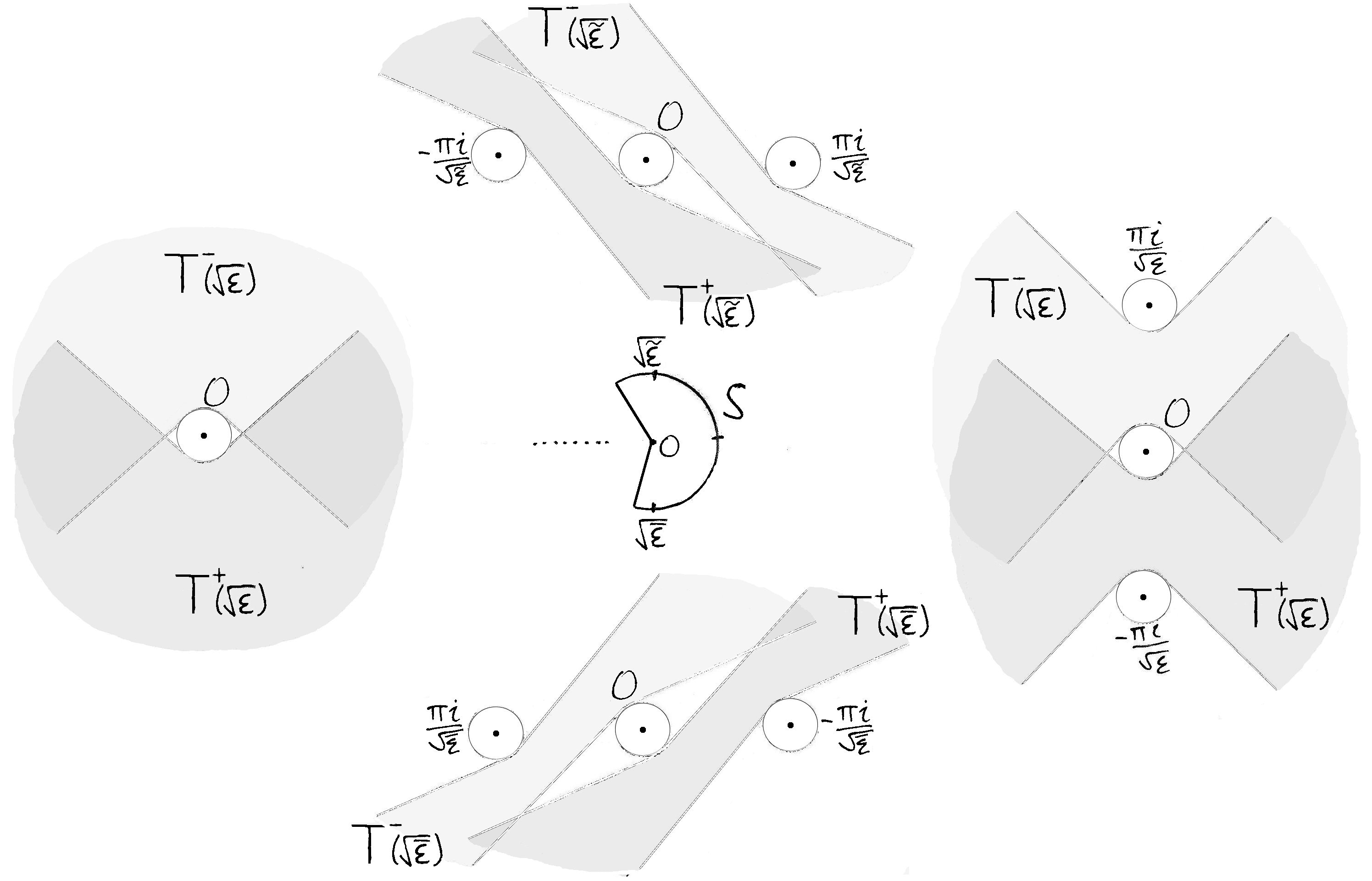}
\vskip6pt
\caption{The domains  $T^\pm(\!\sqrt\epsilon)$ in the time $t$-plane \eqref{eq:BL-t} according to $\sqrt\epsilon\in S$. (Here $\beta_1\sim\frac{\pi}{4},\ \beta_2\sim\frac{3\pi}{4},\ \eta\sim 0$).
}
\label{figure:BL-6}
\bigskip
\centering
\includegraphics[width=0.83\textwidth]{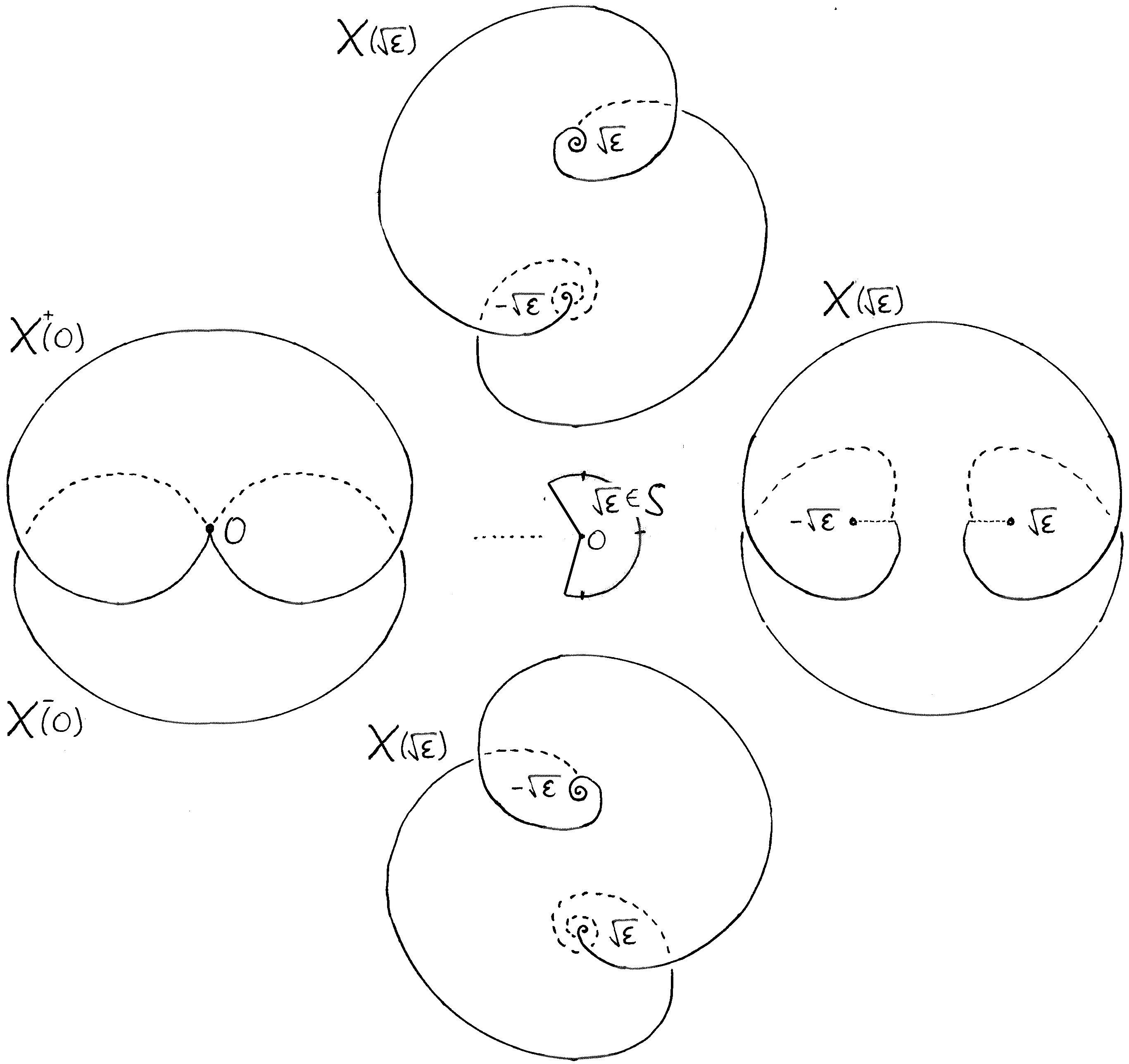}
\caption{Example of the spiraling domains $X(\!\sqrt\epsilon)$ of Theorem~\ref{theorem:BL-2}\,(i) according to $\sqrt\epsilon\in S$.}
\label{figure:BL-4a}
\end{figure}

We define the domains $X^\pm(\!\sqrt\epsilon)$ (see Figure~\ref{figure:BL-4a}) as simply connected ramified projections of $T^\pm(\!\sqrt\epsilon)$ to the
$x$-coordinate \footnote{More precisely to a covering space of the $x$-plane ramified at  $\{\sqrt\epsilon,-\!\sqrt\epsilon\}$, the Riemann surface of $t(x,\epsilon)$ \eqref{eq:BL-t}.} 
by the map
\begin{equation}\label{eq:BL-projection_x}
t\mapsto x(t,\epsilon)=\left\{ 
   \begin{array}{ll}
     \sqrt\epsilon\frac{1+e^{-2\sqrt\epsilon t}}{1-e^{-2\sqrt\epsilon t}}=\sqrt\epsilon\coth\sqrt\epsilon t, \quad &  \text{if $\ \epsilon\neq 0$,}\\[6pt]
     \frac{1}{t}, & \text{if $\ \epsilon=0$,}\\
   \end{array}\right.
\end{equation}
the inverse of \eqref{eq:BL-t},
to which we adjoin the ramification points $\{\sqrt\epsilon,-\!\sqrt\epsilon\}$
(which are approached from within the interior of 
$X(\!\sqrt\epsilon)$ following logarithmic spirals).
For $\sqrt\epsilon\neq 0$, $X^+(\!\sqrt\epsilon)=X^-(\!\sqrt\epsilon)=:X(\!\sqrt\epsilon)$.
On the other hand, $X^+(0)$ is a sectoral domain with $\arg x\in\,]\beta_1-\frac{\pi}{2}+\eta,\beta_2+\frac{\pi}{2}-\eta[$,
while $X^-(0)$ is its opposite -- we define $X(0)$ as their ramified union with $0$ as the only common point.

Clearly, the domains $X(\!\sqrt\epsilon)$ depend continuously on $\sqrt\epsilon\in S\sminus\{0\}$ and
they converge, when $\sqrt\epsilon\to 0$ radially with $\arg\sqrt\epsilon=\beta$, 
to a pair of sub-domains $X^\pm_\beta(0)$ of $X^\pm(0)$, and $X^\pm(0)$ is the union of all these radial limits.
  
If the radius $\rho$ of $S$ is taken small enough, then there exists a fixed neighborhood of the origin in the $x$-plane covered by each domain $X(\!\sqrt\epsilon)$, $\sqrt\epsilon\in S$. 
\end{definition}

\begin{remark}\label{remark:BL-X}
The ramified domains $X(\!\sqrt\epsilon)$ are swept by complete real trajectories of the complex vector fields
$e^{i(\frac{\pi i}{2}-\alpha)}(x^2\!-\epsilon)\frac{\partial}{\partial x}$, with $\alpha$ as in \eqref{eq:BL-alpha}, that stay forever within a neighborhood of $0$ of radius $\sim\frac{1}{\Lambda}$, and tend to the point
$\sqrt\epsilon$ (resp. $-\sqrt\epsilon$) in negative (resp. positive) time.
\end{remark}

\begin{theorem}\label{theorem:BL-2}
Consider a system \eqref{eq:BL-cm} with $M$ an invertible $m\!\times\! m$-matrix and
$\,f(x,\epsilon,y)$ as in \eqref{eq:BL-f}.

\smallskip
\noindent
\textbf{(i)}
Let $\{\Omega^+,\Omega^-\}$ be a pair of opposite components of $\C\sminus\bigcup_{\lambda\in\Spec M}\lambda\R^+\,$ (i.e. such that $\Omega^+\cup\Omega^-\cup\{0\}$ contains some straight line $e^{i\alpha}\R$).
For any arbitrarily small angle $\nu>0$ there are $\Lambda,\rho>0$, such that on the corresponding family of domains $X(\!\sqrt\epsilon)$, $\sqrt\epsilon\in S$, of Definition~\ref{definition:BL-X}, 
\textbf{there is a unique bounded analytic solution $y(x,\sqrt\epsilon)$ to \eqref{eq:BL-cm}.} 
It is uniformly continuous on 
$$X=\{(x,\sqrt\epsilon)\mid x\in X(\!\sqrt\epsilon)\}$$
and analytic on the interior of $X$, and it vanishes (is uniformly $O(x^2\!-\epsilon)$) at the singular points. 
When $\sqrt\epsilon$ tends radially to $0$ with $\arg\sqrt\epsilon=\beta$, then 
$y(x,\sqrt\epsilon)$ converges to $y(x,0)$ uniformly on compact sets of the sub-domains 
$X_\beta^\pm(0)\subseteq X^\pm(0)$, and the restriction of $y(x,0)$ to $X^\pm(0)$ is the Borel sum of the formal series $\hat y(x,0)$ \eqref{eq:BL-formalsolution} in directions of $\Omega^\pm$.

The solution $y(x,\sqrt\epsilon)$, and its domain $X$, associated to each pair $\{\Omega^+,\Omega^-\}$ are unique up to the reflection $(x,\sqrt\epsilon)\to(x,-\sqrt\epsilon)$, or analytic extensions.

\smallskip
\noindent
\textbf{(ii)} 
If, moreover, the spectrum of the matrix $M$ is of \emph{Poincaré type}  
(the convex hull of $\Spec M$ does not contain 0 inside or on the boundary), 
i.e. if there exists a (unique) component $\Omega_1$ of $\C\sminus\bigcup_{\lambda\in\Spec M}\lambda\R^+$ of opening $>\pi$, 
then the solution $y_1(x,\sqrt\epsilon)$ on the domain $X_1(\!\sqrt\epsilon),\ {\sqrt\epsilon\in S_1}$, associated to the pair $\{\Omega_1,\Omega_1\}$
is ramified only at one of the singular points, and analytic at the other (see Figure~\ref{figure:BL-4b}).

Such is the case in dimension $m=1$.
\begin{figure}[t]
\centering
\includegraphics[width=0.83\textwidth]{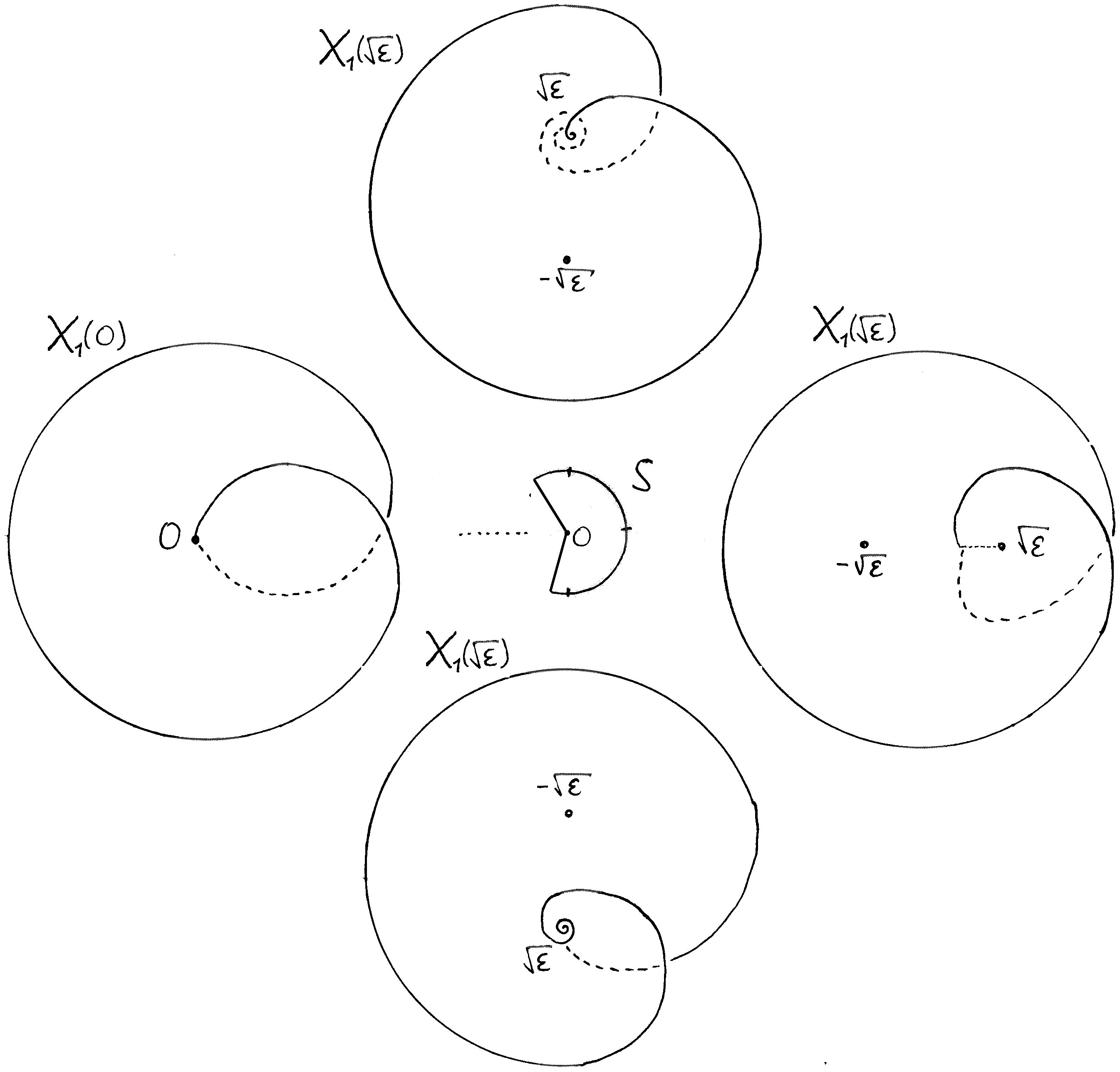}
\caption{Example of the spiraling domains $X_1(\!\sqrt\epsilon)$  of Theorem~\ref{theorem:BL-2}\,(ii) according to $\sqrt\epsilon\in S_1$.}
\label{figure:BL-4b}
\end{figure}

\end{theorem}

The solutions $y(x,\sqrt\epsilon)$ will be constructed in Section~\ref{sec:BL-4} in form of two-sided Laplace integrals
\begin{equation*}
y(x(t,\epsilon),\sqrt\epsilon)=\int_{-\infty e^{i\alpha}}^{+\infty e^{i\alpha}}\upsilon^\pm(\xi,\sqrt\epsilon)\, e^{-t\xi}\,d\xi,\qquad t\in\mathbf{T}_\alpha^\pm(\Lambda,\sqrt\epsilon),
\end{equation*}
with $\upsilon^\pm$ a solution to a non-linear convolution equation (corresponding to \eqref{eq:BL-cm} via the unfolded Borel/Laplace transformations \eqref{eq:BL-unfoldedBL}) on strips in the Borel plane, which will be obtained using a fixed-point argument.



\begin{proposition}\label{proposition:BL-1}
Let $y(x,\sqrt\epsilon)$ be the solution of Theorem I, 
let $\sqrt{\bar\epsilon},\, \sqrt{\tilde\epsilon} =e^{\pi i}\sqrt{\bar\epsilon}\linebreak[0]\in S$ be two opposite roots of $\epsilon$ as in Figure~\ref{figure:BL-6},
and let $\bar y^\pm(t,\sqrt\epsilon)$ (resp. $\tilde y^\pm(t,\sqrt\epsilon)$)
be the lifting of $y(x,\sqrt{\bar\epsilon})$  on $T^\pm(\!\sqrt{\bar\epsilon})$ (resp. $y(x,\sqrt{\tilde\epsilon})$ on $T^\pm(\!\sqrt{\tilde\epsilon})$).
There are constants $0<R<\min_{\lambda\in\Spec M}|\lambda|$ and $C_1,C_2>0$  independent of $\epsilon$, such that
for every $t\in T^\pm(\!\sqrt{\tilde\epsilon})\cap T^\mp(\!\sqrt{\bar\epsilon})$, see Figure~\ref{figure:BL-6a},
\begin{equation}\label{eq:BL-difference}
\|\tilde y^\pm(t,\sqrt\epsilon)-\bar y^\mp(t,\sqrt\epsilon)\| \leq \left(C_1+\tfrac{C_2}{d(t,\epsilon)}\right) e^{-R d(t,\epsilon)},
\end{equation}
where $d(t,\epsilon)$ denotes the distance of $t$ from the border of $T^\pm(\!\sqrt{\tilde\epsilon})\cap T^\mp(\!\sqrt{\bar\epsilon})$. 
\end{proposition}

\begin{figure}[t]
\centering
\includegraphics[width=0.8\textwidth]{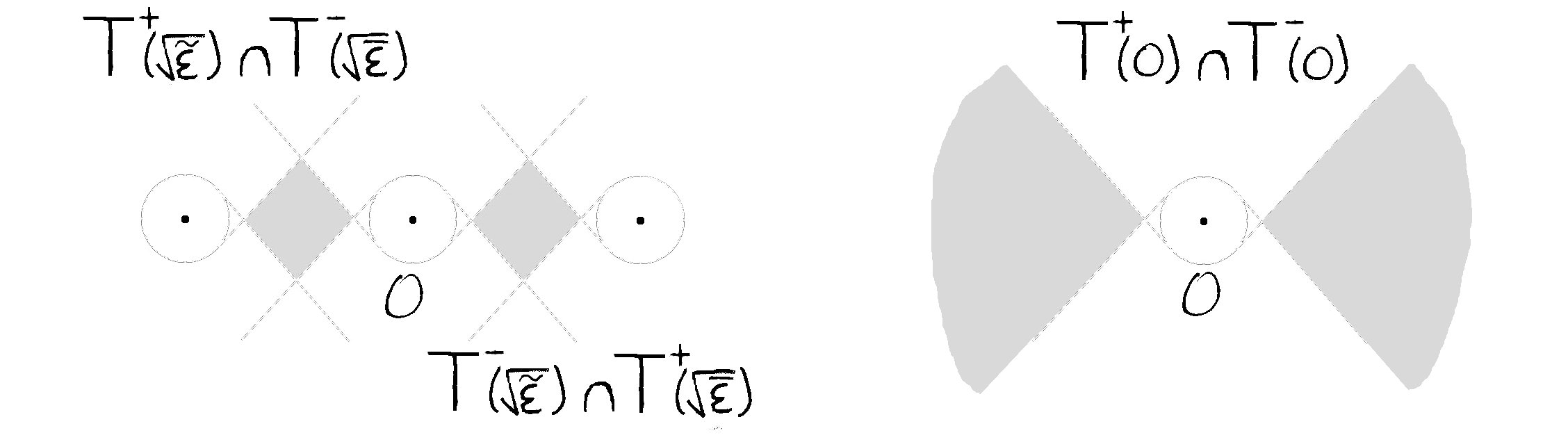}
\caption{The intersections of the domains $T^\pm(\!\sqrt{\tilde\epsilon})\cap T^\mp(\!\sqrt{\bar\epsilon})$
of Figure~\ref{figure:BL-6} for $\sqrt{\bar\epsilon},\, \sqrt{\tilde\epsilon} =e^{\pi i}\sqrt{\bar\epsilon}$ (left)
and their limits as $\epsilon\to 0$ (right).
}
\label{figure:BL-6a}
\end{figure}

\subsubsection{Hadamard--Perron interpretation for $\epsilon\neq 0$ and convergence of local analytic solutions.}

The linearization of the vector field \eqref{eq:BL-vectsyst} at $x=\pm\sqrt\epsilon$ is equal to
\begin{equation}\label{eq:BL-linearizedsystem} 
~\qquad
\dot x=\pm 2\sqrt\epsilon\,(x\mp\sqrt\epsilon),\qquad \dot y=M_{\pm\sqrt\epsilon}\, y, \qquad
M_{\pm\sqrt\epsilon}=M+O(\sqrt\epsilon).
\end{equation}

\textbf{(i)} 
Let a line $e^{i\alpha}\R$ separate the point $2\sqrt\epsilon$ and $k$ of the eigenvalues of $M$ 
from the point $-2\sqrt\epsilon$ and the other $m-k$ eigenvalues ($0\leq k\leq m$), see Figure~\ref{figure:BL-4c}.
Then for $\epsilon$ small enough, the respective eigenvalues of $M_{\pm\sqrt\epsilon}$ lie on the same sides of the line $e^{i\alpha}\R$, hence 
by the Hadamard--Perron theorem \cite{IY} the vector field \eqref{eq:BL-vectsyst} has a unique $(k+1)$-dimensional local invariant manifold at $(\sqrt\epsilon,0)$ 
tangent to the $x$-axis and the corresponding $k$ eigenvectors of $M_{\sqrt\epsilon}$, 
and a unique $(m-k+1)$-dimensional local invariant manifold at $(-\sqrt\epsilon,0)$ 
tangent to the $x$-axis and the corresponding $m-k$ eigenvectors of $M_{-\sqrt\epsilon}$. 
They intersect transversally as the graph of the solution $y(x,\sqrt\epsilon)$ of Theorem~\ref{theorem:BL-2}.
Since the root parameter $\sqrt\epsilon$ can vary as long as $\pm2\sqrt\epsilon$ stay in their respective half-planes bounded by $e^{i\alpha}\R$, whose angle $\alpha$ can also vary a bit,
this gives a sector $S$ of opening $>\pi$. 
We see that one cannot continue this description in $\sqrt\epsilon$ beyond such maximal sector $S$.

\begin{figure}[t]
\centering
\includegraphics[width=0.66\textwidth]{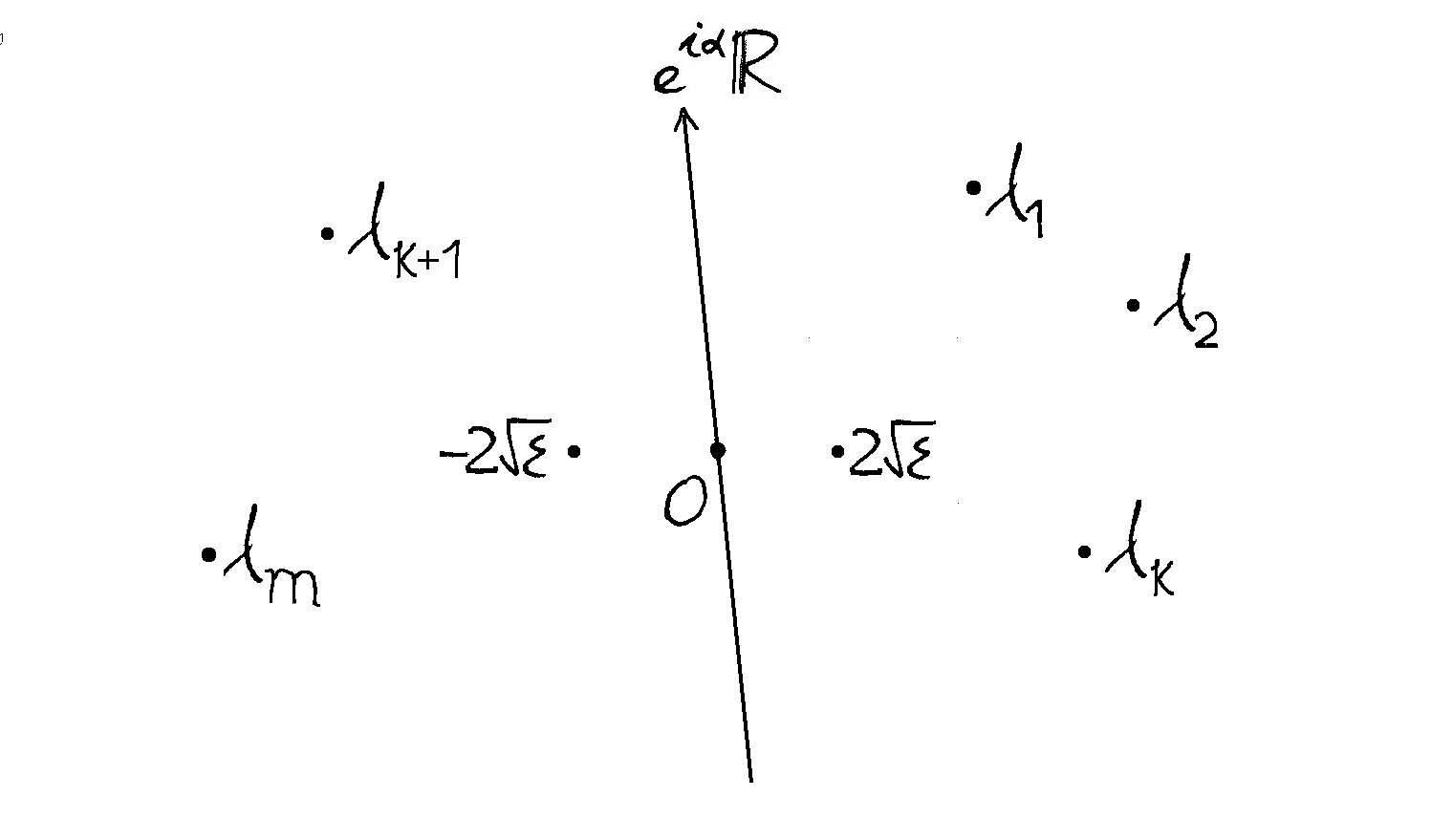}
\caption{The spectrum of $M$ in the Borel plane; the line $e^{i\alpha}\R$ is the dividing line of the Hadamard--Perron theorem
and also the integration path of the Laplace transform $\Cal L_\alpha$ \eqref{eq:BL-ypm}.
}
\label{figure:BL-4c}
\end{figure}

\smallskip
\textbf{(ii)}
If all the eigenvalues of $M$ are in a same open sector of opening $<\pi$ (i.e. $\Spec M$ is of Poincaré type),
and $-2\sqrt\epsilon$ lies in the interior of the complementary sector $\Omega_1$ of opening $>\pi$,
then one obtains the solution $y_1(x,\sqrt\epsilon)$ from Theorem~\ref{theorem:BL-2} as a continuation of the local analytic solution at $x=-\sqrt\epsilon$
(i.e. of the local invariant manifold of \eqref{eq:BL-vectsyst} tangent to the $x$-axis, provided by the Hadamard--Perron theorem) 
to the domain $X_1(\!\sqrt\epsilon)$. 

\smallskip
While this Hadamard--Perron approach explains where do the solutions of Theorem~\ref{theorem:BL-2} come from, it does not provide their natural domain on which they are bounded.
One should however notice the similarities between the description  provided by the Hadamard--Perron theorem for $\epsilon\neq 0$ (Figure~\ref{figure:BL-4c}) and that 
of the Borel summation for $\epsilon=0$ (Figure~\ref{figure:BL-4d1}). 
In Section~\ref{sec:BL-4} we will unify the two of them using the unfolded Borel--Laplace transformations.

\begin{remark}[Local invariant manifolds for non-resonant $\epsilon\neq 0$ and their convergence]\label{remark:BL-localsolutions}
If the simple singular point of \eqref{eq:BL-vectsyst} at $x=\sqrt\epsilon\neq 0$ satisfies the following 
non-resonance condition 
$$2\sqrt\epsilon\,\N\cap\Spec M_{\sqrt\epsilon}=\emptyset,$$
then it is known that the equation \eqref{eq:BL-cm} possesses a unique convergent formal solution near $x=\sqrt\epsilon$, i.e.
the vector field \eqref{eq:BL-vectsyst} has a 1-dimensional local  analytic invariant manifold tangent to the $x$-axis at the singularity.
The resonant values $\sqrt\epsilon= \frac{\lambda}{2n}$, $\lambda\in\Spec M_{\sqrt\epsilon}$, $n\in\N_{>0}$, accumulate at the origin along the rays $\lambda\R^+$, $\lambda\in\Spec M$, dividing the $\sqrt\epsilon$-plane in a finite number of sectors (Figure~\ref{figure:BL-4d}). 
The following theorem was proven by A. Glutsyuk \cite{G}.

\begin{theorem}[Glutsyuk]\label{theorem:BL-localsolutions}
If $\sqrt\epsilon\neq 0$ lies inside one of these sectors (i.e. $\sqrt\epsilon\,\mathbb{R}^+\cap\text{Spec}\, M=\emptyset$), 
then the local analytic solution at $x=\sqrt\epsilon$ converges, when $\sqrt\epsilon$ tends radially to 0, 
to the sectoral Borel sum $L_\alpha[\widehat B[\hat y_0]](x)$ of the formal solution of the limit system (cf. Figure~\ref{figure:BL-4d1}), 
where $\alpha=\arg\sqrt\epsilon$ (this is the direction on which lies the corresponding eigenvalue $2\sqrt\epsilon$ of the linearization \eqref{eq:BL-linearizedsystem}).
\end{theorem}

Unless the spectrum of $M$ is of Poincaré type, these sectors in the $\sqrt\epsilon$-plane  
on which the convergence happens are of opening $<\pi$. 

\begin{figure}[t]
\centering
\includegraphics[width=0.56\textwidth]{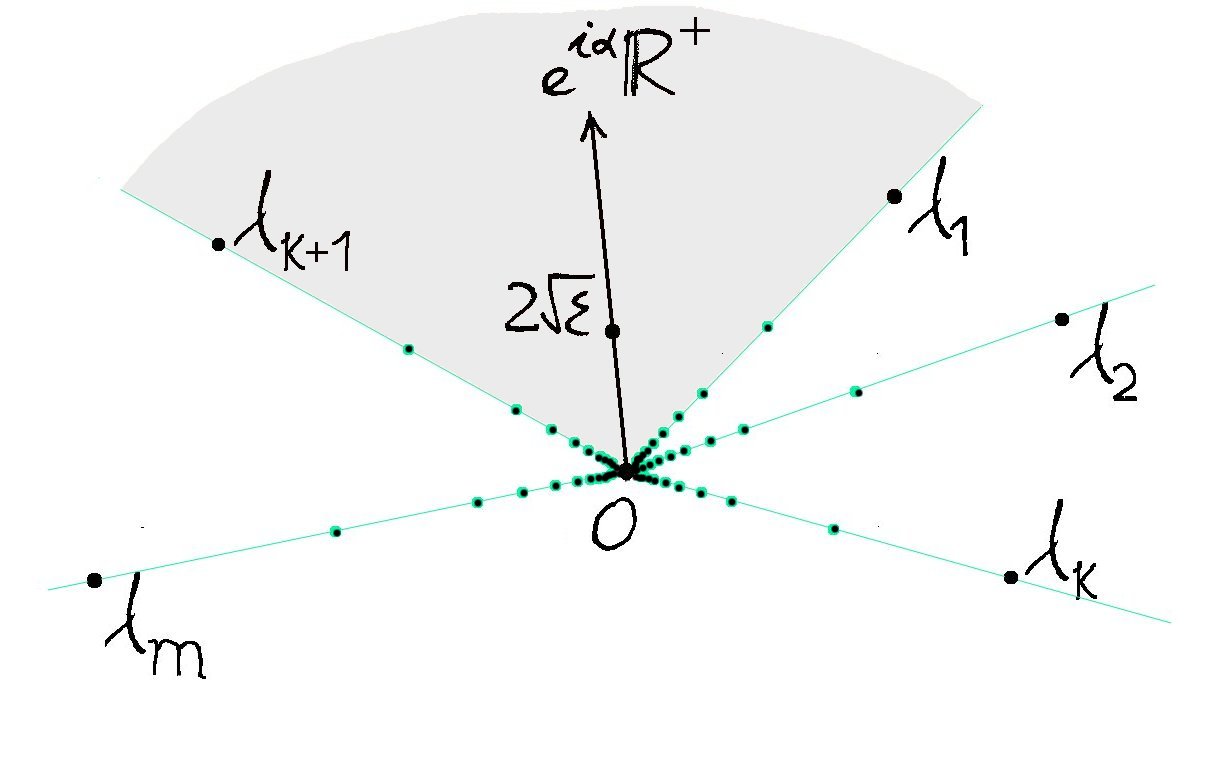}
\caption{The resonant values of $\sqrt\epsilon= \frac{\lambda}{2n}$, $\lambda\in\Spec M_{\sqrt\epsilon}$, $n\in\N_{>0}$,  accumulate along the rays $\lambda\R^+$, dividing the $\sqrt\epsilon$-plane in sectors
on which the local analytic solutions near $x=\sqrt\epsilon\neq0$ converge as $\sqrt\epsilon\to 0$ to sectoral solutions.
}
\label{figure:BL-4d}
\end{figure}

\end{remark}

\subsubsection{Asymptotic expansions}

\paragraph{Inner asymptotic expansion.}
Blowing-up the $x$-coordinate let $z=\frac{x}{\sqrt\epsilon}$ and
$$ Y(z,\sqrt\epsilon):= y(\sqrt\epsilon z,\sqrt\epsilon),\quad \sqrt\epsilon z\in X(\!\sqrt\epsilon),$$
be the solution of Theorem~\ref{theorem:BL-2}, and
\begin{equation}\label{eq:BL-innerexp}
\hat y(\sqrt\epsilon z,\sqrt\epsilon)=:\hat Y(z,\sqrt\epsilon)=\sum_{p=2}^{+\infty} Y_p(z)\sqrt\epsilon^{\,p}, \quad
Y_p(z)=\sum_{\substack{0\leq k\leq p \\ k+2j=p}} y_{kj}z^k
\end{equation}
the formal solution \eqref{eq:BL-formalsolution}.
Since,
$$\sqrt{\epsilon} \,t(x,\epsilon)=t(z,1),$$
it follows from Proposition~\ref{proposition:BL-1}, that for $\sqrt{\tilde\epsilon}$ in the intersection $S\cap e^{\pi i}S$ and a fixed $z$, the difference 
$|Y(z,\sqrt{\tilde\epsilon})-Y(z,-\sqrt{\tilde\epsilon})|$ 
is exponentially flat in $\sqrt{|\tilde\epsilon|}$, therefore by the Ramis--Sibuya theorem (\cite{Si}, \cite{Bal}) 
the bounded function $Y(z,\sqrt\epsilon)$ possesses an asymptotic expansion of Gevrey order 1 on $S$, equal to $\hat Y(z,\sqrt\epsilon)$ by its uniqueness, and therefore it is also its Borel sum  on $S$
(the opening of $S$ is $>\pi$). 
We need yet to specify the domain of $z$ on which this is true.
First, remark that 
$$t(x,\tilde\epsilon)\in T^\pm(\sqrt{\bar\epsilon})\cap T^\mp(\sqrt{\tilde\epsilon})=
\mathbf{T}_{\beta_2-\eta}^\pm(\Lambda,\sqrt{\tilde\epsilon})\cap \mathbf{T}_{\beta_1+\eta}^\pm(\Lambda,\sqrt{\tilde\epsilon}),$$
where $\beta_1,\beta_2,\eta$ are as in Definition~\ref{definition:BL-X}, and $\sqrt{\tilde\epsilon}\in S\cap e^{\pi i}S$,
$\sqrt{\bar\epsilon}=e^{-\pi i}\sqrt{\tilde\epsilon}$. 
Hence $t(z,1)$ belongs to a limit of such rhomboidal domains scaled by $\sqrt{\tilde\epsilon}$, as $\sqrt{\tilde\epsilon}\to 0$ radially
with $\arg\sqrt{\tilde\epsilon}=\beta+\frac{\pi}{2}$
($\beta$ being the direction of the Borel summation): 
\begin{equation}\label{eq:BL-tz}
t(z,1)\in\mathbf{T}_{\beta_2-\eta-\beta-\frac{\pi}{2}}^\pm(0,1)\cap \mathbf{T}_{\beta_1+\eta-\beta-\frac{\pi}{2}}^\pm(0,1).
\end{equation}
In the $z$-coordinate, $z=\coth t(z,1)$, this corresponds to the limit of central region of the intersection $\frac{1}{\sqrt{\tilde\epsilon}}X(\!\sqrt{\tilde\epsilon})\cap \frac{1}{\sqrt{\bar\epsilon}}X(\!\sqrt{\bar\epsilon})$,
which covers a neighborhood of the origin and extends towards $\infty$ as a double sector $\arg z\in\mp\frac{\pi}{2}-\beta\,+\,]\beta_1+\eta,\beta_2-\eta[$.

\begin{proposition}[Inner asymptotic expansion]\label{theorem:BL-IAE}
The blow-up $\,Y(z,\sqrt\epsilon)=y(\sqrt\epsilon z,\sqrt\epsilon)$ of the sectoral solution of Theorem~\ref{theorem:BL-2} is equal to the Borel sum in $\sqrt\epsilon$ of the blown-up formal solution $\hat Y(z,\sqrt\epsilon)=\sum_{p=2}^{+\infty} Y_p(z)\sqrt\epsilon^p$, for $\sqrt\epsilon\in S\sminus\{0\}$ and $z\in\frac{1}{\sqrt\epsilon}X(\!\sqrt\epsilon)$ satisfying \eqref{eq:BL-tz} for some direction $\beta$ covered by $e^{-\frac{\pi i}{2}}S\cap e^{\frac{\pi i}{2}}S$ with $|\beta-\arg\sqrt\epsilon|<\frac{\pi}{2}$.
\end{proposition}

\begin{remark}
The blow-up transforms \eqref{eq:BL-cm} to a singularly perturbed equation.
\end{remark}

Returning back to the $x$-coordinate and using slightly modified version of the Borel-Laplace summation operators, following \cite{Bal2}, we obtain:

\begin{theorem}[Borel sum of $\hat y(x,\sqrt\epsilon)$]
\label{theorem:BL-3}
Let $U(x,\epsilon)$ be analytic extension of the function given by the convergent series in $(x,\epsilon)$ 
$$U(x,\epsilon)=\sum_{j,k} \frac{y_{kj}}{(k+2j)!}x^k\epsilon^j.$$ 
For each point $(x,\sqrt\epsilon)$, for which there is an angle $\theta \in\,]\!-\!\frac{\pi}{2},\frac{\pi}{2}[$ such that
the set $\mbox{$\mathbf{S}_\theta\cdot(x,\sqrt\epsilon)$}\subseteq X$, with $\mathbf{S}_\theta\subset\C$ denoting the circle  through the points $0$ and $1$ with  center on $e^{i\theta}\R^+$,
we can express $y(x,\sqrt\epsilon)$ as the 
following Laplace transform of $U$:
\begin{equation}\label{eq:BL-Ly}
y(x,\sqrt\epsilon)=\int_0^{+\infty e^{i\theta}}\!\!\!U(sx,s^2\epsilon)\,e^{-s}\,ds.
\end{equation}
\end{theorem}


\begin{proof}
For $\sqrt\epsilon\neq 0$, 
the formal Borel transform of $\sqrt\epsilon\,\hat Y(z,\sqrt\epsilon)$ w.r.t. $\sqrt\epsilon$ is equal to 
$U(\sqrt\epsilon z,\epsilon)=\widehat B[\sqrt\epsilon\,\hat Y(z,\sqrt\epsilon)](\sqrt\epsilon)$, hence by Proposition~\ref{theorem:BL-IAE}, 
$$y(x,\epsilon)=Y(z,\sqrt\epsilon)=\frac{1}{\sqrt\epsilon}\int_{0}^{+\infty e^{i\alpha}}\!\!\!\! U(\nu z,\nu^2)\,e^{-\frac{\nu}{\sqrt\epsilon}}\,d\nu=\int_0^{+\infty e^{i\theta}}\!\!\!\!U(sx,s^2\epsilon)\,e^{-s}\,ds$$
after substitution $\nu=\sqrt\epsilon\,s$.
For $\sqrt\epsilon=0$, 
$U(\xi,0)=\widehat B[x\,\hat y(x,0)](\xi)$ w.r.t. $x$, thus one obtains \eqref{eq:BL-Ly} by substituting $\xi=sx$ in the classical Laplace transform \eqref{eq:BL-laplace} of $U(\xi,0)$. Moreover and the condition on the points $(x,0)\in X$ is satisfied for all $x\in X^\pm(0)$.
\qed\end{proof}

\begin{remark}
The Borel-Laplace summation of Theorem~\ref{theorem:BL-3} preserves algebraic operations as well as differentiation with respect to $x,\epsilon$
\cite{Bal2}.
This means that the sectoral solution $y(x,\sqrt\epsilon)$ and the formal solution $\hat y(x,\epsilon)$ satisfy the same polynomial $\{\frac{\partial}{\partial x},\frac{\partial}{\partial\epsilon}\}$-differential relations over the ring $\C\{x,\epsilon\}$ of convergent series.
\end{remark}

\paragraph{Outer asymptotic expansion.}
Let $\hat y(x,\epsilon)=\sum_{k,j=0}^{+\infty}y_{kj}x^k\epsilon^j$ be the formal solution \eqref{eq:BL-formalsolution} and let $y_j^\pm(x)$, $j\in\N$, be Borel sums of $\hat y_j(x)=\sum_{k=0}^{+\infty} y_{kj}x^k$ 
provided by Theorem~\ref{theorem:BL-BERS} on the domains $X^\pm(0)$.
One can then consider the formal series in $\epsilon$
\begin{equation}\label{eq:BL-outerexp}
\sum_{j=0}^{+\infty}y_j^\pm(x)\,\epsilon^j.
\end{equation}
It has been shown in \cite{Pa} (in the case of normalizing transformations for non-resonant irregular linear systems, 
cf. Section~\ref{subsec:BL-1.3} below) that 
the sectoral solution $y(x,\sqrt\epsilon)$ of Theorem~\ref{theorem:BL-2} is asymptotic to \eqref{eq:BL-outerexp} of Gevrey order 1 in $\sqrt\epsilon$
on a sector $S'\subset S$ on which both singularities $\{\sqrt\epsilon,-\sqrt\epsilon\}$ are inside 
the same domain $X^\pm(0)$, i.e. $\sqrt\epsilon\in X^+(0)\cap X^-(0)$.
We call \eqref{eq:BL-outerexp} an \emph{outer asymptotic expansion}
as it is defined for $x$ in an ``outer'' region, $|x|>|\!\sqrt\epsilon|$. 

\begin{remark}
The inner \eqref{eq:BL-innerexp} and outer \eqref{eq:BL-outerexp} asymptotic expansions corresponding to $\hat y(x,\epsilon)$ can be seen as a special case of a \emph{composite asymptotic expansion} with a trivial fast part in the sense of \cite{FS}, with the small difference that the functions $y_j^\pm(x)\sim\sum_{k=0}^{+\infty}y_{kj}x^k$ are sectoral, rather than  analytic on a fixed neighborhood of 0.
\end{remark}

\subsection{Sectoral normalization of families of non-resonant linear differential systems}\label{subsec:BL-1.3}

An  application of Theorem~\ref{theorem:BL-2}, interesting on its own, is the problem of existence of normalizing transformations for linear differential systems 
near an unfolded non-resonant irregular singularity of Poincaré rank 1.  We will show that this problem can be reduced to a system \eqref{eq:BL-cm} of $m=n\,(n-1)$ Ricatti equations  
(where $n$ is the dimension of the system),  
providing thus a proof of a sectoral normalization theorem by Parise \cite{Pa}, Lambert and Rousseau \cite{LR}.

\smallskip

Consider a parametric family of linear systems
$\,\Delta(x,\epsilon)\,y=0$  given by
\begin{equation}\label{eq:BL-Delta}
\Delta(x,\epsilon)=(x^2\!-\epsilon)\frac{d }{dx}-A(x,\epsilon),\qquad (x,\epsilon)\in(\C\times\C,0)
\end{equation} 
where $\,y(x,\epsilon)\in\C^n$, $A(x,\epsilon)$ is analytic,
and assume that the eigenvalues $\lambda_i^{(0)}(0)$, $i=1,\ldots,n$, of the matrix $A(0,0)$ are distinct.
Let $\lambda_i(x,\epsilon)=\lambda_i^{(0)}(\epsilon)+x\lambda_i^{(1)}(\epsilon)$, $i=1,\ldots, n$, be the eigenvalues of $A(x,\epsilon)$ modulo $O(x^2\!-\epsilon)$, and define
\begin{equation}\label{eq:BL-model}
 \widehat\Delta(x,\epsilon)=(x^2\!-\epsilon)\frac{d }{dx}-\Lambda(x,\epsilon),\qquad \Lambda(x,\epsilon)=\text{Diag}(\lambda_1(x,\epsilon),\ldots,\lambda_n(x,\epsilon)),
\end{equation}
the \emph{formal normal form} for $\Delta$.
The problem we address, is to find a bounded invertible linear transformation 
$y=T(x,\sqrt\epsilon)\,u$ between the two systems $\Delta y=0$ and $\widehat\Delta u=0$.
Such $T$ is a solution of the equation
\begin{equation}\label{eq:BL-T}
(x^2\!-\epsilon)\frac{d T}{dx}=AT-T\Lambda.
\end{equation}

Note that if $V(x,\epsilon)$ is an analytic matrix of eigenvectors of $A(x,\epsilon)$ then the transformation $y=V(x,\epsilon)\,y_1$
brings the system $\Delta y=0$ to $\Delta_1 y_1=0$, whose matrix is written as $\,A_1(x,\epsilon)=\Lambda(x,\epsilon)+(x^2\!-\epsilon)R(x,\epsilon),\,$ 
with $\,R=-V^{-1}\tfrac{dV}{dx}$.
Hence we can suppose that system \eqref{eq:BL-Delta} is already in such form.
The following theorem is originally by Parise \cite{Pa}, and by 
Lambert and Rousseau \cite[Theorem 4.21]{LR}, generalizing earlier investigations 
by Zhang \cite{Zh} 
\footnote{Zhang also unfolds the Laplace integral \eqref{eq:BL-laplace}, unlike us he chooses to unfold the kernel $e^{-\frac{\xi}{x}}d\xi$ by $\left(\frac{x-\sqrt\epsilon\xi}{x+\sqrt\epsilon\xi}\right)^{\frac{1}{2\sqrt\epsilon}}d\xi=e^{-t(x,\xi^2\epsilon)\cdot\xi}d\xi$, in our notation.}
of confluence in the hypergeometric equation. 

\begin{theorem}[Parise, Lambert, Rousseau]\label{theorem:BL-linearnormalization}
Let $\Delta(x,\epsilon)$ be a non-resonant system \eqref{eq:BL-Delta} with $A(x,\epsilon)=\Lambda(x,\epsilon)+(x^2\!-\epsilon)R(x,\epsilon)$ for some analytic germ $R(x,\epsilon)$,
and let $\widehat\Delta(x,\epsilon)$ be its formal normal form \eqref{eq:BL-model}.
Then there exists  a family of  ramified ``spiraling'' domains  $X(\!\sqrt\epsilon)$, $\sqrt\epsilon\in S$, as in Theorem~\ref{theorem:BL-2}\,(i) (Figure~\ref{figure:BL-4a})
on which there exists a normalizing transformation $T(x,\sqrt\epsilon)$, solution to the equation \eqref{eq:BL-T}, 
which is uniformly continuous on 
$$X=\{(x,\sqrt\epsilon)\mid x\in X(\!\sqrt\epsilon)\}$$
and analytic on its interior, and such that $\,T(\pm\sqrt\epsilon,\sqrt\epsilon)=I+O(\!\sqrt\epsilon)\,$ is diagonal. 
This transformation $T$ on $X$ is unique modulo right multiplication by an invertible diagonal matrix constant in $x$.
\end{theorem}

\begin{proof}
 Write $T(x,\sqrt\epsilon)=\big(I+U(x,\sqrt\epsilon)\big)\cdot T_D(x,\sqrt\epsilon)$, where $T_D(x,\sqrt\epsilon)$ is the diagonal of $T$, and
the matrix $U(x,\sqrt\epsilon)=O(x^2\!-\epsilon)$  has zeros on the diagonal.
We search for $U(x,\sqrt\epsilon)$, such that
$\,y_D=\big(I+U(x,\sqrt\epsilon)\big)^{-1}y\,$ satisfies
$$(x^2\!-\epsilon)\frac{dy_D}{dx}-\Big(\Lambda(x,\epsilon)+(x^2\!-\epsilon)D(x,\sqrt\epsilon)\Big)y_D=0,$$
for some diagonal matrix $D(x,\sqrt\epsilon)$,
and  set
$$T_D(x,\sqrt\epsilon)=e^{\int_{\sqrt\epsilon}^x D(s,\sqrt\epsilon)\,ds}.$$
The matrix $U(x,\sqrt\epsilon)$ is solution to
$$(x^2\!-\epsilon)\frac{dU}{dx}=\Lambda U-U\Lambda+(x^2\!-\epsilon)\Big(R\,(I+U)-(I+U)D\Big),$$
where one must set $D$ to be equal to the diagonal of $R\,(I+U)$.
Therefore the off-diagonal terms of  $U=(u_{ij})_{i,j=1}^n$ are solution to the system of $n(n-1)$ equations
$$(x^2\!-\epsilon)\frac{du_{ij}}{dx}=(\lambda_i-\lambda_j)u_{ij}+(x^2\!-\epsilon)\Big(r_{ij}+\sum_{k\neq j}r_{ik}u_{kj}-u_{ij}r_{jj}-u_{ij}\sum_{k\neq j}r_{jk}u_{kj}\Big),$$
and one can apply Theorem~\ref{theorem:BL-2}.
\qed\end{proof}

\subsection{Remark on generalization to singularities of greater multiplicities.}	
Saddle-node singularities of codimension $k$ (multiplicity $k+1$) unfold generically as
\begin{equation}\label{eq:BL-gcm}
(x^{k+1}+\epsilon_{k-1}x^{k-1}+\ldots+\epsilon_0)\frac{dy}{dx}=My+f(x,\epsilon,y),\qquad
(x,\epsilon,y)\in\C\times\C^k\times\C^m.
\end{equation}
The case of dimension $m=1$ was studied in \cite{RT}; their construction of the center-manifold should probably generalize also to the case $m\geq 1$ with $M$ having spectrum of Poincaré type.
The non-Poincaré situation is hinted in \cite{HLR} where a generalization of Theorem~\ref{theorem:BL-linearnormalization} on sectoral normalization of unfolded irregular singularities of linear systems is given.
As in Remark~\ref{remark:BL-X}, the domains constructed in \cite{HLR} are linked to the real phase space of the complex vector fields 
$e^{i(\frac{\pi}{2}-\alpha)}(x^{k+1}+\epsilon_{k-1}x^{k-1}+\ldots+\epsilon_0)\frac{\partial}{\partial x}$
(cf. \cite{BD}).

Theorem~\ref{theorem:BL-3} on summability of the unique formal power series solution in $(x,\epsilon)$ seems for some reason to be rather particular to the codimension $k=1$ (multiplicity 2).
Already in the case of \eqref{eq:BL-gcm} with the derivation on the left side $(x^{k+1}\!-\epsilon_0)\frac{\partial}{\partial x}$, the sectors in $\epsilon_0$-space for the domains constructed in \cite{HLR} are only of opening $>\frac{2\pi}{k}$, while one would rather want them $>\frac{(k+1)\pi}{k}$ 
in order to correspond with the expected Gevrey order $(\frac{1}{k},\frac{k+1}{k})$ of the formal solution in $(x,\epsilon_0)$.

Let us note however, that the special case of \eqref{eq:BL-gcm} with the derivation $x(x^{k}\!-\epsilon_1)\frac{\partial}{\partial x}$,
which is invariant with respect to the rotation $x\mapsto e^{\frac{2\pi i}{k}}x$, can be reduced to
$\tilde x(\tilde x-\epsilon_1)\frac{\partial}{\partial\tilde x}$ by the ramification $\tilde x=x^{k}$ and a rank reduction
(cf. \cite{CMS}), therefore it does not encounter the above mentioned issue.


\section{Preliminaries on Fourier--Laplace transformations}\label{sec:BL-2}

We will recall some basic elements of the classical theory of Fourier--Laplace transformations on a line in the complex plane. 
The book \cite{D} can serve as a good reference.

For $\alpha\in\R$ and a locally integrable function $\phi:e^{i\alpha}\R\to\C$, one defines its \emph{two-sided Laplace transform} 
\begin{equation}\label{eq:BL-Laplacealpha}
\Cal L_\alpha[\phi](t)=\int_{-\infty e^{i\alpha}}^{+\infty e^{i\alpha}}\phi(\xi)\,e^{-t\xi}\,d\xi
\end{equation}
whenever it exists.
Later on, in Section~\ref{sec:BL-3}, we will replace the variable $t$ by the time variable $t(x,\epsilon)$ \eqref{eq:BL-t} of the vector field \eqref{eq:BL-vectfield}.

\begin{definition}
 For $A<B\in\R$, let us introduce the two following exponential weighted norms on locally integrable functions  $\phi:e^{i\alpha}\R\to\C$:
\begin{align*}
 |\phi|_{e^{i\alpha}\R}^{A,B} &=\sup_{s\in\R} |\phi(e^{i\alpha}s)|\cdot\left(|e^{-As}|+|e^{-Bs}|\right), \\
 \|\phi\|_{e^{i\alpha}\R}^{A,B} &=\int_{s\in\R} |\phi(e^{i\alpha}s)|\cdot \left(|e^{-As}|+|e^{-Bs}|\right)\,ds.
\end{align*}
\end{definition}

\begin{proposition}
 If $\|\phi\|_{e^{i\alpha}\R}^{A,B}<+\infty$, then the Laplace transform $\Cal L_\alpha[\phi](t)$ converges absolutely and is analytic for $t$ in the closed strip
$$\widebar T_\alpha^{A,B}:=\{t\in\C\mid A \leq \Re(e^{i\alpha}t) \leq B\}.$$
Moreover, $\Cal L_\alpha[\phi](t)$ tends uniformly to 0 as $t\to\infty$ in $\widebar T_\alpha^{A,B}$.
\end{proposition}

\begin{proof}
The integral $\int_{-\infty e^{i\alpha}}^{0}\phi(\xi)\,e^{-t\xi}\,d\xi$ converges absolutely in the closed half-plane $\Re(e^{i\alpha}t)\leq B$,
while the integral $\int_{0}^{+\infty e^{i\alpha}}\phi(\xi)\,e^{-t\xi}\,d\xi$ converges absolutely in the closed half-plane $\Re(e^{i\alpha}t)\geq A$.
For the second statement see \cite{D}, Theorem 23.6.
\qed\end{proof}

\begin{lemma}\label{lemma:BL-1}
If $0< D < \tfrac{B-A}{2}$, then for any function  $\phi:e^{i\alpha}\R\to\C$,
$$\|\phi\|_{e^{i\alpha}\R}^{A+D,B-D}\leq\tfrac{4}{D}\,|\phi|_{e^{i\alpha}\R}^{A,B}.$$
\end{lemma}

\begin{proof}
\begin{equation*}
\begin{split}
\int_{-\infty}^{0} |\phi(e^{i\alpha}s)|&\left(|e^{-(A+D)s}|+|e^{-(B-D)s}|\right)\,ds \\
&\leq \int_{-\infty}^{0} |e^{D s}|\, d s \cdot \sup_{s\in\R} |\phi(e^{i\alpha}s)|\left(|e^{-As-2Ds}|+|e^{-Bs}|\right) \\
&\leq \tfrac{1}{D}\cdot 2|\phi|_{e^{i\alpha}\R}^{A,B},
\end{split}
\end{equation*}
since $|e^{-As-2Ds}|\leq |e^{-Bs}|\leq |e^{-As}|+|e^{-Bs}|,\,$ for $s<0$.
The same kind of estimate is obtained also for $\int_0^{+\infty}$. 
\qed\end{proof}

\begin{corollary}\label{corollary:BL-1}
If $|\phi|_{e^{i\alpha}\R}^{A,B}<+\infty$, then the Laplace transform $\Cal L_\alpha[\phi](t)$ converges absolutely and is analytic for $t$ in the open strip
$$T_\alpha^{A,B}:=\{t\in\C\mid A < \Re(e^{i\alpha}t) < B \}.$$
Moreover, $\Cal L_\alpha[\phi](t)$ tends to 0 as $t\to\infty$ uniformly in each $\widebar T_\alpha^{A_1,B_1}\subseteq T_\alpha^{A,B}$.
\end{corollary}

\begin{definition}\label{definition:BL-2}
The \emph{Borel transformation} is defined for any function $f$ analytic on some open strip $T_\alpha^{A,B}$, that vanishes at infinity uniformly in each closed
substrip $\widebar T_\alpha^{A_1,B_1}\!\subseteq T_\alpha^{A,B}$, by
\begin{equation}\label{eq:BL-B}
 \Cal B_\alpha [f](\xi)= \tfrac{1}{2\pi i} \, V.P.\!\int_{C-\infty ie^{-i\alpha}}^{C+\infty ie^{-i\alpha}} \!f(t)\,e^{t\xi}\,dt, \qquad\text{for }\ \xi\in e^{i\alpha}\R,
\end{equation}
where $\,V.P.\!\int_{C-\infty ie^{-i\alpha}}^{C+\infty ie^{-i\alpha}}$ stands for the ``Cauchy principal value'' 
$\!\!{\displaystyle\lim_{N\to+\infty}}\!\int_{C-Nie^{-i\alpha}}^{C+Nie^{-i\alpha}}\!$
and $C\in T_\alpha^{A,B}$.
\end{definition}

The two-sided Laplace transformation \eqref{eq:BL-Laplacealpha} and the Borel transformation \eqref{eq:BL-B} of analytic functions are inverse
one to the other when defined.
We will only need the following particular statement.

\begin{theorem} \label{theorem:BL-1}~

\smallskip\noindent
1)  Let $\,f\in\Cal O(T_\alpha^{A,B})\,$ be absolutely integrable on each line $C\!+\! ie^{-i\alpha}\R\subseteq T_\alpha^{A,B}$ and
vanishing at infinity uniformly in each closed sub-strip of $T_\alpha^{A,B}$.
Then the Borel transform
 $\phi(\xi)=\Cal B_\alpha[f](\xi)$ is absolutely convergent and continuous for all $\xi\in e^{i\alpha}\R$,
$$|\phi|_{e^{i\alpha}\R}^{A_1,B_1} \leq \tfrac{1}{2\pi}\!\! \displaystyle\sup_{A_1\leq C\leq B_1} \displaystyle\int_{\tau\in C+\R}\!\! |f(ie^{-i\alpha}\tau)|\,d\tau \quad\text{for}\quad A<A_1<B_1<B,
$$ 
and  $\,f(t)=\Cal L_\alpha[\phi](t)\,$ for all $t\in T_\alpha^{A,B}$.

\smallskip\noindent
2) Let $f$ be as in 1) with  $B=B_1=+\infty$, the strips being replaced by half-planes.
Then the Borel transform $\phi(\xi)=\Cal B_\alpha[f](\xi)$ is absolutely convergent and continuous on $e^{i\alpha}\R$, and
$\phi(\xi)=0$ for $\xi\in\,]\!-\!\infty e^{i\alpha},0[$,
$$|\phi|_{e^{i\alpha}\R}^{A_1,+\infty}:=\sup_{s\in\R^+}|\phi(e^{i\alpha}s)\,e^{-A_1 s}| 
\,\leq\, \tfrac{1}{2\pi} \sup_{C\geq A_1} \int_{\tau\in C+\R} |f(ie^{-i\alpha}\tau)|\,d\tau,$$
and
$$f(t)=\Cal L_\alpha[\phi\,](t)=\int_{0}^{+\infty e^{i\alpha}}\!\!\!\! \phi(\xi)\,e^{-t\xi}\,d\xi$$
is the one-sided Laplace transform of $\phi$ in the direction $\alpha$.
\end{theorem}
\begin{proof}
See \cite{D}, Theorems 28.1 and 28.2.
\qed\end{proof}

\medskip
Under the assumptions of Theorem~\ref{theorem:BL-1}, the Borel transformation converts derivative to multiplication by $-\xi$:
$$\Cal B_\alpha[\tfrac{df}{dt}](\xi)=-\xi\cdot\Cal B_\alpha [f](\xi),$$
which can be seen by integration by parts.
It also converts the product to the convolution:
$$\Cal B_\alpha[f\!\cdot\!g]\,(\xi)=\left[\Cal B_\alpha[f] * \Cal B_\alpha[g]\right]_\alpha(\xi),$$
defined by 
\begin{equation}
 \big[\phi*\psi\big]_\alpha(\xi)= \big[\psi*\phi\big]_\alpha(\xi) :=\int_{-\infty e^{i\alpha}}^{+\infty e^{i\alpha}} \phi(\zeta)\,\psi(\xi-\zeta)\,d\zeta.
\end{equation}
Indeed, we have
$\,\Cal L_\alpha[\phi *\psi](t)=\Cal L_\alpha[\phi](t)\cdot\Cal L_\alpha[\psi](t)$ using Fubini theorem and Theorem~\ref{theorem:BL-1}, 
and the assertion is obtained by the inversion theorem of the Laplace transform: $\Cal B_\alpha\big[\Cal L_\alpha[\phi]\big](\xi)=\frac{1}{2}\lim_{\nu\to 0+}\phi(\xi+e^{i\alpha}\nu)+\phi(\xi-e^{i\alpha}\nu)\,$ (cf. \cite{D}, Theorem 24.3), using the continuity of $[\phi\!*\!\psi]_\alpha(\xi)$.

\begin{lemma}[Young's inequality]\label{lemma:BL-2}
\begin{align*}
|\phi*\psi|_{e^{i\alpha}\R}^{A,B} &\leq \,|\phi|_{e^{i\alpha}\R}^{A,B} \cdot \|\psi\|_{e^{i\alpha}\R}^{A,B} \quad(\text{ and } \leq \|\phi\|_{e^{i\alpha}\R}^{A,B}\cdot |\psi|_{e^{i\alpha}\R}^{A,B}\,),\\[8pt]
\|\phi*\psi\|_{e^{i\alpha}\R}^{A,B} &\leq  \|\phi\|_{e^{i\alpha}\R}^{A,B}\!\cdot \|\psi\|_{e^{i\alpha}\R}^{A,B}.
 \end{align*}
\end{lemma}

\begin{proof}
Observe that
\begin{equation}\label{eq:BL-300}
(|e^{-As}|+|e^{-Bs}|) \leq (|e^{-A\sigma}|+|e^{-B\sigma}|)\cdot (|e^{-A(s-\sigma)}|+|e^{-B(s-\sigma)}|),
\end{equation}
the rest follows easily.
\qed\end{proof}

\begin{remark}[Convolution of analytic functions on open strips]~
In the subsequent text, rather then dealing with functions on a single line $e^{i\alpha}\R$, 
one will work with  functions which are analytic on some open strips in the $\xi$-plane (also called the \emph{Borel plane}), 
or on more general regions obtained as connected unions of open strips  of varying directions $\alpha$.


If  $\Omega_j,\ j\!=\!1,2,$ are two open strips of the same direction $\alpha$, and $\phi_j\in\Cal O(\Omega_j)$ are two analytic functions of bounded $\|\cdot\|_{c_j+e^{i\alpha}\R}^{A,B}$-norms, 
then their convolution
$$[\phi_1*\phi_2](\xi)=\int_{c_1-\infty e^{i\alpha}}^{c_1+\infty e^{i\alpha}}\phi_1(\zeta)\,\phi_2(\xi-\zeta)\,d\zeta, \quad \xi\in c_1+c_2+e^{i\alpha}\R, \ c_j\in\Omega_j$$
is well defined and analytic on the strip $\Omega_1+\Omega_2$.

\end{remark}

\begin{definition}[Dirac distributions in the Borel plane]\label{sec:BL-distributions}
It will be convenient to introduce for each $a\in\C$ the Dirac mass distribution $\delta_a(\xi)$, acting  on the Borel plane as shift operators $\xi\mapsto \xi-a$:
If $\phi(\xi)$ is an analytic function  on some strip $\Omega$ in a direction $\alpha$ one defines
$$\big[\delta_a*\phi\big](\xi):=\phi(\xi-a),$$
its translation to the strip $\Omega-a$.
With this definition, the operator $\delta_0$ plays the role of the unity of convolution.
One can represent each  $\delta_a$ as a ``boundary value'' of the function
 $\frac{1}{2\pi i\,(\xi-a)}$ (cf. \cite{B}): 
Let
$$\delta_a^\downarrow(\xi):= \tfrac{1}{2\pi i\,(\xi-a)}\restriction_{\C\sminus[a,a+\infty\,ie^{i\alpha})},\qquad
\delta_a^\uparrow(\xi):= \tfrac{1}{2\pi i\,(\xi-a)}\restriction_{\C\sminus[a,a-\infty\,ie^{i\alpha})},$$
be its restrictions to the two cut regions (see Figure~\ref{figure:BL-0a}), 
one then writes 
$$\delta_a(\xi)=\delta_a^\downarrow(\xi)-\delta_a^\uparrow(\xi),$$
and defines the convolution and the Laplace transform involving $\delta_a$ by integrating each term 
$\delta_a^\downarrow$ (resp. $\delta_a^\uparrow$)
along  deformed paths $\gamma_\alpha^\downarrow$ (resp. $\gamma_\alpha^\uparrow$) of direction $\alpha$ in their respective domains as in Figure~\ref{figure:BL-0a},
\begin{align*}
\big[\delta_a*\phi\big](\xi)&=V.P.\int_{\gamma_\alpha^\downarrow-\gamma_\alpha^\uparrow}\tfrac{1}{2\pi i\,(\zeta-a)}\,\phi(\xi-\zeta)\,d\zeta=\phi(\xi-a),\\
\Cal L_\alpha[\delta_a](t)&=V.P.\int_{\gamma_\alpha^\downarrow-\gamma_\alpha^\uparrow}\tfrac{1}{2\pi i\,(\xi-a)}\,e^{-t\xi}\,d\xi=e^{-at}.
\end{align*}
\end{definition}
\begin{figure}[ht]
\centering
\includegraphics[width=0.8\textwidth]{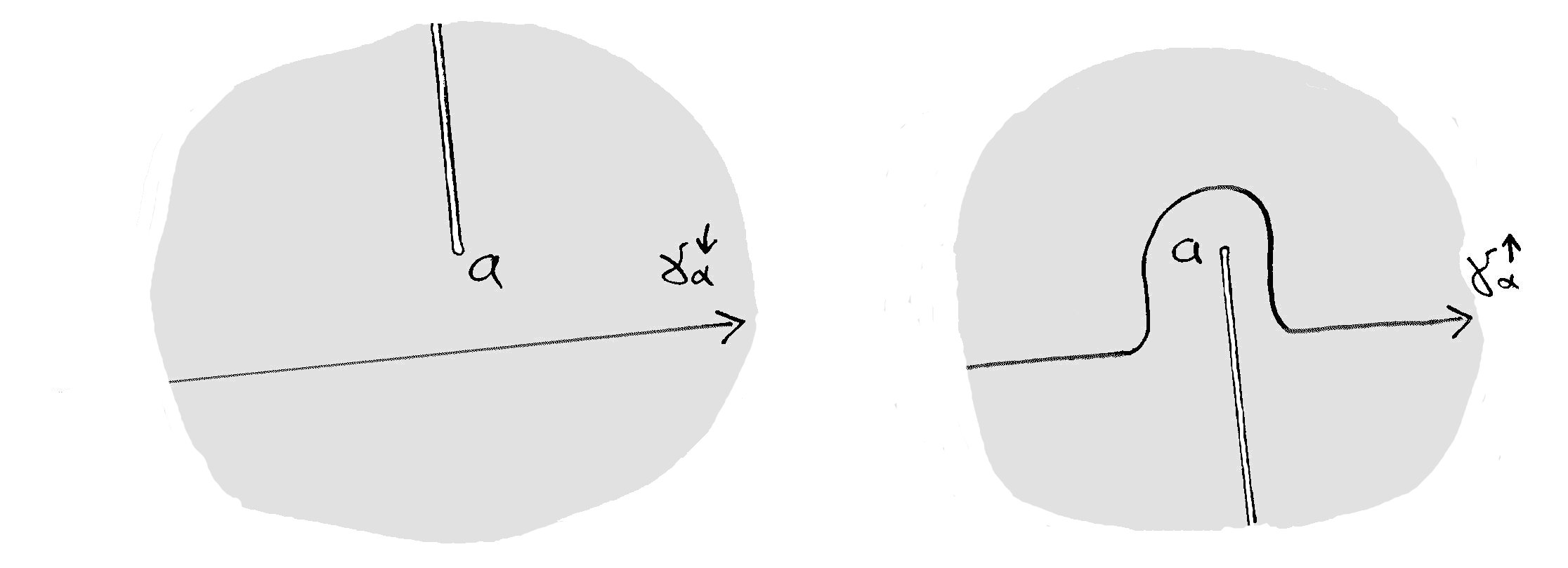}
\caption{The domains of definition of $\delta_a^\downarrow$ (resp. $\delta_a^\uparrow$) together with the deformed integration paths $\gamma_\alpha^\downarrow$ (resp. $\gamma_\alpha^\uparrow$).}
\label{figure:BL-0a}
\end{figure}


\section{The unfolded Borel and Laplace transformations associated to the vector field $(x^2\!-\!\epsilon)\frac{\partial}{\partial x}$}\label{sec:BL-3}

In this section we define the unfolded Borel and Laplace transformations $\Cal B_\alpha$, $\Cal L_\alpha$ \eqref{eq:BL-unfoldedBL} and summarize their basic properties.
We need to specify:
\begin{itemize}[noitemsep,topsep=2pt,parsep=2pt,partopsep=0pt]
\item[-] the branch of the multivalued time function $t(x,\epsilon)$ \eqref{eq:BL-t} of the kernel,
\item[-] the paths of integration,
\item[-] the domains in $x$-space and $\xi$-space where the transformations live,
\item[-] sufficient conditions on functions for which the transformations exist.
\end{itemize}
We provide these depending analytically on a root parameter $\sqrt\epsilon\in\C$. 
Here $\sqrt\epsilon$ is to be interpreted simply as a symbol for a new parameter 
(a coordinate on the ``$\sqrt\epsilon$-plane''), that naturally projects on the original parameter $\epsilon=(\sqrt\epsilon)^2$. 

\medskip

Let $t(x,\epsilon)$ \eqref{eq:BL-t}
be the complex time of the vector field $-(x^2\!-\epsilon)\frac{\partial}{\partial x}$ with $t(\infty,\epsilon)=0$, which is well defined for $x\in\CP^1\sminus[-\sqrt\epsilon,\sqrt\epsilon]$ and extended analytically as a ramified function.
Let us remark that the limit of the Riemann surface of $t(\cdot,\epsilon)$ 
as $\epsilon\to 0$ 
is composed of $\Z$-many complex planes identified at the origin, but the Riemann surface of $t(\cdot,0)$ 
is just the punctured $x$-plane in the middle.

\begin{definition}
For $0\leq \Lambda<\frac{\pi}{2\sqrt{|\epsilon|}}$, denote
$$\mathbf{X}(\Lambda,\sqrt\epsilon):=\{x\in\C \mid |t(x,\epsilon)-k\tfrac{\pi i}{\sqrt\epsilon}|>\Lambda,\ k\in\Z \}$$
an open neighborhood of the origin in the $x$-plane (of radius $\sim\frac{1}{\Lambda}$ when $\epsilon$ is small) containing the roots $\pm\sqrt\epsilon$.

\smallskip

For a direction $\alpha\in\R$ and $0\leq \Lambda<-\frac{1}{2}\Re\big(\frac{e^{i\alpha}\pi i}{\sqrt{\epsilon}}\big)$, 
let $\mathbf{T}_\alpha^\pm(\Lambda,\sqrt\epsilon)$ be as in \eqref{eq:BL-TalphaLambda}, a slanted strip in direction $-\alpha+\frac{\pi}{2}$ in the $t$-coordinate passing between two discs of radius $\Lambda$
centered at $0$ and $\mp\frac{\pi i}{\sqrt\epsilon}$,
and let $\check{\mathbf{X}}_\alpha^\pm(\Lambda,\sqrt\epsilon)$ be its projection to the $x$-plane (see Figures~\ref{figure:BL-1} and \ref{figure:BL-2}). 
More precisely, we shall consider them as subsets of the ramified Riemann surface of  $t(\cdot,\epsilon)$. 
Then the limits of $\check{\mathbf{X}}_\alpha^\pm(\Lambda,\sqrt\epsilon)\,$ when $\epsilon\to 0$ radially, split each into two opposed discs of radius $\frac{1}{2\Lambda}$ tangent at the origin, 
of which only one lies inside the $x$-plane: 
$\,\check{\mathbf{X}}_\alpha^+(\Lambda,0)$ is a disc centered at $e^{i\alpha}\frac{1}{2\Lambda}$, and 
$\,\check{\mathbf{X}}_\alpha^-(\Lambda,0)$ is a disc centered at $-e^{i\alpha}\frac{1}{2\Lambda}$ (Figure~\ref{figure:BL-2}\,(b)).

The interior of the domains $X^\pm(\!\sqrt\epsilon)$ of Theorem~\ref{theorem:BL-2} are ramified unions of such domains $\check{\mathbf{X}}_\alpha^\pm(\Lambda,\sqrt\epsilon)$.

\begin{figure}[ht]
\centering
\subfigure[$\sqrt\epsilon\neq 0$]{\includegraphics[width=0.51\textwidth]{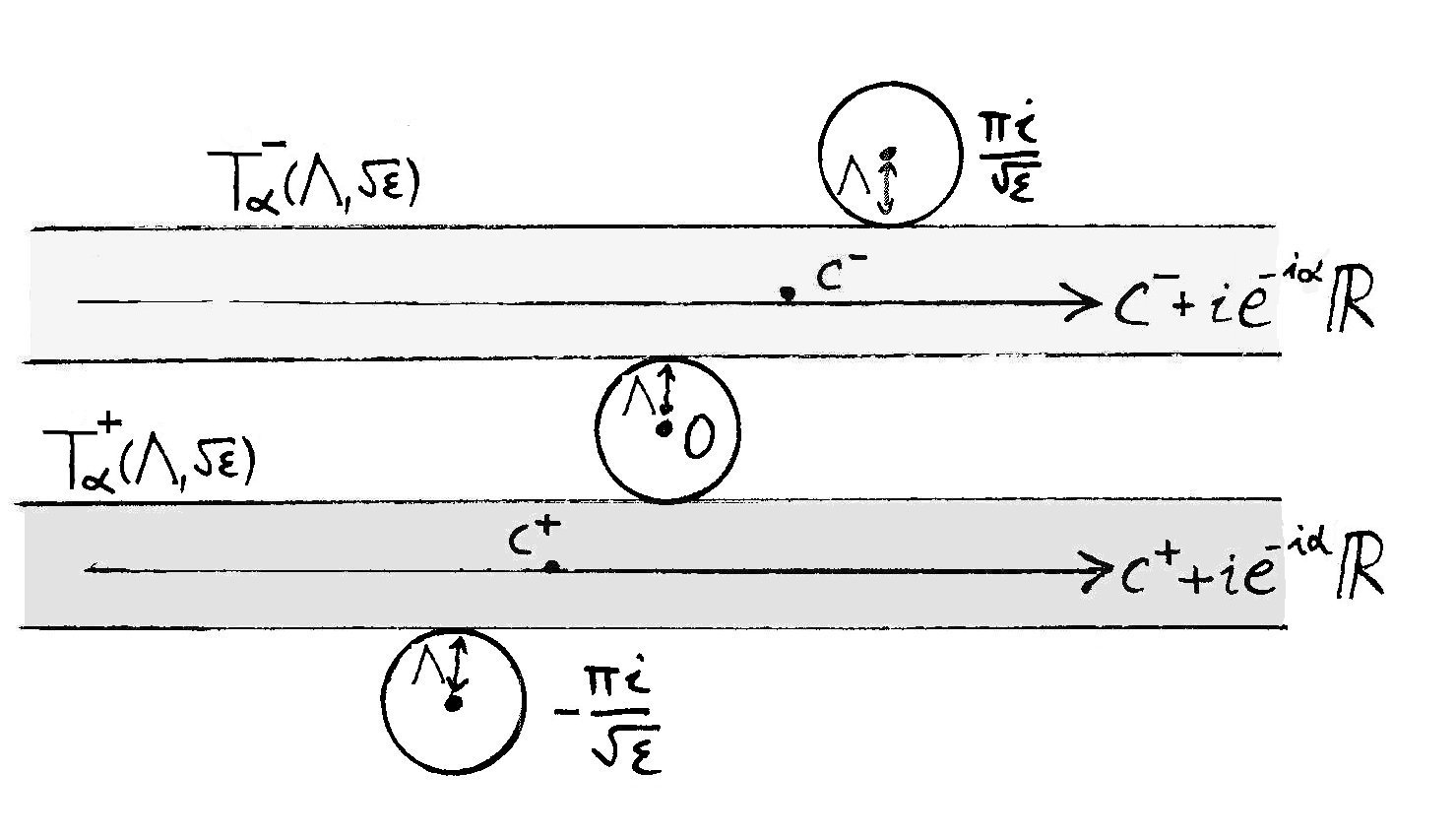}}
\subfigure[$\sqrt\epsilon=0$]{\includegraphics[width=0.46\textwidth]{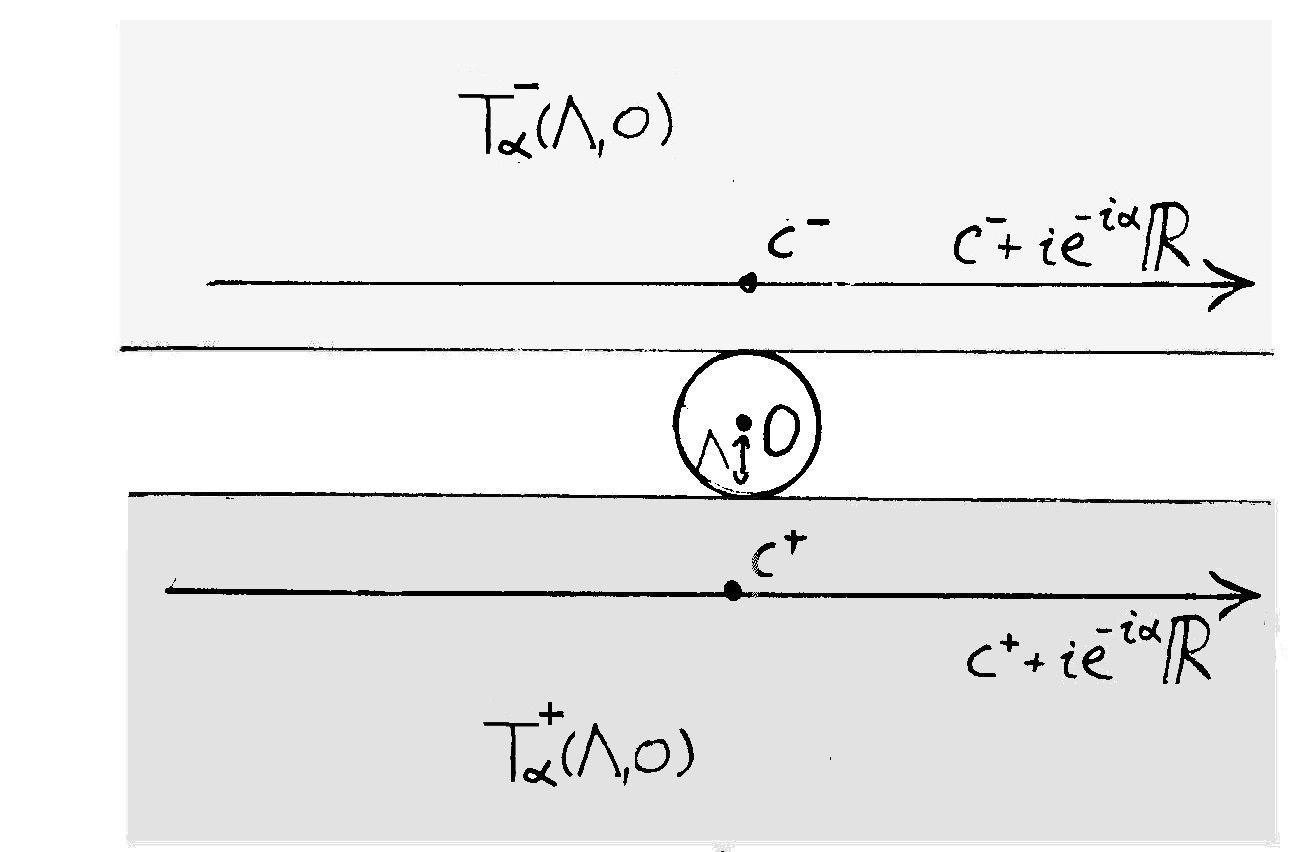}}
\caption{The domains $\mathbf{T}_\alpha^\pm(\Lambda,\sqrt\epsilon)$ in the time coordinate $t$ 
with the integration paths of the Borel transformation for $\alpha=\frac{\pi}{2}$.}
\label{figure:BL-1}
\end{figure}

\begin{figure}[ht]
\centering
\subfigure[$\sqrt\epsilon\neq 0$]{\includegraphics[width=0.46\textwidth]{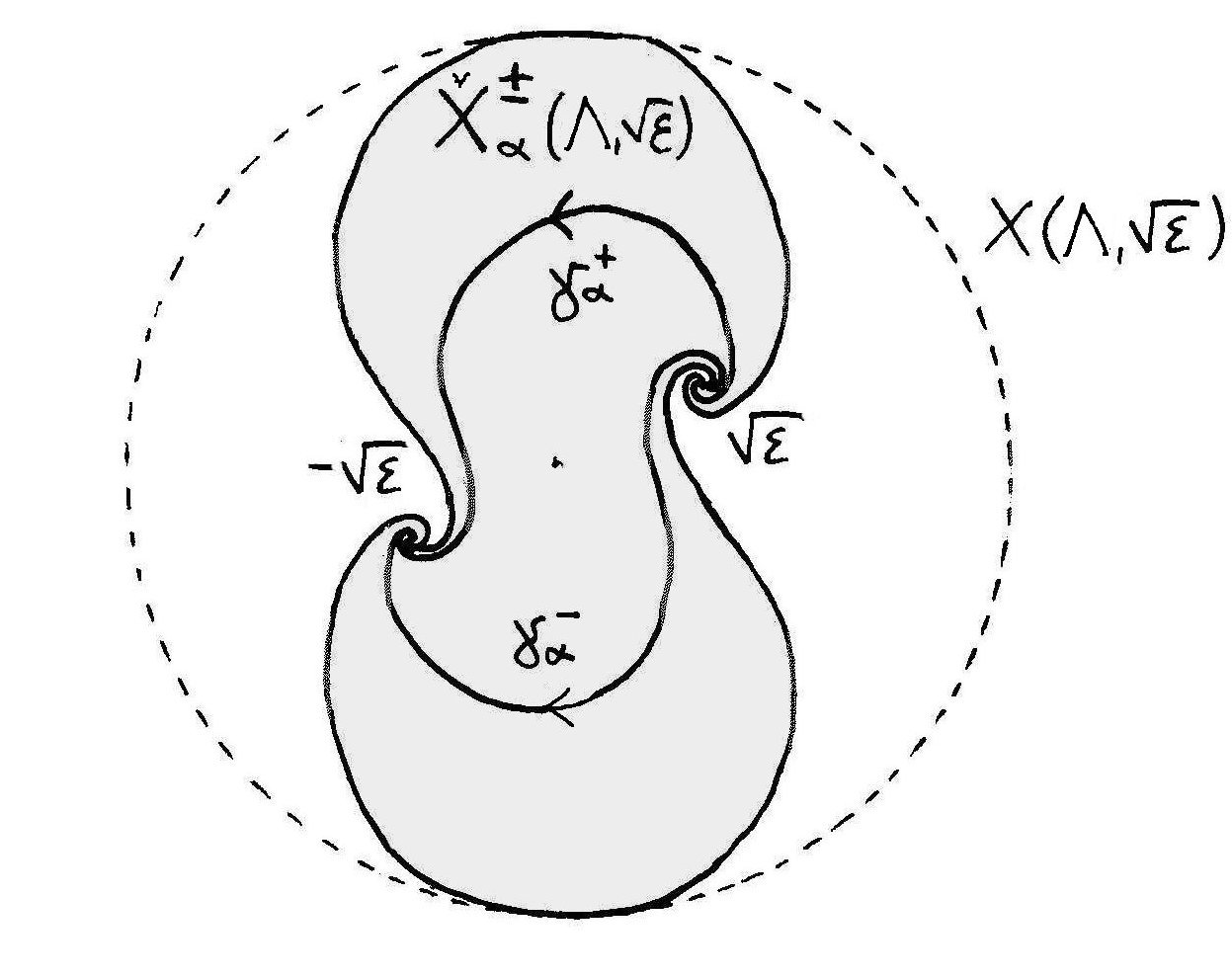}}
\subfigure[$\sqrt\epsilon=0$]{\includegraphics[width=0.5\textwidth]{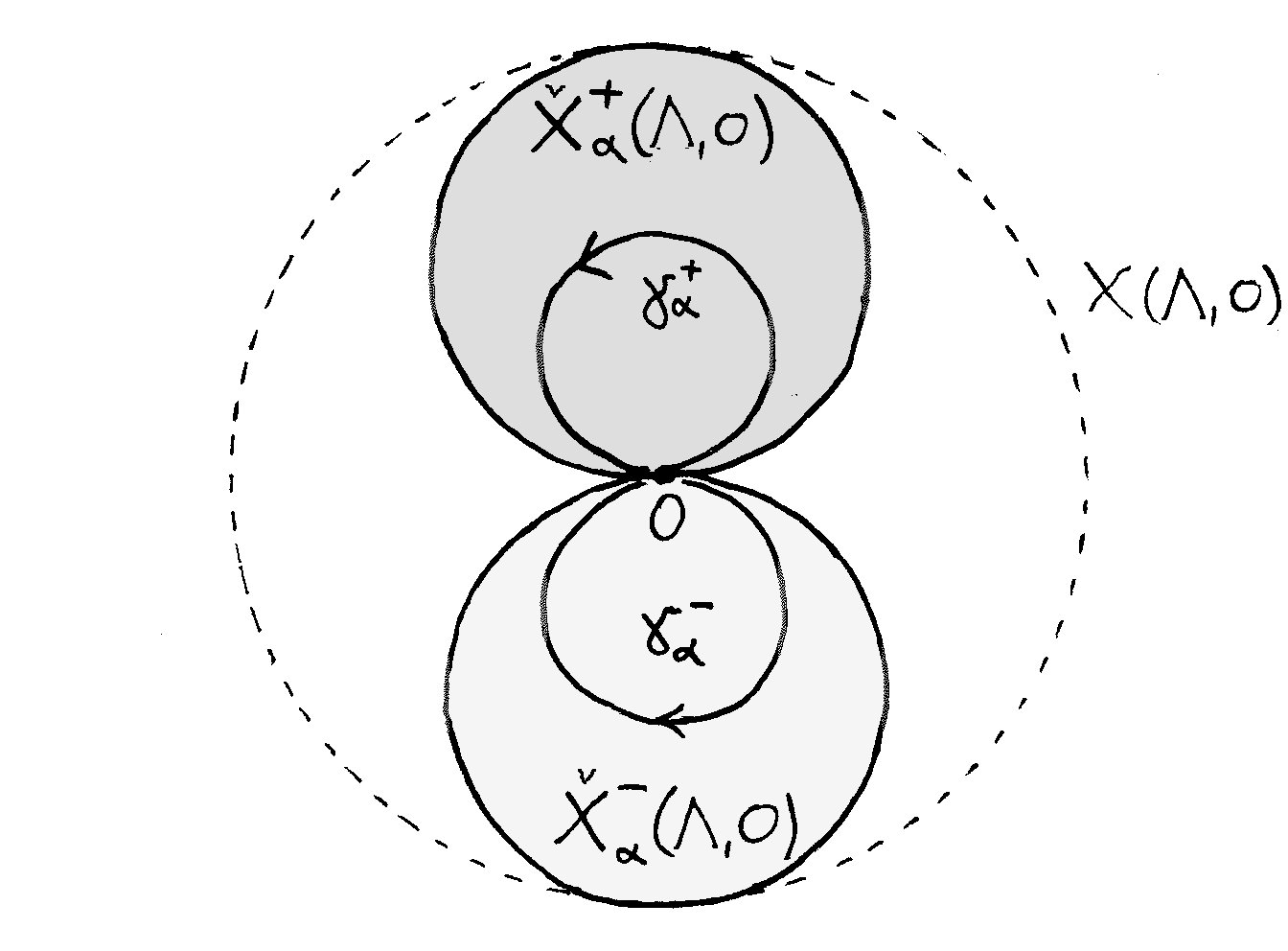}}
\caption{The domains $\check{\mathbf{X}}_\alpha^\pm(\Lambda,\sqrt\epsilon)$ projected to the $x$-plane for $\alpha=\frac{\pi}{2}$.
The integration paths $\gamma_\alpha^\pm$ are projections of the paths $c^\pm-ie^{-i\alpha}\R$ in the $t$-coordinate (they have opposite direction than those in Figure~\ref{figure:BL-1}).}
\label{figure:BL-2}
\end{figure}

\end{definition}

In order to apply the Borel transformation \eqref{eq:BL-B} in a direction $\alpha$ to a function $f$ analytic on the neighborhood 
$\mathbf{X}(\Lambda,\sqrt\epsilon)$,  
one may choose to lift $f$ to the Riemann surface of $t(\cdot,\sqrt\epsilon)$ either as a function on $\check{\mathbf{X}}_\alpha^+(\Lambda,\sqrt\epsilon)$ or to $\check{\mathbf{X}}_\alpha^-(\Lambda,\sqrt\epsilon)$ giving rise to two different transforms 
$\Cal B_\alpha^+[f]$ and $\Cal B_\alpha^-[f]$:

\begin{definition}\label{definition:BL-4}
Assume that $\check{\mathbf{X}}_\alpha^\pm(\Lambda,\sqrt\epsilon)\neq\emptyset$, $\alpha\in\,]\arg\sqrt\epsilon,\arg\sqrt\epsilon+\pi[$, and
let $f\in\Cal O\big(\check{\mathbf{X}}_\alpha^\pm(\Lambda,\sqrt\epsilon)\big)$  vanish at both points $\sqrt\epsilon,-\sqrt\epsilon$.
The \emph{unfolded Borel transforms} $\Cal B_\alpha^\pm[f]$ are defined as:
\begin{equation*}
\Cal B_\alpha^\pm [f](\xi,\sqrt\epsilon)=\tfrac{1}{2\pi i} \int_{c^\pm-\infty ie^{-i\alpha}}^{c^\pm+\infty ie^{-i\alpha}} f(x(t,\epsilon))\, e^{t\xi}\,d t, 
\qquad c^\pm\in {T}_\alpha^\pm(\Lambda,\sqrt\epsilon).
\end{equation*}

\noindent
\emph{For $\sqrt\epsilon\neq 0$}\,:
If $x\in\check{\mathbf{X}}_\alpha^\pm(\Lambda,\sqrt\epsilon)$, then 
\, $t(x,\epsilon)=-\tfrac{1}{2\sqrt\epsilon}\left(\log\frac{\sqrt\epsilon-x}{\sqrt\epsilon+x}\pm\pi i\right),\,$ 
\begin{equation}
\label{eq:BL-B+-}
\Cal B_\alpha^\pm [f](\xi,\sqrt\epsilon) = e^{\mp\frac{\xi\pi i}{2\sqrt\epsilon}}\cdot \tfrac{1}{2\pi i} 
\int_{\gamma_\alpha^\pm} \tfrac{f(x)}{x^2\!-\epsilon}\, \left(\tfrac{\sqrt\epsilon-x}{\sqrt\epsilon+x}\right)^{-\frac{\xi}{2\sqrt\epsilon}}\,dx, 
\end{equation}
where the integration path $\gamma_\alpha^\pm$ (see Figure~\ref{figure:BL-2}) follows a real time trajectory of the vector field $ie^{-i\alpha}(x^2\!-\!\epsilon)\frac{\partial}{\partial x}$ inside $\check{\mathbf{X}}_\alpha^\pm(\Lambda,\sqrt\epsilon)$.
Hence
\begin{align}\label{eq:BL-Bpm}
\Cal B_\alpha^-[f](\xi,\sqrt\epsilon)&=e^{\frac{\xi\pi i}{\sqrt\epsilon}}\cdot \Cal B_\alpha^+[f](\xi,\sqrt\epsilon)\\[6pt]
 &=-\Cal B_{\alpha+\pi}^+[f](\xi,e^{\pi i}\sqrt\epsilon),\label{eq:BL-Bpm2}
\end{align}
as $\check{\mathbf{X}}_\alpha^-(\Lambda,\sqrt\epsilon)=\check{\mathbf{X}}_{\alpha+\pi}^+(\Lambda,e^{\pi i}\sqrt\epsilon)$.

\medskip
\noindent 
\emph{For $\sqrt\epsilon=0$}\,:
\begin{equation}
\Cal B_\alpha^\pm [f](\xi,0) = \tfrac{1}{2\pi i} \int_{\gamma_\alpha^\pm} \tfrac{f(x)}{x^2}\, e^{\frac{\xi}{x}}\,dx,
\end{equation}
where $\gamma_\alpha^\pm$ is a real time trajectory of the vector field $ie^{-i\alpha}x^2\frac{\partial}{\partial x}$ inside $\check{\mathbf{X}}_\alpha^\pm(\Lambda,0)$. 
It is the radial limit of the precedent case as $\sqrt\epsilon\to 0$,
$$\Cal B_\alpha^\pm[f](\xi,0)=\lim_{\nu\to 0+} \Cal B_\alpha^\pm[f](\xi,\nu\sqrt\epsilon).$$
The transformation $\Cal B_\alpha^+ [f](\xi,0)$ is the standard analytic Borel transform \eqref{eq:BL-1} in direction $\alpha$, and
\begin{equation}
\Cal B_\alpha^-[f](\xi,0)=-\Cal B_{\alpha+\pi}^+[f](\xi,0).
\end{equation}

\end{definition}

The following proposition summarizes some basic proprieties of these unfolded Borel transformations.

\begin{proposition} \label{proposition:BL-4}~
Let $\alpha$ be a direction, and suppose that  $\,\arg\sqrt\epsilon\in\,]\alpha-\pi,\alpha[$ if $\epsilon\neq 0$.

\smallskip\noindent
1) If $\sqrt\epsilon\neq 0$, let a function  $f\in\Cal O(\check{\mathbf{X}}_\alpha^\pm(\Lambda,\sqrt\epsilon))$, 
be uniformly $O(|x\!-\!\sqrt\epsilon|^a|x\!+\!\sqrt\epsilon|^b)$ at the points $\pm\sqrt\epsilon$, for some $a,b\in\R$ with $a+b>0$.
Then the transforms $\Cal B_\alpha^\pm [f](\xi,\sqrt\epsilon)$ converge absolutely for $\xi$ in the strip
\begin{equation}\label{eq:BL-4}
\Omega_\alpha=\{-\Im(e^{-i\alpha}2b\sqrt\epsilon) > \Im(e^{-i\alpha}\xi) > \Im(e^{-i\alpha}2a\sqrt\epsilon) \},
\qquad\text{see Figure~\ref{figure:BL-3}},
\end{equation}
and are analytic extensions of each other for varying $\alpha$.
Moreover for any $\Lambda<\Lambda_1<-\Re\big(\frac{e^{i\alpha}\pi i}{2\sqrt\epsilon}\big):=\lambda_\alpha(\!\sqrt\epsilon)$,
they are of bounded norm $|\Cal B_\alpha^\pm [f]|_{c+e^{i\alpha}\R}^{\Lambda_1,\lambda_\alpha-\Lambda_1}$ on  any line $c+e^{i\alpha}\R\subseteq\Omega_\alpha$.

\begin{figure}[ht]
\vskip-12pt
\centering
\includegraphics[width=0.55\textwidth]{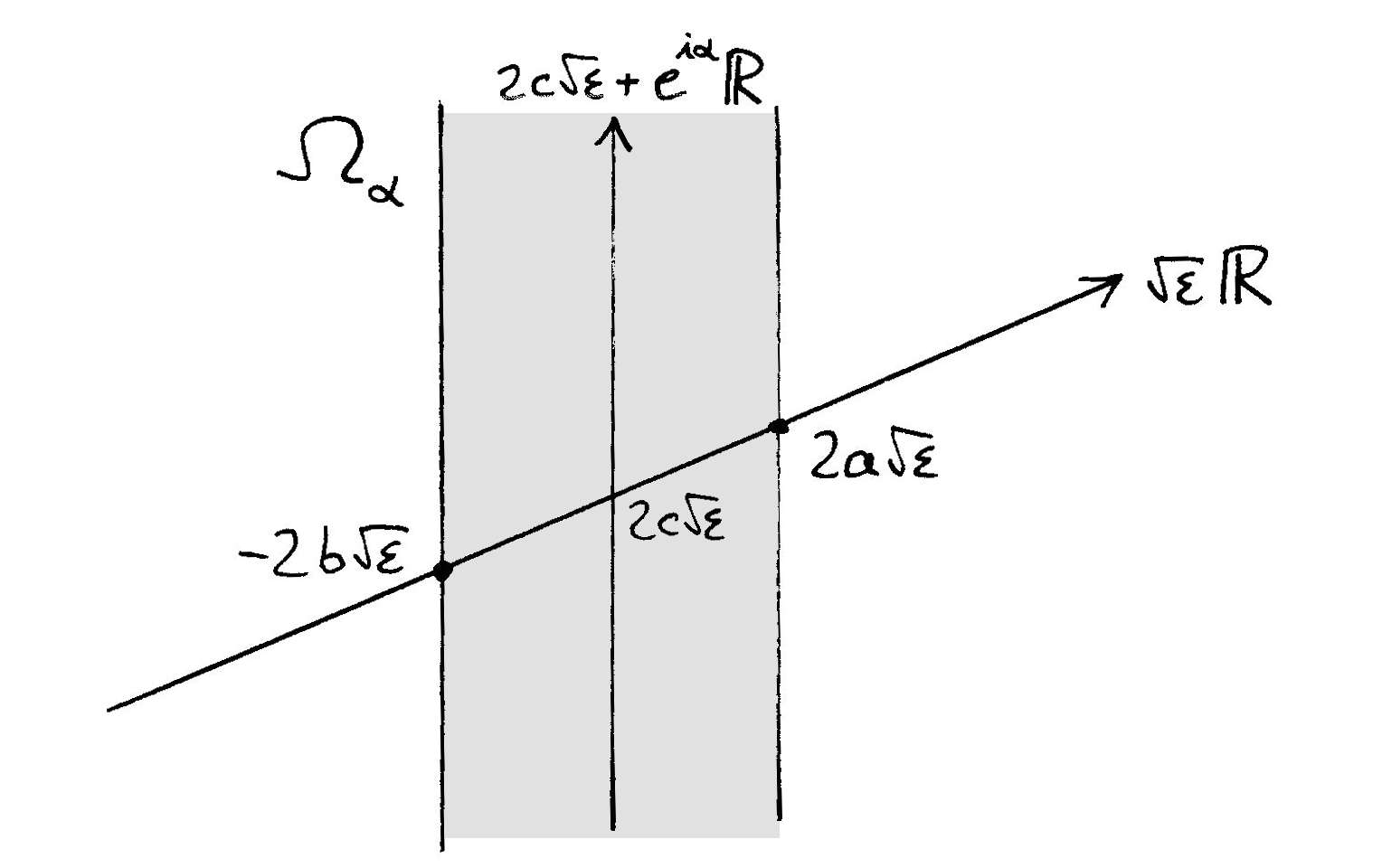}
\caption{The strip $\Omega_\alpha$ in the $\xi$-plane.}
\label{figure:BL-3}
\end{figure}

\smallskip\noindent
2)  If $\sqrt\epsilon\neq0$ and $a+b>0$, then for $\xi\in\Omega_\alpha$ (defined in \eqref{eq:BL-4})
$$\Cal B_\alpha^+[(x\!-\!\sqrt\epsilon)^a(x\!+\!\sqrt\epsilon)^b](\xi,\sqrt\epsilon)=
e^{-\frac{\xi\pi i}{2\sqrt\epsilon}+a\pi i}\cdot (2\sqrt\epsilon)^{a+b-1}\cdot \tfrac{1}{2\pi i}\,B(a-\tfrac{\xi}{2\sqrt\epsilon},b+\tfrac{\xi}{2\sqrt\epsilon}),$$
where $B$ is the Beta function.

\medskip\noindent
3) In particular, for a positive integer $n$, and $\xi$ in the strip in between $0$ and $2n\sqrt\epsilon$,
$$\Cal B_\alpha^\pm[(x\!-\!\sqrt\epsilon)^n](\xi,\sqrt\epsilon)=\chi_\alpha^\pm(\xi,\sqrt\epsilon)\cdot
\big(\tfrac{\xi}{(n-1)}-2\sqrt\epsilon\big)\cdot\big(\tfrac{\xi}{(n-2)}-2\sqrt\epsilon\big)\cdot\ldots\cdot\big(\tfrac{\xi}{1}-2\sqrt\epsilon\big),
$$
where
for $\sqrt\epsilon\neq 0$ and $\alpha\in\,]\arg\sqrt\epsilon,\arg\sqrt\epsilon+\pi[$ 
\begin{equation}\label{eq:BL-chi} 
\chi_\alpha^+(\xi,\sqrt\epsilon):= \frac{1\ \ }{1-e^{\frac{\xi\pi i}{\sqrt\epsilon}}}, \qquad
\chi_\alpha^-(\xi,\sqrt\epsilon):= \frac{-1\ \ \ }{1-e^{-\frac{\xi\pi i}{\sqrt\epsilon}}}=\chi_\alpha^+(\xi,\sqrt\epsilon)-1,
\end{equation}
and for $\sqrt\epsilon=0$
\begin{equation*}
\chi_\alpha^+(\xi,0):=\left\{ \begin{array}{ll}     
     1, & \text{if $\ \xi\in\,]0,+\infty e^{i\alpha}[$,}\\[3pt]
     \frac{1}{2}, & \text{if $\ \xi=0$,}\\[3pt]
     0, & \text{if $\ \xi\in\,]\!-\!\infty e^{i\alpha},0[$,}
   \end{array}\right. \qquad \chi_\alpha^-(\xi,0):=\, \chi_\alpha^+(\xi,0)-1.\qquad
\end{equation*}
Let us remark that $\,\chi_\alpha^\pm(\xi,\nu\sqrt\epsilon)\xrightarrow{\nu\to 0+}\chi_\alpha^\pm(\xi,0)\,$ for $\xi\in e^{i\alpha}\R\sminus\{0\}$.

\medskip\noindent
4) If $f(x)$ is analytic on an open disc of radius $r>2\sqrt{|\epsilon|}$ centered at $x_0=-\sqrt\epsilon$ 
(or $x_0=\sqrt\epsilon$) and $f(x_0)=0$, then
$$\Cal B_\alpha^\pm[f](\xi,\sqrt\epsilon)=\chi_\alpha^\pm(\xi,\sqrt\epsilon)\cdot\phi(\xi)$$
where $\phi$ is is an entire function with at most exponential growth at infinity
$\,\leq e^\frac{|\xi|}{R-2\sqrt{|\epsilon|}}\cdot O(\sqrt{|\xi|})\,$ for any $2\sqrt{|\epsilon|}<R<r$  (where the big $O$ is uniform for $(\xi,\sqrt\epsilon)\to(\infty, 0)$).

\medskip\noindent
5) For $\sqrt\epsilon\neq 0$, $c\in\C$, the Borel Transform $\,\Cal B_\alpha^\pm\!\left[\left(\frac{x\!-\!\sqrt\epsilon}{x\!+\!\sqrt\epsilon}\right)^{\!c}\right]\!(\xi,\sqrt\epsilon)=\delta_{2c\sqrt\epsilon}(\xi)\,$ 
is the Dirac mass at $2c\sqrt\epsilon$, acting as translation operator on the Borel plane by
$\xi\mapsto \xi-2c\sqrt\epsilon$:
$$\Cal B_\alpha^\pm\!\left[\left(\tfrac{x\!-\!\sqrt\epsilon}{x\!+\!\sqrt\epsilon}\right)^c\!\!\cdot f\right]\!(\xi,\sqrt\epsilon)=\Cal B_\alpha^\pm [f](\xi-2c\sqrt\epsilon,\sqrt\epsilon).$$
\end{proposition}

\smallskip
\begin{remark}
Although in \textit{1)} and  \textit{2)} of Proposition~\ref{proposition:BL-4} the function $f=O((x\!-\!\sqrt\epsilon)^a(x\!+\!\sqrt\epsilon)^b)$, $a+b>0$, might not vanish at both points $\pm\sqrt\epsilon$ as demanded in Definition~\ref{definition:BL-4}, 
one can write
$$(x\!-\!\sqrt\epsilon)^a(x\!+\!\sqrt\epsilon)^b=\left(\tfrac{x\!-\!\sqrt\epsilon}{x\!+\!\sqrt\epsilon}\right)^c(x\!-\!\sqrt\epsilon)^{a-c}(x\!+\!\sqrt\epsilon)^{b+c},
\quad\text{ for any } \ -b<c<a,$$
hence, using \textit{5)} of Proposition~\ref{proposition:BL-4}, the Borel transform $\Cal B_\alpha^\pm[f]$
is well defined as the translation by $2c\sqrt\epsilon$ of the Borel transform of the function $f\cdot\left(\frac{x\!-\!\sqrt\epsilon}{x\!+\!\sqrt\epsilon}\right)^{\!-c}$, this time vanishing at both points:
$$\Cal B_\alpha^\pm [f](\xi,\sqrt\epsilon)=\Cal B_\alpha^\pm \big[f\cdot\left(\tfrac{x\!-\!\sqrt\epsilon}{x\!+\!\sqrt\epsilon}\right)^{\!-c}\big](\xi-2c\sqrt\epsilon,\sqrt\epsilon).$$
\end{remark}

\smallskip
\begin{proof}[Proof of Proposition~\ref{proposition:BL-4}]
\textit{1)} For $\sqrt\epsilon\neq 0$, one can express 
$$x\!-\!\sqrt\epsilon=2\sqrt\epsilon\frac{e^{-2\sqrt\epsilon\,t}}{1-e^{-2\sqrt\epsilon\,t}}, \quad x\!+\!\sqrt\epsilon=2\sqrt\epsilon\frac{1}{1-e^{-2\sqrt\epsilon\,t}}.$$
If $\xi$ is in the strip $\Omega_\alpha$, \, $\xi\in 2c\sqrt\epsilon+e^{i\alpha}\R\,$ for some $c\in\,]\!-\!b,a[$, one writes
\begin{equation*}
\begin{split}
\Cal B_\alpha^\pm [(x\!-\!\sqrt\epsilon)^a(x\!+&\!\sqrt\epsilon)^b](\xi,\sqrt\epsilon)= \\
&=(2\sqrt\epsilon)^{a+b}\tfrac{1}{2\pi i}\int_{C^\pm+e^{-i\alpha}i\R} \frac{(e^{-2\sqrt\epsilon\,t})^{a-c}}{(1-e^{-2\sqrt\epsilon\,t})^{a+b}}\, e^{(\xi-2c\sqrt\epsilon)t}\,dt.
\end{split}
\end{equation*}
The term $e^{(\xi-2c\sqrt\epsilon)t}$ stays bounded along the integration path, while the term $\frac{e^{-2\sqrt\epsilon\,t\,(a-c)}}{(1-e^{-2\sqrt\epsilon\,t})^{a+b}}$ decreases
exponentially fast as $t\!-\!C^\pm\to +\infty i e^{-i\alpha}$ and $t\!-\!C^\pm \to -\infty i e^{-i\alpha}\,$, if $\alpha\notin \arg\sqrt\epsilon+\pi\Z$.

\smallskip
\textit{2)} From \eqref{eq:BL-B+-}
\begin{equation*}
\begin{split}
\Cal B_\alpha^+[(x\!-\!\sqrt\epsilon)^a(x\!+&\!\sqrt\epsilon)^b](\xi,\sqrt\epsilon)=\\
&= -e^{-\frac{\xi\pi i}{2\sqrt\epsilon}+a\pi i}\, \tfrac{1}{2\pi i} \int_{\gamma_\alpha^+}\! 
(\sqrt\epsilon\!-\!x)^{a-1-\frac{\xi}{2\sqrt\epsilon}}(\sqrt\epsilon\!+\!x)^{b-1+\frac{\xi}{2\sqrt\epsilon}}\,dx \\
&=  e^{-\frac{\xi\pi i}{2\sqrt\epsilon}+a\pi i} (2\sqrt\epsilon)^{a+b-1} 
\tfrac{1}{2\pi i}\int_{0}^{1}\! (1-s)^{a-1-\frac{\xi}{2\sqrt\epsilon}}\,s^{b-1+\frac{\xi}{2\sqrt\epsilon}}\,ds,  
\end{split}
\end{equation*}
substituting $s=\frac{\sqrt\epsilon+x}{2\sqrt\epsilon}$.
For $\alpha=\arg\sqrt\epsilon+\frac{\pi}{2}$, the integration path $\gamma_\alpha^+$ 
(= a real trajectory of the vector field $e^{-i\arg\sqrt\epsilon}(x^2\!-\!\epsilon)\frac{\partial}{\partial x}$) 
can be chosen as the straight oriented segment $(\sqrt\epsilon,-\sqrt\epsilon)$. The result follows.

\smallskip
\textit{3)} From \textit{2)} using standard formulas.\, 

\smallskip
\textit{4)} 
For $x_0=-\sqrt\epsilon$, one can write $f(x)$ as a convergent series $f(x)=\sum_{n=1}^{+\infty} a_n\,(x\!+\!\sqrt\epsilon)^n$ with $|a_n|\leq CK^n$ for some $C>0$ and $\frac{1}{r}<K<\frac{1}{R}$.
Hence
$$\big(1-e^{\frac{\xi\pi i}{\sqrt\epsilon}}\big)\cdot\Cal B^+[f](\xi,\sqrt\epsilon)=
\sum_{n=1}^{+\infty} a_n \left(\tfrac{\xi}{n-1}-2\sqrt\epsilon\right)\cdots\left(\tfrac{\xi}{1}-2\sqrt\epsilon\right)=:
\sum_{n=1}^{+\infty} b_n(\xi,\sqrt\epsilon),$$
where the series on the right is absolutely convergent for any $\xi\in\C$.
Indeed, let $N=N(\xi,\sqrt{\epsilon})$ be the positive integer such that 
\begin{equation}\label{eq:BL-104}
\tfrac{|\xi|}{N+1}\leq R-2\sqrt{|\epsilon|}<\tfrac{|\xi|}{N},
\end{equation}
then
\begin{itemize}
\item for $n\geq N+1$: $K\cdot(\frac{|\xi|}{n}+2\sqrt{|\epsilon|})\leq RK,\,$ 
\item for $n\leq N$: $2\sqrt{|\epsilon|}<\frac{2\sqrt{|\epsilon|}}{R-2\sqrt{|\epsilon|}}\frac{|\xi|}{n}\ $ and hence
$$\ K\cdot(\tfrac{|\xi|}{n}+2\sqrt{|\epsilon|})\leq \tfrac{K|\xi|}{n}\big(1+\tfrac{2\sqrt{|\epsilon|}}{R-2\sqrt{|\epsilon|}}\big)\leq\tfrac{1}{n}\cdot\tfrac{|\xi|}{R-2\sqrt{|\epsilon|}}.$$
\end{itemize}
\begin{align*}
\sum_{n=1}^{+\infty} |b_n(\xi,&\sqrt\epsilon)| = \sum_{n=0}^{N-1} |b_{n+1}(\xi,\sqrt\epsilon)|+\sum_{n=N}^{+\infty} |b_{n+1}(\xi,\sqrt\epsilon)| \\
&\leq \sum_{n=0}^{N-1} CK\tfrac{1}{n!}\left(\tfrac{|\xi|}{R-2\sqrt{|\epsilon|}}\right)^n+ CK\tfrac{1}{N!}\left(\tfrac{|\xi|}{R-2\sqrt{|\epsilon|}}\right)^N \cdot\sum_{n=N}^{+\infty} (RK)^{n-N} \\
&\leq CKe^\frac{|\xi|}{R-2\sqrt{|\epsilon|}}+CK\cdot{\Gamma\Big(\tfrac{|\xi|}{R-2\sqrt{|\epsilon|}}\Big)}^{-1}\left(\tfrac{|\xi|}{R-2\sqrt{|\epsilon|}}\right)^\frac{|\xi|}{R-2\sqrt{|\epsilon|}}\cdot\tfrac{1}{1-RK} \\
&= e^\frac{|\xi|}{R-2\sqrt{|\epsilon|}}\cdot \left(CK+\tfrac{CK}{1-RK}\sqrt{\tfrac{|\xi|}{2\pi(R-2\sqrt{|\epsilon|})}}+O\Big(\sqrt{\tfrac{R-2\sqrt{|\epsilon|}}{|\xi|}}\Big)  \right), 
\end{align*}
using \eqref{eq:BL-104} and the Stirling formula: $\,\Gamma(z)^{-1}=\left(\frac{e}{z}\right)^z\cdot\left(\sqrt{\frac{z}{2\pi}}+O\big(\frac{1}{\sqrt z}\big)\right),\ z\to+\infty$.

\smallskip
\textit{5)} From the definition.
\qed\end{proof}


There is also a converse statement to point \textit{1)} of Proposition~\ref{proposition:BL-4}.

\begin{proposition}
Let $\epsilon\neq 0$ and $\alpha\in\,]\arg\sqrt\epsilon,\arg\sqrt\epsilon+\pi[$.
If $\phi(\xi)$ is an analytic function in  a strip $\Omega_\alpha$ \eqref{eq:BL-4}, with $a+b>0$, 
such that it has a finite norm $|\phi|_{2c\sqrt\epsilon+e^{i\alpha}\R}^{\Lambda, \lambda_\alpha-\Lambda}$ on each line $2c\sqrt\epsilon+e^{i\alpha}\R\subseteq\Omega_\alpha$,
for some $0\leq \Lambda <\lambda_\alpha(\!\sqrt\epsilon):=-\Re\big(\frac{e^{i\alpha}\pi i}{2\sqrt\epsilon})$,
then the \emph{unfolded Laplace transform} of $\phi$
\begin{equation}
\Cal L_\alpha[\phi](x,\sqrt\epsilon)=\int_{2c\sqrt\epsilon-\infty e^{i\alpha}}^{2c\sqrt\epsilon+\infty e^{i\alpha}} \phi(\xi) e^{-t(x,\epsilon)\xi}\,d\xi,\quad c\in\,]\!-\!b,a[
\end{equation}
is analytic on the domain $\check{\mathbf{X}}_\alpha^+(\Lambda,\sqrt\epsilon)$, and is uniformly $o(|x\!-\!\sqrt\epsilon|^{a_1}|x\!+\!\sqrt\epsilon|^{b_1})$ for any $a_1<a,\ b_1<b$, on any
sub-domain $\check{\mathbf{X}}_\alpha^+(\Lambda_1,\sqrt\epsilon)$, $\Lambda_1>\Lambda$.
\end{proposition}

\begin{proof}
This is a reformulation of Corollary~\ref{corollary:BL-1}, 
which also implies that $\Cal L_\alpha[\phi]$ is $o\!\left(\left|\frac{x-\sqrt\epsilon}{x+\sqrt\epsilon}\right|^c\right)$ for any $-b<c<a$.
\qed\end{proof}

\begin{definition}[Borel transform of $x$]\label{def:widetildex}
We know form Proposition~\ref{proposition:BL-4} that for $\sqrt\epsilon\neq0$, $\Cal B_\alpha^\pm[x+\!\sqrt\epsilon]=\chi_\alpha^\pm$ in the strip in between $-2\sqrt\epsilon$ and $0$, 
while $\Cal B_\alpha^\pm[x-\!\sqrt\epsilon]=\chi_\alpha^\pm$ in the strip in between $0$ and $2\sqrt\epsilon$,
and the function $\chi_\alpha^\pm$ has a simple pole at 0 with residue $\res_0 \chi_\alpha^\pm=\frac{\sqrt\epsilon}{\pi i}$, therefore
$$\Cal B_\alpha^\pm[x+\!\sqrt\epsilon]-\Cal B_\alpha^\pm[x-\!\sqrt\epsilon]=2\sqrt\epsilon\,\delta_0$$
in the sense of distributions (see Section~\ref{sec:BL-distributions}), where $\delta_0$ is the Dirac distribution (identity of convolution). 
Hence one can define the distribution 
$$\Cal B_\alpha^\pm[x]:=\Cal B_\alpha^\pm[x-\!\sqrt\epsilon]+\sqrt\epsilon\,\delta_0=\Cal B_\alpha^\pm[x+\!\sqrt\epsilon]-\sqrt\epsilon\,\delta_0.$$
Correspondingly, the convolution of $\Cal B_\alpha^\pm[x]$ with a function $\phi$, analytic on an open strip in direction $\alpha$, is then defined as 
\begin{align*}
[\Cal B_\alpha^\pm[x]*\phi]_\alpha(\xi,\sqrt\epsilon)
&=\int_{c_1+e^{i\alpha}\R}\!\!\!\!\phi(\xi\!-\!\zeta)\,\chi_\alpha^\pm(\zeta,\sqrt\epsilon)\, d\zeta+\sqrt\epsilon\,\phi(\xi),\quad c_1\in\,]0,2\sqrt\epsilon[ \\
&=\int_{c_2+e^{i\alpha}\R}\!\!\!\!\phi(\xi\!-\!\zeta)\,\chi_\alpha^\pm(\zeta,\sqrt\epsilon)\, d\zeta-\sqrt\epsilon\,\phi(\xi),\quad c_2\in\,]\!-\!2\sqrt\epsilon,0[.
\end{align*}
\end{definition}


\subsection{Remark on Fourier expansions}

For $\sqrt\epsilon\neq 0$, we have defined the Borel transformations $\Cal B_\alpha^\pm$ for directions transverse to $\sqrt\epsilon\,\R$: 
in fact, we have restricted ourselves to $\alpha\in\,]\arg\sqrt\epsilon,\arg\sqrt\epsilon+\pi[$.
Let us now take a look at the direction $\arg\sqrt\epsilon$. 
So instead of integrating on a line $c^\pm+ i e^{-i\alpha}\R$ in the $t$-coordinate as in Figure~\ref{figure:BL-1}, this time we shall consider an integrating path
$c_R+\frac{i}{\sqrt\epsilon}\R$ in the half plane $\Re(e^{i\arg\sqrt\epsilon}t)>\Lambda$ \, 
(resp. $c_L+\frac{i}{\sqrt\epsilon}\R$ in the half plane $\Re(e^{i\arg\sqrt\epsilon}t)<-\Lambda\,$),
see  Figure~\ref{figure:BL-3a}.

\begin{figure}[ht]
\centering
\includegraphics[width=0.5\textwidth]{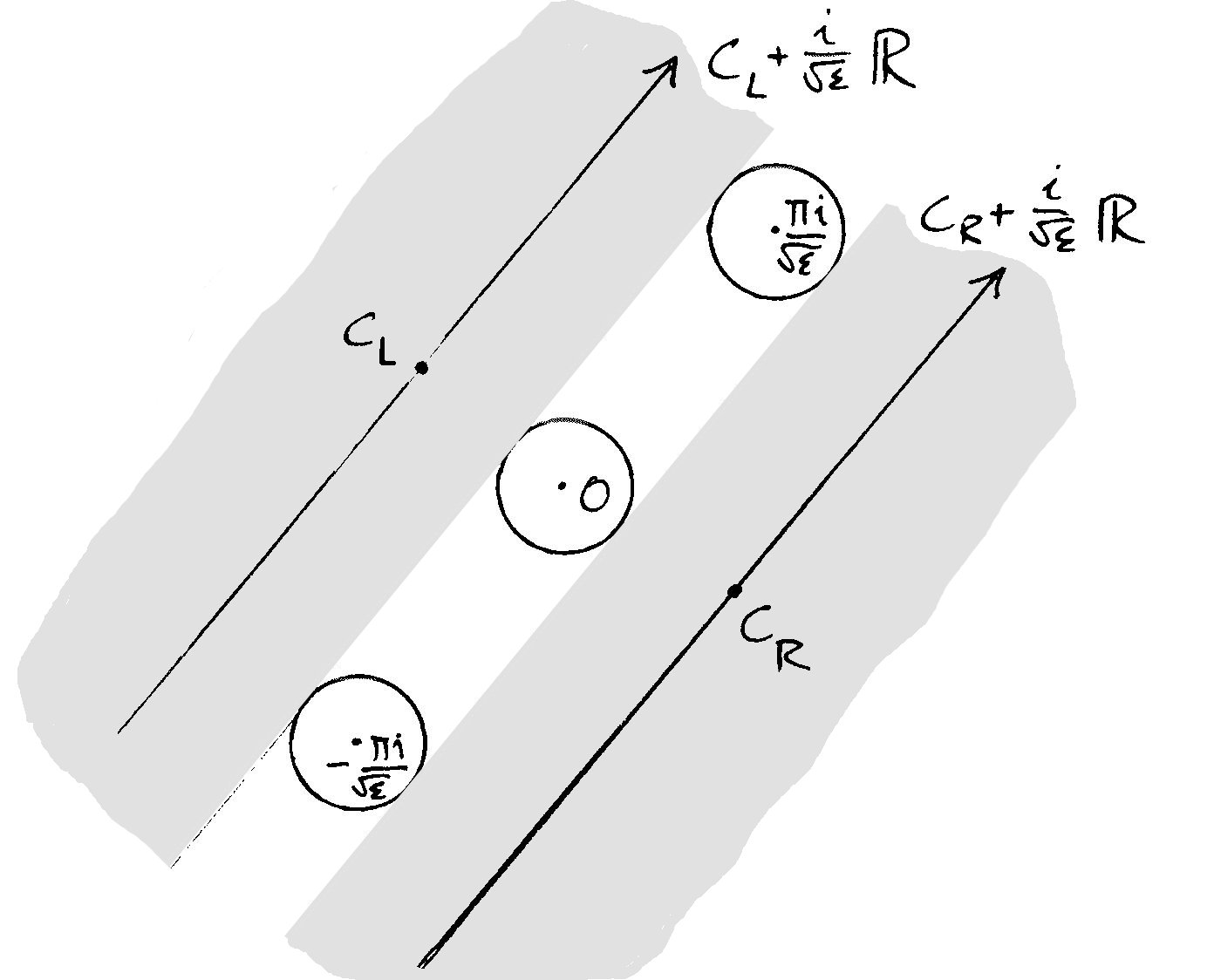}
\caption{The integration paths $c_\bullet+\frac{i}{\sqrt\epsilon}\R$  ($\bullet=L,R$)  in the time $t$-coordinate.}
\label{figure:BL-3a}
\end{figure}

If $f$ is analytic on a neighborhood of $x=\sqrt\epsilon$ (resp. $x=-\sqrt\epsilon$), then
the lifting of $f$ to the time coordinate, $f(x(t,\epsilon))$, is $\frac{\pi i}{\sqrt\epsilon}$-periodic in the half-plane $\Re(e^{i\arg\sqrt\epsilon}\,t)>\Lambda$
(resp. $\Re(e^{i\arg\sqrt\epsilon}t)<-\Lambda\,$) for $\Lambda$ large enough, and can be written as a sum of
its Fourier series expansion:
\begin{align*}
f(x)&=\sum_{n=0}^{+\infty} a_n^R\, e^{-2n\sqrt\epsilon\,t(x)}= \sum_{n=0}^{+\infty} a_n^R\cdot\!\left(\tfrac{x-\sqrt\epsilon}{x+\sqrt\epsilon}\right)^n,\\
\text{resp.}\quad  f(x)&=\sum_{n=0}^{+\infty} a_{n}^L\, e^{2n\sqrt\epsilon\,t(x)}=\sum_{n=0}^{+\infty} a_n^L\cdot\! \left(\tfrac{x+\sqrt\epsilon}{x-\sqrt\epsilon}\right)^n.
\end{align*}
The Borel transform \eqref{eq:BL-B} of $f(x(t,\epsilon))$ on the line $c_R+\frac{i}{\sqrt\epsilon}\R$ (resp. $c_L+\frac{i}{\sqrt\epsilon}\R$) is equal to the formal sum of distributions
\begin{align*}
\Cal B^R[f](\xi,\sqrt\epsilon):=\tfrac{1}{2\pi i}\int_{c_R-\frac{i}{\sqrt\epsilon}\infty}^{c_R+\frac{i}{\sqrt\epsilon}\infty} f(x(t,\epsilon))\,e^{t\xi}dt 
&=\sum_{n=0}^{+\infty} a_n^R\,\delta_{2n\sqrt\epsilon}(\xi),\\
\text{resp.}\quad 
\Cal B^L[f](\xi,\sqrt\epsilon):=\tfrac{1}{2\pi i}\int_{c_L-\frac{i}{\sqrt\epsilon}\infty}^{c_L+\frac{i}{\sqrt\epsilon}\infty} f(x(t,\epsilon))\,e^{t\xi}dt
&=\sum_{n=0}^{+\infty} a_n^L\,\delta_{-2n\sqrt\epsilon}(\xi).
\end{align*}
These transformations were studied by Sternin and Shatalov in \cite{SS}.
One can connect the coefficients $a_n^\bullet$ of these expansions to residues of the unfolded Borel transforms $\Cal B_\alpha^\pm$, $\arg\sqrt\epsilon<\alpha<\arg\sqrt\epsilon+\pi$,
\begin{align*}
 a^R_0=f(\!\sqrt\epsilon), \qquad a_n^R&=2\pi i\,\res_{2n\sqrt\epsilon}\,\Cal B_\alpha^\pm[f],\quad n\in\N_{>0},\\
 a^L_0=f(\!-\sqrt\epsilon),\qquad a_n^L&=2\pi i\,\res_{-2n\sqrt\epsilon}\,\Cal B_\alpha^\pm[f],\quad n\in\N_{>0},
\end{align*}
(the residues of $\Cal B_\alpha^+[f]$ and $\Cal B_\alpha^-[f]$ at the points $\xi\in 2\sqrt\epsilon\,\Z$ are equal).

\begin{remark}
Without providing details, let us remark that one could follow \cite{SS} and apply these Borel transformations $\Cal B^R$ (resp. $\Cal B^L$) to the system \eqref{eq:BL-cm} 
to show the convergence of its unique local analytic solution at $x=\sqrt\epsilon\neq 0$ (resp. $x=-\sqrt\epsilon\neq 0$)
to a Borel sum in direction $\arg\sqrt\epsilon$ of the formal solution $\hat y_0(x)$ of the limit system, when $\sqrt\epsilon\to 0$ radially in a sector not containing any eigenvalue of $M$, as stated in Theorem~\ref{theorem:BL-localsolutions}.
\end{remark}


\section{Solution to the equation \eqref{eq:BL-cm} in the Borel plane}\label{sec:BL-4}

We will use the unfolded Borel transformation ${\Cal B_\alpha^\pm}$ to transform the equation 
\begin{equation*}
\leqno\eqref{eq:BL-cm}: \qquad  (x^2\!-\epsilon)\frac{dy}{dx}=My+f(x,\epsilon,y)
\end{equation*}
to a convolution equation in the Borel plane, and study its solutions there.
We write the function  $\,f\,$ \eqref{eq:BL-f} as
\begin{equation}\label{eq:BL-pcm}
 f(x,\epsilon,y)=\sum_{|l|\geq 0}f_l(x,\epsilon)\,y^l,
\end{equation}
where $y^l:=y_1^{l_1}\cdot\ldots\cdot y_m^{l_m}$ for each multi-index $l=(l_1,\ldots,l_m)\in\N^m$, $|l|=l_1+\ldots+l_m$,
and $f_0(x,\epsilon)=O(x^2\!-\epsilon)$, $f_l(x,\epsilon)=O(x)+O(\epsilon)$, for $|l|=1$.

Let a vector variable $\upsilon=\upsilon(\xi,\sqrt\epsilon)$ 
correspond to the Borel transform ${\Cal B_\alpha^\pm}[y](\xi,\sqrt\epsilon)$, with $\alpha\in\,]\arg\sqrt\epsilon,\arg\sqrt\epsilon+\pi[$ if $\sqrt\epsilon\neq 0$. 
Then the equation \eqref{eq:BL-cm} is transformed to a convolution equation in the Borel plane 
\begin{equation}\label{eq:BL-bpcm}
 \xi\upsilon =M\upsilon+ \sum_{|l|\geq 0} \Cal B_\alpha^\pm[f_l] *\upsilon^{\,*l},
\end{equation}
where 
$\,\upsilon^{\,*l}:=\upsilon_{1}^{\,*l_1}*\ldots *\upsilon_{m}^{\,*l_m}\,$ is the convolution product of components of $\upsilon$, each taken $l_i$-times, the convolutions being done in the direction $\alpha$,
$\Cal B_\alpha^\pm[f_l](\xi,\sqrt\epsilon)$ is a sum of an analytic function
and a multiple of the Dirac distribution $\delta_0(\xi)$.
In Proposition~\ref{proposition:BL-3}, we will find a unique analytic solution $\upsilon^\pm(\xi,\sqrt\epsilon)$ of the convolution equation \eqref{eq:BL-bpcm} as a fixed point of the operator
\begin{equation}\label{eq:BL-operatorG}
 \Cal G^\pm[\upsilon](\xi,\sqrt\epsilon):=\big(\xi I-M\big)^{-1}\! \sum_{|l|\geq 0} \left[\Cal B^\pm[f_l]*\upsilon^{\,*l}\right]\!(\xi,\sqrt\epsilon)
\end{equation}
on a domain $\Omega(\!\sqrt\epsilon)$ in the $\xi$-plane, obtained as union of strips $\Omega_\alpha(\!\sqrt\epsilon)$ 
of continuously varying direction $\alpha$, passing in between the points $-2\sqrt\epsilon$ and $2\sqrt\epsilon$,
that stay away from the eigenvalues of the matrix $M$ (see Figure~\ref{figure:BL-5}).
In general, several ways of choosing such a domain $\Omega(\sqrt\epsilon)$ are possible, depending on its position relative with respect to the eigenvalues of $M$.
Different choices of the domain $\Omega(\sqrt\epsilon)$ will, in general, lead to different solutions $\upsilon^\pm(x,\sqrt\epsilon)$ of \eqref{eq:BL-bpcm}, as shown in Example~\ref{example:BL-2} below.

\begin{figure}[hbt]
\hskip-10pt
\includegraphics[width=1.15\textwidth]{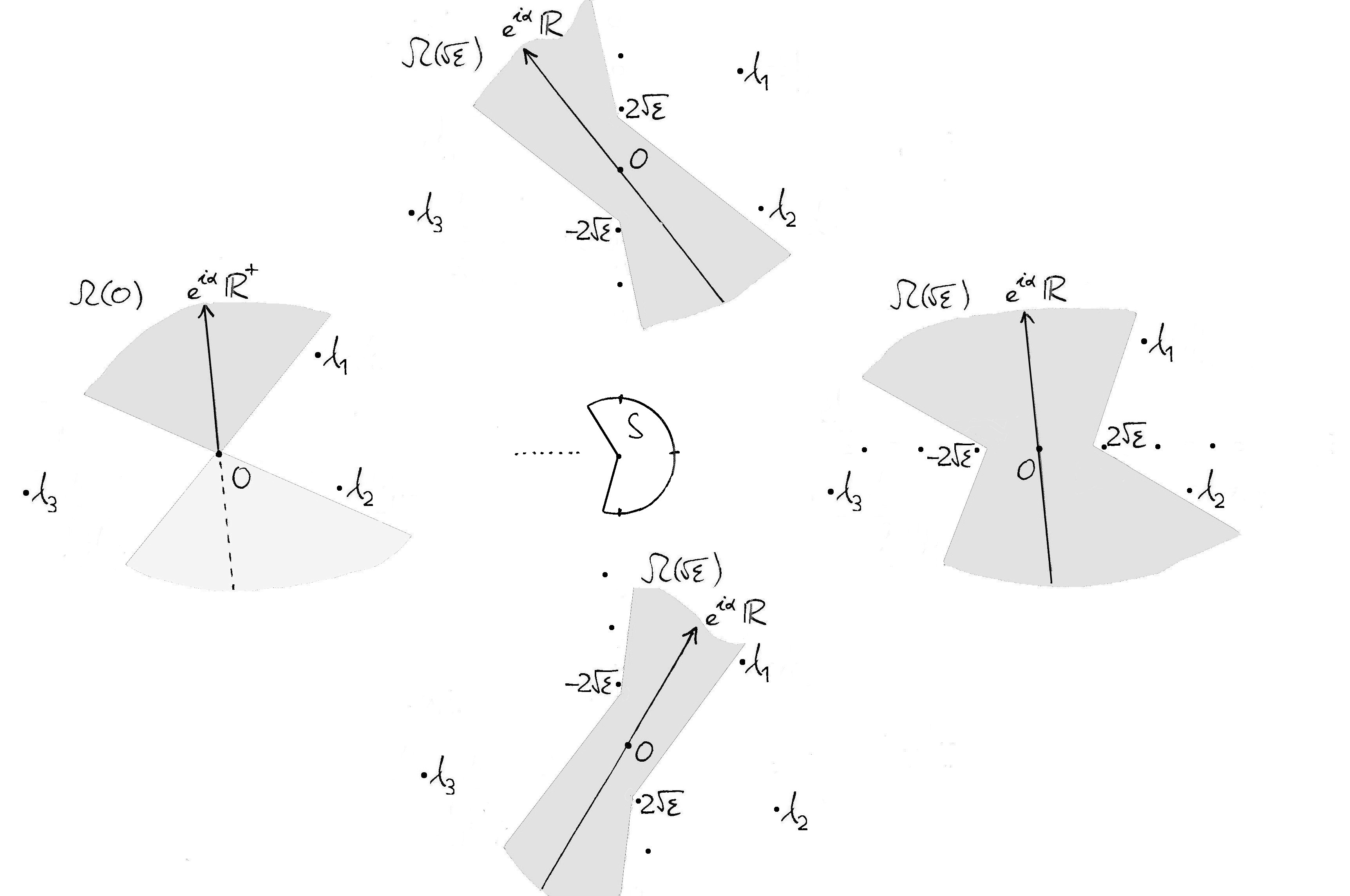}
\caption{The regions $\Omega(\!\sqrt\epsilon)$ and the eigenvalues $\lambda_1,\ldots,\lambda_m$ (here $m=3$) of $M$ in the $\xi$-plane according to $\sqrt\epsilon\in S$, together with integration paths $e^{i\alpha}\R$ of the unfolded Laplace transformation $\Cal L_\alpha$.}
\label{figure:BL-5}
\end{figure}


\begin{definition}[Family of regions $\Omega(\!\sqrt\epsilon)$ in the Borel plane] \label{definition:BL-5}
Let the two directions $\beta_1<\beta_2$ and an arbitrarily small angle $0<\eta<\frac{1}{4}(\beta_2-\beta_1)$
be as in Definition~\ref{definition:BL-X}, 
and let $\rho>0$ be small enough so that for $|\sqrt\epsilon|<\rho$ none of the closed strips
\begin{equation}\label{eq:BL-Omega_alpha}
 \Omega_\alpha(\!\sqrt\epsilon)=\bigcup_{c\in[-\frac{3}{2}\sqrt\epsilon,\frac{3}{2}\sqrt\epsilon]}c+e^{i\alpha}\R,
\end{equation}
with $\alpha\in\,]\beta_1+\eta,\beta_2-\eta[$, contains any eigenvalue of $M_{\pm\sqrt\epsilon}$.
We define a family of regions $\Omega(\!\sqrt\epsilon)$ in the $\xi$-plane depending parametrically on $\sqrt\epsilon\in S$ \eqref{eq:BL-S} as
\begin{equation*}
\Omega(\!\sqrt\epsilon):=\bigcup_{\alpha} \Omega_\alpha(\!\sqrt\epsilon), \qquad\alpha\text{ as in \eqref{eq:BL-alpha}},
\end{equation*}
and denote 
\begin{equation}\label{eq:BL-Omega}
\Omega:=\coprod_{\sqrt\epsilon\in S}\Omega(\!\sqrt\epsilon) 
\end{equation}
their union in the $(\xi,\sqrt\epsilon)$-space.

Let $\rho$, $\eta>0$ be as in the definition of $S$ \eqref{eq:BL-S}, and let $0\leq\Lambda<\frac{\pi\sin\eta}{2\rho}$.
For a vector function $\phi=(\phi_1,\ldots,\phi_m):\Omega\to\C^m$, 
we say that it is \emph{analytic on $\Omega$}, if it is continuous on $\Omega$, analytic on the interior of $\Omega$, and $\phi(\cdot,\sqrt\epsilon)$ is analytic on $\Omega(\sqrt\epsilon)$ for all
$\sqrt\epsilon\in S$.
We define the norms 
\begin{align*}
  |\phi|_{\Omega}^\Lambda&:=\max_{i}\,\sup_{\sqrt\epsilon,\,\alpha} \sup_{\ c\in \Omega_\alpha(\!\sqrt\epsilon)} 
  |\phi_i|_{c+e^{i\alpha}\R}^{\Lambda,\lambda_\alpha-\Lambda}, \\
  \|\phi\|_{\Omega}^\Lambda&:=\max_{i}\,\sup_{\sqrt\epsilon,\,\alpha} \sup_{\ c\in \Omega_\alpha(\!\sqrt\epsilon)} 
  \|\phi_i\|_{c+e^{i\alpha}\R}^{\Lambda,\lambda_\alpha-\Lambda}, 
\end{align*}
where $\sqrt\epsilon\in S$ and $\alpha$ as in \eqref{eq:BL-alpha}, i.e. such that $\Omega_\alpha(\!\sqrt\epsilon)\subset\Omega(\!\sqrt\epsilon)$,
and $\lambda_\alpha(\!\sqrt\epsilon)=-\Re\big(\frac{e^{i\alpha}\pi i}{\sqrt\epsilon}\big)$.
\end{definition}

Then the convolution of two analytic functions $\phi,\psi$ on $\Omega(\!\sqrt\epsilon)$ does not depend on the direction $\alpha$ \eqref{eq:BL-alpha},
and the norms $|\phi*\psi|_{\Omega}^\Lambda$, $\|\phi*\psi\|_{\Omega}^\Lambda$ satisfy the Young's inequalities (Lemma~\ref{lemma:BL-2}):
\begin{align}\label{eq:BL-lemma2aa}
|\phi*\psi|_{\Omega}^{\Lambda} &\leq \min\big\{\,|\phi|_{\Omega}^{\Lambda}\cdot \|\psi\|_{\Omega}^{\Lambda},\ \ \|\phi\|_{\Omega}^{\Lambda}\cdot |\psi|_{\Omega}^{\Lambda}\big\} \\[8pt]
\|\phi*\psi\|_{\Omega}^{\Lambda} &\leq \|\phi\|_{\Omega}^{\Lambda}\cdot \|\psi\|_{\Omega}^{\Lambda}.
\label{eq:BL-lemma2bb}
\end{align}
We extend these relations also to the Dirac distribution with mass at 0 by
setting $|\delta_0|_{\Omega}^{\Lambda}=\|\delta_0\|_{\Omega}^{\Lambda}=1$.

\medskip


\begin{proposition}[Solution to the convolution equation \eqref{eq:BL-bpcm}]\label{proposition:BL-3}
Suppose that the vector function $f(x,\epsilon,y)$ in the equation \eqref{eq:BL-cm} are analytic for 
$$x\in \mathbf{X}(\Lambda_1,\sqrt\epsilon),\quad \sum_{i=1}^m|y_i|<\tfrac{1}{L_1},\quad |\sqrt\epsilon|<\rho_1,\quad\text{for }\ \Lambda_1, L_1,\rho_1>0.$$
Then there exists $\Lambda>\Lambda_1$, $0<\rho\leq\rho_1$, and a constant  $c>0$, such that the operator 
$\Cal G^+: \phi(\xi,\sqrt\epsilon)\mapsto\Cal G^+[\phi](\xi,\sqrt\epsilon)$ \eqref{eq:BL-operatorG} 
is well-defined and contractive on the space
$$\{\phi:\Omega\to\C^m\mid \phi \text{ is analytic on } \Omega,\ \|\phi\|_\Omega^{\Lambda}\leq c,\ |\phi|_\Omega^{\Lambda}<+\infty\}$$
with respect to both the $\|\cdot\|_\Omega^\Lambda$-norm and the $|\cdot|_\Omega^\Lambda$-norm. 
\textbf{Hence the equation $\Cal G^+[\upsilon^+]=\upsilon^+$ possesses a unique analytic solution 
$\upsilon^+(\xi,\sqrt\epsilon)$ on $\Omega$,  satisfying 
$\|\upsilon^+\|_\Omega^\Lambda\leq c$ and $|\upsilon^+|_\Omega^\Lambda<+\infty$.}
Similarly, the vector function $\upsilon^-(\xi,\sqrt\epsilon):=e^{\frac{\xi\pi i}{\sqrt\epsilon}}\cdot\upsilon^+(\xi,\sqrt\epsilon)$ is a unique analytic solution  of the equation 
$\Cal G^-[\upsilon^-]=\upsilon^-$ on $\Omega$.
\end{proposition}

While for $\epsilon=0$ the germ of the solutions $\upsilon^+(\cdot,0)$ at $\xi=0$ equals to the Borel transform of the unique formal solution $\hat y_0(x)=\hat y(x,0)$ \eqref{eq:BL-formalsolution}, and therefore it is independent 
of the domain $\Omega(0)$,  
this is no longer true when one unfolds. The reason is that the convolution is no longer defined locally, but involves integration on a whole line $c+e^{i\alpha}\R$.
As the following example shows, the analytic solutions $\upsilon^\pm$ of \eqref{eq:BL-bpcm} therefore depend in general on the position of the region $\Omega(\sqrt\epsilon)$ with respect to the eigenvalues of $M$, and are not analytic extensions one of another. 
We'll see later (Corollary~\ref{corollary:BL-3}) that their difference is  exponentially small in $\sqrt\epsilon$ for each fixed small $\xi$.

\begin{example}\label{example:BL-2}
Let $u$ satisfy
\begin{equation}\label{eq:BL-example0}
 (x^2\!-\epsilon)\frac{du}{dx}=u+(x^2\!-\epsilon),
\end{equation}
and let $y=(x^2\!-\epsilon)u\,$; it satisfies
\begin{equation}\label{eq:BL-example}
 (x^2\!-\epsilon)\frac{dy}{dx}=y+2xy+(x^2\!-\epsilon)^2.
\end{equation}
The Borel transform of the equation \eqref{eq:BL-example0} is
$$\xi\phi_\alpha^\pm=\phi_\alpha^\pm+\xi\chi_\alpha^\pm,\quad \phi_\alpha^\pm=\Cal B_\alpha^\pm[u]$$
therefore $\phi_\alpha^\pm(\xi,\sqrt\epsilon)=\frac{\xi}{\xi-1}\chi_\alpha^\pm(\xi,\sqrt\epsilon)$, which is independent of the direction $\alpha$. 
This is no longer true for the solution $\upsilon_\alpha^\pm=\phi_\alpha^\pm*\Cal B_\alpha^\pm[x^2\!-\epsilon]$ of
the Borel transform of the equation \eqref{eq:BL-example}
$$\xi\upsilon_\alpha^\pm=\upsilon_\alpha^\pm+2\Cal B^\pm[x]\!*\upsilon_\alpha^\pm+\chi_\alpha^\pm\cdot (\xi^3\!-4\epsilon\xi).$$
If, for instance, $\Im(\sqrt\epsilon)<0$, and $\arg\sqrt\epsilon<\alpha_1<0<\alpha_2<\arg\sqrt\epsilon+\pi$,
then the strips $\Omega_{\alpha_1}(\!\sqrt\epsilon)$, $\Omega_{\alpha_2}(\!\sqrt\epsilon)$ \eqref{eq:BL-Omega_alpha} in directions $\alpha_1$, $\alpha_2$, are separated by the point $\xi=1$,
and one easily calculates that for $\xi\in\Omega_{\alpha_1}(\!\sqrt\epsilon)\cap\Omega_{\alpha_2}(\!\sqrt\epsilon)$
$$\upsilon_{\alpha_1}^\pm(\xi,\sqrt\epsilon)-\upsilon_{\alpha_2}^\pm(\xi,\sqrt\epsilon)=(\xi-1)\,\chi_\alpha^\pm(1,\sqrt\epsilon)\,\chi_\alpha^\pm(\xi-1,\sqrt\epsilon),$$
i.e. the two solutions $\upsilon_{\alpha_1}^\pm$, $\upsilon_{\alpha_2}^\pm$ differ near $\xi=0$ by a term that is exponentially flat in $\sqrt\epsilon$.
\end{example}

To prove Proposition~\ref{proposition:BL-3} we will make use of the following technical lemmas which will allow us to estimate the norms of $\Cal G^+[\phi]$.

\begin{lemma}\label{lemma:BL-3b}
There exists a constant  $C=C(\Lambda_1,\eta)>0$ such that, 
if $h\in\Cal O(\mathbf{X}(\Lambda_1,\sqrt\epsilon))$, $|\sqrt\epsilon|<\rho$, and $\Lambda_1<\Lambda<\frac{\pi\sin\eta}{2\rho}$ (where $\eta,\rho>0$ are as in \eqref{eq:BL-alpha}, \eqref{eq:BL-S}), then
$$|\Cal B_\alpha^+[(x^2\!-\epsilon)h]|_{\Omega}^\Lambda\leq C\rho\sup_{x\in \mathbf{X}(\Lambda_1,\sqrt\epsilon)} |h(x)|.$$ 
\end{lemma}

\begin{proof}
Essentially, one needs to estimate the integral $\int_{\Re(e^{i\alpha}t)=\Lambda}\big|\frac{x-\sqrt\epsilon}{x+\sqrt\epsilon}\big|^c\,d|x|$, 
with $c\in[-\frac{3}{4},\frac{3}{4}]$ and $\alpha\in\,]\arg\sqrt\epsilon+\eta,\arg\sqrt\epsilon+\pi-\eta[$.
\qed\end{proof}

\begin{lemma}\label{lemma:BL-widetildex}
 Let $\phi$ be an analytic function on $\Omega$ with a finite $|\phi|_{\Omega}^\Lambda$ (resp. $\|\phi\|_{\Omega}^\Lambda$).
Then its convolution with the distribution $\Cal B^+[x]$ (Definition~\ref{def:widetildex}) is again an analytic function on $\Omega$ whose norm satisfies
\begin{align}\label{eq:BL-widtildex1}
|\Cal B^+[x]*\phi|_{\Omega}^\Lambda   &\leq |\phi|_{\Omega}^\Lambda\cdot   \left(\rho+\|\chi_\alpha^+\|_{\Omega_L}^\Lambda \right), \\
\text{resp.}\quad \|\Cal B^+[x]*\phi\|_{\Omega}^\Lambda &\leq \|\phi\|_{\Omega}^\Lambda\cdot \left(\rho+\|\chi_\alpha^+\|_{\Omega_L}^\Lambda \right),
\label{eq:BL-widtildex2}
\end{align}
where $\chi_\alpha^+$ is given in \eqref{eq:BL-chi}, $\rho$ is the radius of $S$, and
\begin{equation}\label{eq:BL-Omega_L}
 \Omega_L(\!\sqrt\epsilon)=\Omega(\!\sqrt\epsilon)\cap\big(\Omega(\!\sqrt\epsilon)\!-\!2\sqrt\epsilon\,\big),\qquad\text{for each }\ \sqrt\epsilon\in S.
\end{equation}
\end{lemma}

\begin{proof}
It follows from Definition~\ref{def:widetildex} and $2\sqrt\epsilon$-periodicity of  $\chi_\alpha^+$.
\qed\end{proof}

\begin{lemma}\label{lemma:BL-3a}
If $\phi,\psi:\Omega\to\C^m$ are analytic vector functions such that $\|\phi\|_\Omega^{\Lambda},$ 
$\|\psi\|_\Omega^{\Lambda}\leq a$, then for any multi-index $l\in\N^m$, $|l|\geq1$,
$$|\phi^{*l}-\psi^{*l}|_\Omega^{\Lambda}\leq |l|\cdot a^{|l|-1}\cdot |\phi-\psi|_\Omega^{\Lambda}.$$
The same holds for the $\|\cdot\|_\Omega^{\Lambda}$-norm as well.
\end{lemma}
\begin{proof}
Writing $\phi^{*l}=\phi_{i_1}\!*\ldots*\phi_{i_{|l|}}$, $i_j\in\{1,\ldots,m\}$, we have
\begin{equation*}
\begin{split}
\phi^{*l}-\psi^{*l}=&\,(\phi_{i_1}\!-\!\psi_{i_1})*\phi_{i_2}\!*\ldots*\phi_{i_{|l|}}\,+\,\psi_{i_1}\!*(\phi_{i_2}\!-\!\psi_{i_2})*\phi_{i_3}\!*\ldots*\phi_{i_{|l|}}\,+\, \\
& \ldots\,+\, \psi_{i_1}\!*\ldots*\psi_{i_{|l|-1}}\!*(\phi_{i_{|l|}}\!-\!\psi_{i_{|l|}}).
\end{split}
\end{equation*}
The statement now follows from the convolution inequalities \eqref{eq:BL-lemma2aa} (resp. \eqref{eq:BL-lemma2bb}).
\qed\end{proof}

\begin{proof}[Proof of Proposition~\ref{proposition:BL-3}.]
If $L>m\cdot L_1$, then there exists $K>0$ such that for each multi-index  $l\in\N^m$, $|f_l(x,\epsilon)| \leq K\cdot\textstyle\binom{|l|}{l}\, L^{|l|},$
where for $y\in\C^m$, $|y|=\sum_{i=1}^m|y_i|$, and 
where $\binom{|l|}{l}$ are the multinomial coefficients given by $\,(y_1+\ldots+y_m)^k=\sum_{|l|=k}\binom{|l|}{l}\,y^l$, satisfying
$$\sum_{|l|=k} \textstyle\binom{|l|}{l}=m^k.$$
It follows from Lemma~\ref{lemma:BL-3b}, Lemma~\ref{lemma:BL-widetildex} and Lemma~\ref{lemma:BL-1},
that if $\Lambda>\Lambda_1$, then 
\begin{align*}
\|\Cal B^+[f_0]\|_\Omega^\Lambda &\leq \rho K_0, \\ 
\|\Cal B^+[f_l]\|_\Omega^\Lambda &\leq \rho K_1L,\, \text{ if }|l|=1, \\
\|\Cal B^+[f_l]\|_\Omega^\Lambda &\leq K_2\textstyle\binom{|l|}{l}L^{|l|},\, \ |l|\geq 2,
\end{align*}
for some $K_0,K_1,K_2>0$. 
Moreover, if we take $\Lambda$ sufficiently large and $\rho$ sufficiently small, then we can make the constant $K_0$ arbitrarily small. 
Let 
$$\sigma=\max_{(\xi,\sqrt\epsilon)\in\Omega} \big|\left(I\xi-M\right)^{-1}\!\cdot\!\left(\begin{smallmatrix}1\\[-3pt] \vdots \\[3pt] 1\end{smallmatrix}\right)\big|,$$  
then $\sigma<+\infty$ if the radius $\rho$ of $S$ is small, and 
suppose that $\rho K_0$ is small enough, so that there is $c>0$ satisfying 
\eqref{eq:BL-xx1} and \eqref{eq:BL-xx2} below.
Then if $\|\phi\|_{\Omega}^\Lambda\leq c$
\begin{equation}\label{eq:BL-xx1}
\|\Cal G^+[\phi]\|_{\Omega}^\Lambda
\leq \sigma\cdot\Big(\rho K_0+ \rho K_1cmL +K_2\sum_{k=2}^{+\infty} (cmL)^k\Big) \leq c,
\end{equation}
using \eqref{eq:BL-lemma2bb}, and similarly,  $|\Cal G^+[\phi]|_{\Omega}^\Lambda\leq\max\{c,\, |\phi|_{\Omega}^\Lambda\}$ if $\|\phi\|_{\Omega}^\Lambda\leq c$.
And if $\,\|\phi\|_{\Omega}^\Lambda,\|\psi\|_{\Omega}^\Lambda\leq c$, then
\begin{equation}\label{eq:BL-xx2}
\frac{|\Cal G^+[\phi]-\Cal G^+[\psi]|_{\Omega}^\Lambda}{|\phi-\psi|_{\Omega}^\Lambda} 
\leq \sigma\cdot\Big(\rho K_1 mL+K_2mL\sum_{k=2}^{+\infty} k(cmL)^{k-1}\Big) \leq\frac{1}{2},
\end{equation}
using  Lemma~\ref{lemma:BL-3a} and the convolution inequality \eqref{eq:BL-lemma2aa}. The same holds for the $\|\cdot\|_{\Omega}^\Lambda$-norm.
Hence the operator $\Cal G^+$ is $|\cdot|_{\Omega}^\Lambda$-contractive, and the sequence $\big(\Cal G^+\big)^n[0]$ converges, as $n\to+\infty$, $|\cdot|_{\Omega}^\Lambda$-uniformly   
to an analytic function $\upsilon^+$ satisfying $\Cal G^+\big[\upsilon^+]=\upsilon^+$.

From \eqref{eq:BL-Bpm} it follows that $\Cal B^-[f_l]=e^{\frac{\xi\pi i}{\sqrt\epsilon}}\cdot \Cal B^+[f_l]$, 
hence $\Cal G^-[\upsilon^-]=\Cal G^-[e^{\frac{\xi\pi i}{\sqrt\epsilon}}\upsilon^+]\linebreak[0]=e^{\frac{\xi\pi i}{\sqrt\epsilon}} \cdot\Cal G^+[\upsilon^+]=e^{\frac{\xi\pi i}{\sqrt\epsilon}} \cdot\upsilon^+=\upsilon^-$ is a fixed point of $\Cal G^-$.
\qed\end{proof}

\begin{proposition}[Poincaré case]\label{lemma:BL-4}
If the spectrum of $M$ is of Poincaré type, i.e. if it is contained in a sector of opening $<\pi$, then, for small $\sqrt\epsilon$, the region 
$\Omega(\!\sqrt\epsilon)$ may be chosen so that it has all the eigenvalues of $M$ on the same side---let's say the side where $2\sqrt\epsilon$ is.
In such case, let
$\Omega_1(\!\sqrt\epsilon)$ be the extension of $\Omega(\!\sqrt\epsilon)$ to the whole region on the opposite side (see Figure~\ref{figure:BL-7}).
The solutions $\upsilon^\pm(\xi,\sqrt\epsilon)$ of Proposition~\ref{proposition:BL-3} can be analytically extended to $\Omega_1(\!\sqrt\epsilon)\sminus (-2\sqrt\epsilon)\N_{>0}$ 
with at most simple poles at the points  $-2\sqrt\epsilon\N_{>0}$. The function $\frac{\upsilon^\pm}{\chi_\alpha^\pm}$ 
is analytic in $\Omega_1$ and has at most exponential growth $<Ce^{\Lambda|\xi|}$  for some $\Lambda,C>0$ independent of $\sqrt\epsilon$.

\begin{figure}[t]
\centering
\includegraphics[width=\textwidth]{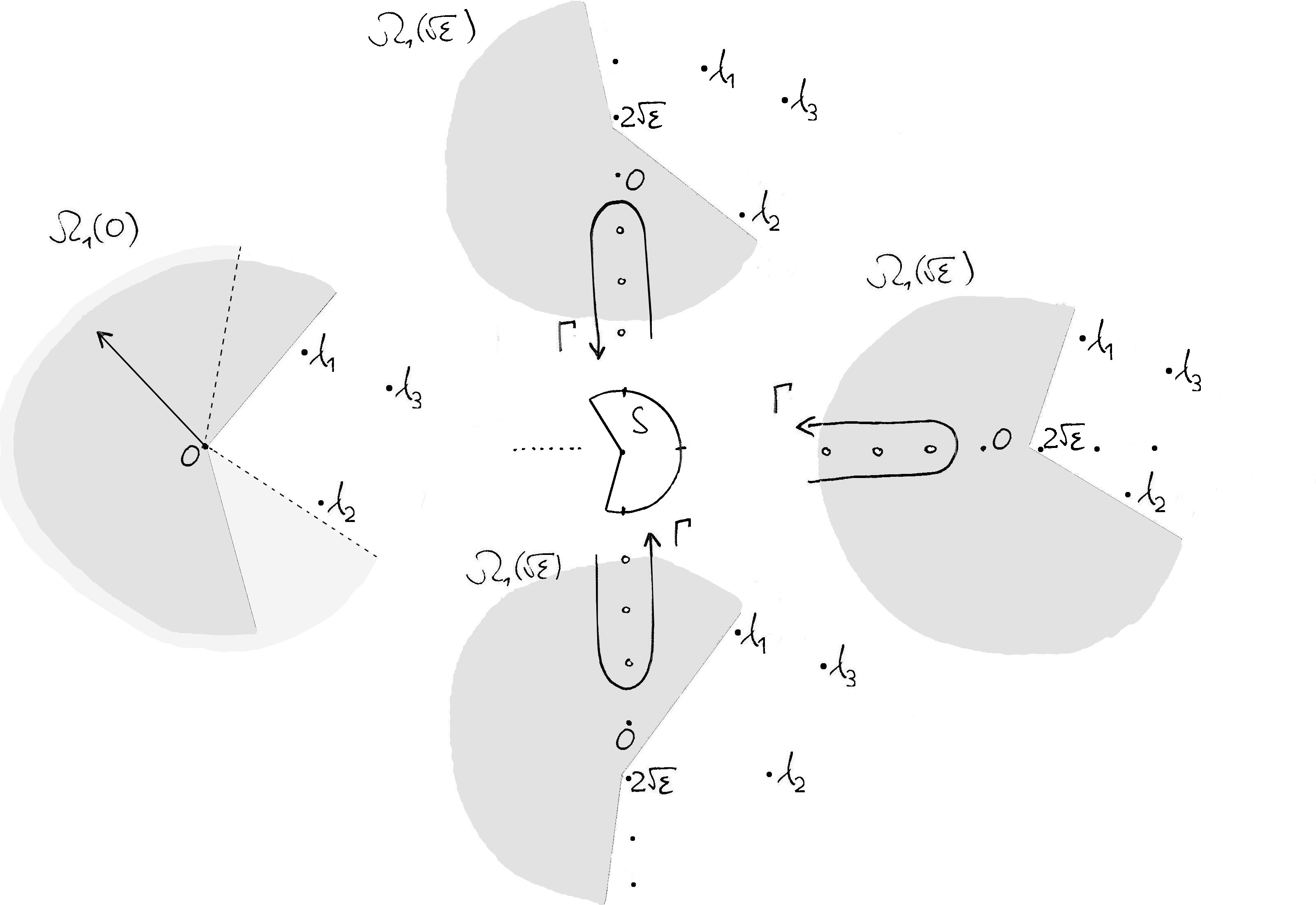}
\caption{The extended regions $\Omega_1(\!\sqrt\epsilon)$ in the Borel plane, together with the modified integration path $\Gamma$ of the Laplace transform (compare with Figure~\ref{figure:BL-5}).
The limit region $\Omega_1(0):=\bigcup_{\sqrt\epsilon\in S}\bigcap_{\nu\to 0+}\Omega_1(\nu\sqrt\epsilon)\sminus(-2\sqrt\epsilon)\N_{>0}$ is composed of two sectors connected at the origin; 
the solution $\upsilon^+(\xi,0)$ vanishes on the lower sector, while the solution $\upsilon^-(\xi,0)$ vanishes on the upper one.}
\label{figure:BL-7}
\end{figure}
\end{proposition}

\begin{proof}
As in the proof of Proposition~\ref{proposition:BL-3}, the solution $\upsilon^+$ is constructed as a limit of the iterative sequence of functions $\,(\Cal G^+)^n[0]$, $n\to+\infty$. 
We will show by induction that for each $n$, the function $\,(\Cal G^+)^n[0]$ is analytic on 
$\Omega_1 \sminus\{\xi\in\!-2\sqrt\epsilon\N_{>0}\}$ and has at most simple poles at the points $\xi\in-2\sqrt\epsilon\N_{>0}$,
and that the sequence converges uniformly to $\upsilon^+$ with respect to the norm
\begin{equation}\label{eq:BL-boxnorm}
\boxnorm\phi_{\Omega_1}^\Lambda:=\sup_{(\xi,\sqrt\epsilon)\in\Omega_1} |\tfrac{\phi}{\,\chi^+\!}(\xi,\sqrt\epsilon)|\,e^{-\Lambda|\xi|}.
\end{equation}
To do so we will introduce another norm $\fatnorm{\hskip-3pt|\,\cdot\,}_{\Omega_1}^\Lambda$, defined in \eqref{eq:BL-fatnorm} below, such that the two norms satisfy
convolution inequalities similar to those satisfied by \hbox{$|\cdot|_{\Omega}^\Lambda$} and  \hbox{$\|\cdot\|_{\Omega}^\Lambda$} (Lemma~\ref{lemma:BL-6} below). 
Then one can simply replicate the proof of Proposition~\ref{proposition:BL-3}
with the norm $\boxnorm{\cdot}_{\Omega_1}^\Lambda$ in place of $|\cdot|_{\Omega}^\Lambda$ 
and the norm $\fatnorm{\hskip-3pt|\,\cdot\,}_{\Omega_1}^\Lambda$ in place of $\|\cdot\|_{\Omega}^\Lambda$.

\smallskip
\begin{figure}[t]
\centering
\subfigure[$\Im\big(\frac{\xi}{\sqrt\epsilon}\big)>0$]{\includegraphics[width=0.49\textwidth]{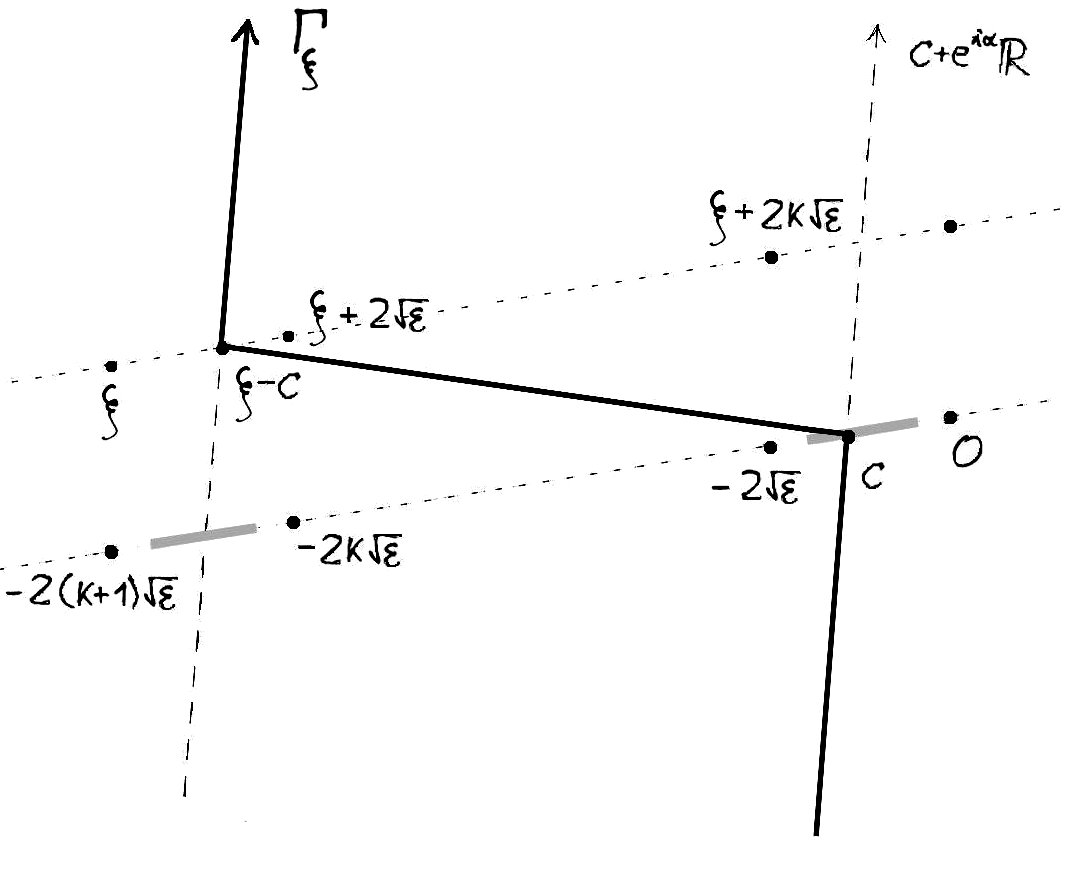}} 
\subfigure[$\Im\big(\frac{\xi}{\sqrt\epsilon}\big)<0$]{\includegraphics[width=0.49\textwidth]{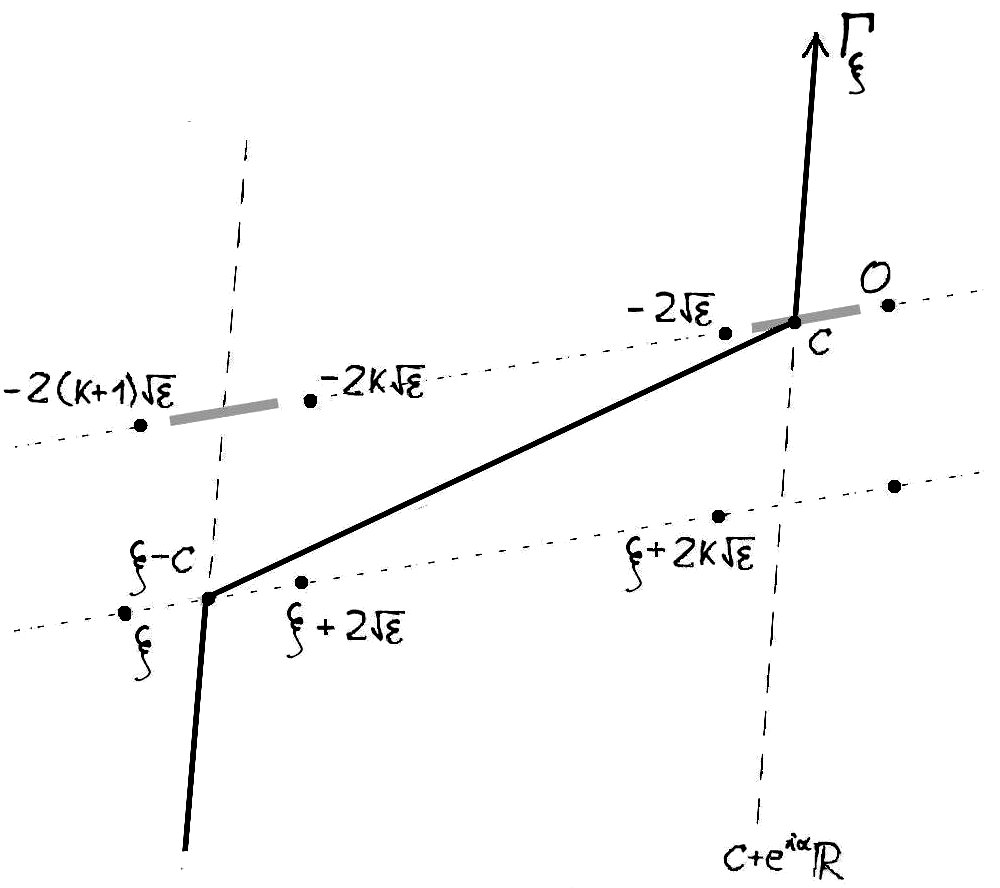}}
\caption{The integration path $\Gamma_\xi$ of convolution $(\phi*\psi)(\xi)$, \, $\xi\in\Omega_1(\!\sqrt\epsilon)$.}
\label{figure:BL-8}
\end{figure}

Let us first show that if $\phi,\psi$ are two functions analytic on $\Omega_1(\!\sqrt\epsilon)\sminus(-2\sqrt\epsilon\N_{>0})$, 
then so is their convolution  $\phi*\psi$.
If $\xi\in\Omega_1(\!\sqrt\epsilon)\sminus\sqrt\epsilon\,\R$, then the analytic continuation of $\phi*\psi$ at the point $\xi$ 
is given by the integral
$$(\phi*\psi)(\xi)=\int_{\Gamma_\xi}\phi(s)\,\psi(\xi-s)\,ds$$
with $\Gamma_\xi$ a symmetric path with respect to the point $\frac{\xi}{2}$ passing through the segments $[-\frac{3}{2}\sqrt\epsilon,\frac{3}{2}\sqrt\epsilon]$ and 
$[\xi-\frac{3}{2}\sqrt\epsilon,\xi+\frac{3}{2}\sqrt\epsilon]$, as in Figure~\ref{figure:BL-8}.
Note that when $\xi$ approaches a point on 
$]\!-\!\infty\sqrt\epsilon,-2\sqrt\epsilon[\,\smallsetminus\,\big(\!-\!2\sqrt\epsilon\mathbb{N}_{>0}\big)$
from one side or another, the values of the two integrals are identical, 
since both paths $\Gamma_\xi$ pass in between the same singularities.

Suppose now that $\phi,\psi$ have at most simple poles at the points  $-2\sqrt\epsilon\N_{>0}$.
If $\xi$ is in $\Omega(\!\sqrt\epsilon)\cup2\Omega_L(\!\sqrt\epsilon)$ ($\Omega_L$ is defined in \eqref{eq:BL-Omega_L}), then $\Gamma_\xi=c+e^{i\alpha\R}$ for some 
$c\in[-\frac{3}{2}\sqrt\epsilon,-\frac{1}{2}\sqrt\epsilon]\subset\Omega_L(\!\sqrt\epsilon)$. 
Else $\xi\in2\Omega_L(\!\sqrt\epsilon)-2k\sqrt\epsilon$ for some $k\in\N_{>0}$, and one can express the convolution as
\begin{align}
&(\phi*\psi)(\xi)\,
=\kern-.3em\int\displaylimits_{c-2k\sqrt\epsilon+e^{i\alpha}\R} \kern-2em\phi(s)\,\psi(\xi-s)\,ds\,+\,2\pi i\sum_{j=1}^{k}\res_{-2j\sqrt\epsilon}\,\phi\cdot\psi(\xi+2j\sqrt\epsilon)\nonumber \\
&=\kern-.3em\int\displaylimits_{c+e^{i\alpha}\R} \kern-1em\phi(t-2k\sqrt\epsilon)\,\psi(\xi_0-t)\,dt\,-\,2\sqrt\epsilon \sum_{j=0}^{k-1}\tfrac{\phi}{\,\chi^+\!}(-2(k-j)\sqrt\epsilon)
\cdot\psi(\xi_0-2j\sqrt\epsilon), \label{eq:BL-303}
\end{align}
where $c\in[-\frac{3}{2}\sqrt\epsilon,-\frac{1}{2}\sqrt\epsilon]\subset\Omega_L(\!\sqrt\epsilon)$ and $\xi_0=\xi+2k\sqrt\epsilon\in c+\Omega_L(\!\sqrt\epsilon)$, i.e. $\xi-s\in\Omega_L(\!\sqrt\epsilon)$, 
see Figure ~\ref{figure:BL-8}.
We will use this formula to obtain an estimate for the norm $\boxnorm{\phi*\psi}_{\Omega_1}^\Lambda$, $\Lambda\geq0$.
Since $|\tfrac{1}{\,\chi^+(\xi,\sqrt\epsilon)}|\leq(1+|e^{\frac{s\pi i}{\sqrt\epsilon}}|)(1+|e^{\frac{(\xi-s)\pi i}{\sqrt\epsilon}}|)$, cf. \eqref{eq:BL-300},
we have
\begin{align}
|\tfrac{\phi\,*\,\psi}{\,\chi^+\!}(\xi)|\,e^{-\Lambda|\xi|}\leq &
\sup_{s\in\Omega_L(\!\sqrt\epsilon)-2k\sqrt\epsilon}|\phi(s)|(1+|e^{\frac{s\pi i}{\sqrt\epsilon}}|)\,e^{-\Lambda|s|}\cdot \|\psi\|_{\Omega(\!\sqrt\epsilon)}^\Lambda \nonumber \\
&+2\sqrt{|\epsilon|}\cdot\boxnorm\phi_{\Omega_1}^\Lambda \cdot \sum_{j=0}^{k-1} |\tfrac{\psi}{\,\chi^+\!}(\xi_0-2j\sqrt\epsilon)|\,e^{-\Lambda|\xi_0-2j\sqrt\epsilon|}, \label{eq:BL-301}
\end{align}
due to the $2\sqrt\epsilon$-periodicity of $\chi^+$.

Let $\mu\geq 1$ be such that
\begin{equation}\label{eq:BL-mu}
1+|e^{\frac{s\pi i}{\sqrt\epsilon}}|\leq\mu\,|\tfrac{1}{\,\chi^+(s,\sqrt\epsilon)}| \quad \text{for all }\ s\in\Omega_L(\!\sqrt\epsilon)
\end{equation}
and define
\begin{equation}\label{eq:BL-fatnorm}
\fatnorm\psi_{\Omega_1}^\Lambda:=\mu\,\|\psi\|_{\Omega}^\Lambda
+\!\!\!\sup_{\substack{\sqrt\epsilon\in S\\ \xi\in2\Omega_L(\!\sqrt\epsilon)}}\!\!\! 2\sqrt{|\epsilon|}\sum_{k=0}^{+\infty} |\tfrac{\psi}{\,\chi^+\!}(\xi-2k\sqrt\epsilon)|\,e^{-\Lambda|\xi-2k\sqrt\epsilon|}.
\end{equation}
Then \eqref{eq:BL-301} implies that
$$\boxnorm{\phi*\psi}_{\Omega_1}^\Lambda\leq \boxnorm{\phi}_{\Omega_1}^\Lambda\cdot\fatnorm\psi_{\Omega_1}^\Lambda.$$

Note that by \textit{4)} of Proposition~\ref{proposition:BL-4}, if $\frac{f(x,\epsilon)}{x^2-\epsilon}$ is analytic on $\{|x^2-\epsilon|<r^2\}\times\{|\epsilon|<\rho^2\}$ for some $r>2\rho>0$, then for any
$\Lambda>\frac{1}{r-2\rho}$
$$\boxnorm{\Cal B_\alpha[f]}_{\Omega_1}^\Lambda<+\infty,\qquad \fatnorm{\Cal B_\alpha^+[f]}_{\Omega_1}^\Lambda<+\infty,$$
and one can see that $\fatnorm{\Cal B_\alpha[f]}_{\Omega_1}^\Lambda$ can be made arbitrarily small taking $\Lambda$ sufficiently large (cf. Lemma~\ref{lemma:BL-1}).
\qed\end{proof}

\begin{lemma}\label{lemma:BL-6}
Let $\frac{\phi}{\,\chi^+\!},\,\frac{\psi}{\,\chi^+\!}$ be analytic functions on $\Omega_1$ such that $\frac{\phi}{\,\chi^+\!}(0,\sqrt\epsilon)=\frac{\psi}{\,\chi^+\!}(0,\sqrt\epsilon)=0$. Then
\begin{align*}
\boxnorm{\phi*\psi}_{\Omega_1}^\Lambda &\leq \boxnorm\phi_{\Omega_1}^\Lambda\cdot\fatnorm\psi_{\Omega_1}^\Lambda,\\
\fatnorm{\phi*\psi}_{\Omega_1}^\Lambda &\leq \fatnorm\phi_{\Omega_1}^\Lambda\cdot\fatnorm\psi_{\Omega_1}^\Lambda.
\end{align*}
\end{lemma}

\begin{proof}
The first inequality is given in the proof of Lemma~\ref{lemma:BL-4}. We need to prove the second one. By definition
\begin{equation*}
\fatnorm{\phi*\psi}_{\Omega_1}^\Lambda=\mu\,\|\phi*\psi\|_{\Omega}^\Lambda
+\!\!\!\sup_{\substack{\sqrt\epsilon\in S\\ \xi\in2\Omega_L(\!\sqrt\epsilon)}}\!\!\! 2\sqrt{|\epsilon|}\sum_{k=0}^{+\infty} |\tfrac{\phi*\psi}{\,\chi^+\!}(\xi-2k\sqrt\epsilon)|\,e^{-\Lambda|\xi-2k\sqrt\epsilon|}.
\end{equation*}
The first term is smaller than
$$\mu\,\|\phi\|_{\Omega}^\Lambda\,\|\psi\|_{\Omega}^\Lambda\leq\mu^2\,\|\phi\|_{\Omega}^\Lambda\,\|\psi\|_{\Omega}^\Lambda
\qquad\text{since }\ \mu\geq 1.$$
For the second term, using \eqref{eq:BL-303}, \eqref{eq:BL-mu} and $2\sqrt\epsilon$-periodicity of $\chi^+$, we have
\begin{align*}
&\sum_{k=0}^{+\infty} |\tfrac{\phi*\psi}{\,\chi^+\!}(\xi\!-\!2k\sqrt\epsilon)|\,e^{-\Lambda|\xi-2k\sqrt\epsilon|}\\
&\leq\kern-.5em \int\displaylimits_{c+e^{i\alpha}\R}  \kern-1em \mu\sum_{k=0}^{+\infty} |\tfrac{\phi}{\,\chi^+\!}(t\!-\!2k\sqrt\epsilon)|\,e^{-\Lambda|t-2k\sqrt\epsilon|}
		\cdot\!\Big(|\psi(\xi\!-\!t)|\,(1+|e^{\frac{(\xi-t)\pi i}{\sqrt\epsilon}}|)\,e^{-\Lambda|\xi-t|}\Big)\,d|t| \\[-4pt]
	&\quad	 + 2\sqrt{|\epsilon|}\sum_{k=0}^{+\infty}\sum_{j=1}^{k} |\tfrac{\phi}{\,\chi^+\!}(\!-2j\sqrt\epsilon)|\,e^{-\Lambda|2j\sqrt\epsilon|} 
		\cdot|\tfrac{\psi}{\,\chi^+\!}(\xi\!-\!2(k\!-\!j)\sqrt\epsilon)|\,e^{-\Lambda|\xi-2(k-j)\sqrt\epsilon|}\\
&\leq \sup_{\xi\in2\Omega_L(\!\sqrt\epsilon)}\,\sum_{k=0}^{+\infty} |\tfrac{\phi}{\,\chi^+\!}(\xi-2k\sqrt\epsilon)|\,e^{-\Lambda|\xi-2k\sqrt\epsilon|}\,\cdot\\[-8pt]
	&\qquad\qquad\qquad\quad\cdot\Big(\mu\,\|\psi\|_{\Omega}^\Lambda 
		+ \sup_{\xi\in2\Omega_L(\!\sqrt\epsilon)}2\sqrt{|\epsilon|}\sum_{j=0}^{+\infty} |\tfrac{\psi}{\,\chi^+\!}(\xi-2j\sqrt\epsilon)|\,e^{-\Lambda|\xi-2j\sqrt\epsilon|}\Big).
\end{align*}
\qed\end{proof}



\subsubsection{Proof of Theorem~\ref{theorem:BL-2}.}
\textit{i)} Let $\upsilon^\pm(\xi,\sqrt\epsilon)$ be the solution of the convolution equation \eqref{eq:BL-bpcm} on $\Omega$, provided by Proposition~\ref{proposition:BL-3}, with bounded
$\|\cdot\|_{\Omega}^\Lambda$-norm. Its Laplace transform
\begin{equation}\label{eq:BL-ypm}
\Cal L [\upsilon^\pm](x,\sqrt\epsilon)=\int_{-\infty e^{i\alpha}}^{+\infty e^{i\alpha}}\upsilon^\pm(\xi,\sqrt\epsilon)\, e^{-t(x,\epsilon)\xi}\,d\xi,
\end{equation}
where $\alpha$ can vary as in \eqref{eq:BL-alpha}, is a solution of \eqref{eq:BL-pcm} defined for 
$t(x,\epsilon)$ in the domain $T^\pm(\!\sqrt\epsilon)=\bigcup_{\alpha}\mathbf{T}_\alpha^\pm(\Lambda,\sqrt\epsilon)$, (Figure~\ref{figure:BL-6}). 
Both $\Cal L [\upsilon^+]$ and $\Cal L [\upsilon^-]$ give the same ramified solution $y(x,\sqrt\epsilon)$  on a domain $X(\!\sqrt\epsilon)$ in the $x$-plane (Figure~\ref{figure:BL-4a}).

\smallskip
\textit{ii)} If the spectrum of $M$ is of Poincaré type and $\upsilon^\pm(\xi,\sqrt\epsilon)$ is defined on $\Omega_1$ as in Proposition~\ref{lemma:BL-4}, 
with $\boxnorm{\upsilon^+}_{\Omega_1}^\Lambda<+\infty$,
then, for $x\in X_1(\!\sqrt\epsilon)\,\cap\,\{\Re(e^{i\arg\sqrt\epsilon}t(x,\epsilon))\!<\!-\Lambda\}$,
one may deform the integration path of the Laplace transform \eqref{eq:BL-ypm} to $\Gamma$, indicated in Figure~\ref{figure:BL-7},
and use the Cauchy formula to express $y^\pm(x,\sqrt\epsilon)$, for $\sqrt\epsilon\neq 0$, as a sum of residues at the points $\xi=-2k\sqrt\epsilon$, $k\in\N_{>0}$,
\begin{align}
\Cal L [\upsilon^\pm](x,\sqrt\epsilon)&=\int_\Gamma \upsilon^\pm(\xi,\sqrt\epsilon)\, e^{-t(x,\epsilon)\xi}\,d\xi=
2\pi i\sum_{k=1}^\infty \res_{-2k\sqrt\epsilon}\, \upsilon^\pm\cdot \left(\tfrac{x+\sqrt\epsilon}{x-\sqrt\epsilon}\right)^k \nonumber\\
&=-2\sqrt\epsilon\sum_{k=1}^\infty \left(\tfrac{\upsilon^\pm}{\chi^\pm}\right)\!(-2k\sqrt\epsilon,\sqrt\epsilon)\cdot \left(\tfrac{x+\sqrt\epsilon}{x-\sqrt\epsilon}\right)^k.
\label{eq:BL-101}
\end{align}
This series is convergent for $\big|\frac{x\!+\!\sqrt\epsilon}{x\!-\!\sqrt\epsilon}\big|<e^{-2\sqrt{|\epsilon|}\Lambda}$, and 
its coefficients are the same in both cases $\upsilon^+$ and $\upsilon^-$.
It defines a solution $y_1(x,\sqrt\epsilon)$ of \eqref{eq:BL-cm} on a domain $X_1(\!\sqrt\epsilon)$, analytic at $x=-\sqrt\epsilon$ and ramified at $x=\sqrt\epsilon$ (Figure~\ref{figure:BL-4b}).
\qed


\subsubsection{Proof of Proposition~\ref{proposition:BL-1}.}

Note first that for any integrable function $\phi:e^{i\alpha}\R\to\C$ with bounded $|\cdot|_{e^{i\alpha}\R}^\Lambda$-norm,
the difference between the two-sided Laplace transform $\Cal L_\alpha[\phi](t)$ and its truncation of the corresponding integral to $[-Re^{i\alpha},Re^{i\alpha}]$
can be estimated, for $t$ in the strip of convergence
$$\mathbf{T}_\alpha^+(\Lambda,\sqrt\epsilon)=\{\Lambda<\Re(e^{i\alpha} t) < -\Re(\tfrac{e^{i\alpha}\pi i}{\sqrt\epsilon})-\Lambda \}, $$
 by
\begin{align*}
\bigg| \Cal L_\alpha[\phi](t)-\int_{-Re^{i\alpha}}^{Re^{i\alpha}} \!\!\phi(\xi)\,e^{-t\xi}\,d\xi \bigg| 
&\leq |\phi|_{e^{i\alpha}\R}^\Lambda \bigg(\int_{-\infty}^{-R} \!\!\!e^{d_\alpha''(t,\epsilon)s}ds+
\int_{R}^{+\infty} \!\!\!\!e^{-d_\alpha'(t,\epsilon)s}ds \bigg) \\
&\leq \frac{2|\phi|_{e^{i\alpha}\R}^\Lambda}{d_\alpha(t,\epsilon)}e^{-d_\alpha(t,\epsilon)R},
\end{align*} 
where
\begin{equation}
d_\alpha'(t,\epsilon)=\Re(e^{i\alpha}t)-\Lambda>0,\quad 
d_\alpha''(t,\epsilon)=-\Re(\tfrac{e^{i\alpha}\pi i}{\sqrt\epsilon})-\Re(e^{i\alpha}t)-\Lambda>0,
\end{equation}
and
\begin{equation*}
d_\alpha(t,\epsilon)=\min\{d_\alpha'(t,\epsilon),\ d_\alpha''(t,\epsilon)\}
\end{equation*}
is the distance of $t$ from the border of the strip $\mathbf{T}_\alpha^+(\Lambda,\sqrt\epsilon)$.

Therefore to estimate the difference between 
$\bar y^+(t,\sqrt\epsilon)=\Cal L_{\bar\alpha}[\bar\upsilon^+](t)$ and
$\tilde y^-(t,\sqrt\epsilon)=\Cal L_{\tilde\alpha}[\tilde\upsilon^-](t)$
on $\mathbf{T}_{\bar\alpha}^+(\Lambda,\sqrt{\bar\epsilon})\cap\mathbf{T}_{\tilde\alpha}^-(\Lambda,\sqrt{\tilde\epsilon})=\mathbf{T}_{\bar\alpha}^+(\Lambda,\sqrt{\bar\epsilon})\cap\mathbf{T}_{\tilde\alpha}^+(\Lambda,\sqrt{\bar\epsilon})$
we need only to estimate the difference between the truncated integrals.

For $\epsilon=0$, the two solutions $\bar\upsilon^+(\cdot,0)$ and $\tilde\upsilon^-(\cdot,0)$ agree on the disc 
of radius $R$ not containing any eigenvalue of $M$, and the two Laplace integrals can be compared directly.
For $\epsilon\neq 0$ this is no longer true. 
Instead, we will construct a set $\bar\upsilon^+=\phi_0,\phi_1,\ldots,\phi_{N-1},\phi_N=\tilde\upsilon^-$
of approximative solutions to \eqref{eq:BL-bpcm}
defined on some fixed neighborhoods $U_j$ of $[-Re^{i\alpha_j},Re^{i\alpha_j}]$,
$\bar\alpha=\alpha_0<\alpha_1<\ldots<\alpha_{N-1}<\alpha_N=\tilde\alpha$,
covering the double-sector
$\bigcup_{\alpha\in[\bar\alpha,\tilde\alpha]}[-Re^{i\alpha},Re^{i\alpha}]$
that satisfy
\begin{itemize}
\item[i)] $|\phi_j(\xi)|\leq K_1\frac{e^{\Lambda|\xi|}}{1+|e^{\frac{\xi\pi i}{\sqrt\epsilon} }|}, \quad\xi\in U_j,$
\item[ii)] $|\phi_{j-1}(\xi)-\phi_j(\xi)|\leq K_2\frac{e^{\Lambda|\xi|+(2\Lambda-\lambda_\alpha)(R-|\xi|)}}{1+|e^{\frac{\xi\pi i}{\sqrt\epsilon} }|}$
for $\xi\in[-Re^{i\alpha},Re^{i\alpha}]\subset U_{j-1}\cap U_j,$
\end{itemize}
where $\lambda_\alpha:=-\Re(\frac{e^{i\alpha}\pi i}{\sqrt\epsilon}).$

\medskip
We can then estimate
\begin{align*}
&\int_{-Re^{i\alpha}}^{Re^{i\alpha}} \frac{e^{\Lambda|\xi|+(2\Lambda-\lambda_\alpha)(R-|\xi|)}}{1+|e^{\frac{\xi\pi i}{\sqrt\epsilon} }|} \,|e^{-t\xi}|\,d|\xi| \\
&\leq \int_0^R e^{\Lambda s+(2\Lambda-\lambda_\alpha)(R-s)-\Re(e^{i\alpha}t)s}ds+
\int_0^{-R} e^{-\Lambda s+\lambda_\alpha s+(2\Lambda-\lambda_\alpha)(R+s)-\Re(e^{i\alpha}t)s}ds \\
&= \tfrac{1}{d_\alpha''}(e^{-d'_\alpha R}-e^{-(d'_\alpha+d_\alpha'')R})+\tfrac{1}{d'_\alpha}(e^{-d_\alpha'' R}-e^{-(d'_\alpha+d_\alpha'')R})\ 
\leq\ \tfrac{2}{d_\alpha}e^{-d_\alpha R},
\end{align*}
and
\begin{align*}
\int_{Re^{i\bar\alpha}}^{Re^{i\tilde\alpha}} \frac{e^{\Lambda|\xi|}}{1+|e^{\frac{\xi\pi i}{\sqrt\epsilon} }|}|e^{-t\xi}|\,d|\xi| 
+ \int_{-Re^{i\bar\alpha}}^{-Re^{i\tilde\alpha}} \frac{e^{\Lambda|\xi|}}{1+|e^{\frac{\xi\pi i}{\sqrt\epsilon} }|}|e^{-t\xi}|\,d|\xi| 
\leq 2R(\tilde\alpha-\bar\alpha)\, e^{-dR},
\end{align*}
where
$d=\min_{\alpha\in[\bar\alpha,\tilde\alpha]} d_\alpha$.
Combining these estimates results in the estimate \eqref{eq:BL-difference} for
$\|\tilde y^-(t,\sqrt\epsilon)-\bar y^+(t,\sqrt\epsilon)\|$.
The estimate $\|\tilde y^+(t,\sqrt\epsilon)-\bar y^-(t,\sqrt\epsilon)\|$ is symmetric.

The approximate solutions $\phi_j$ are constructed as in the proof of Proposition~\ref{proposition:BL-4} as a fixed point of the operator $\Cal G_{\alpha_j}^+$ \eqref{eq:BL-operatorG} but this time with the convolution $*$ in the direction $\alpha=\alpha_j$ replaced by its symmetric truncation 
\begin{equation*}
[\phi*\psi]_{\alpha}^R(\xi):=\left\{ \begin{array}{ll}     
     \displaystyle\int_{\xi-Re^{i\alpha}}^{Re^{i\alpha}}\phi(\xi-s)\psi(s)\,ds, & \text{ if }\ \xi\in[0,Re^{i\alpha}], \\[12pt]
     \displaystyle\int_{-Re^{i\alpha}}^{Re^{i\alpha}+\xi}\!\!\phi(\xi-s)\psi(s)\,ds, & \text{ if }\ \xi\in[-Re^{i\alpha},0],
        \end{array}\right.
\end{equation*}
For any integrable bounded function $\phi:[-Re^{i\alpha},Re^{i\alpha}]\to\C$ we can still define its norms
$|\phi|_{e^{i\alpha}\R}^\Lambda$ and $\|\phi\|_{e^{i\alpha}\R}^\Lambda$ by setting $\phi(\xi)=0$ outside of the interval, and hence use the same Youngs' inequalities for the convolution as before. 
One can then again prove that for each $\alpha$ there is a fixed point solution for the truncated version of the convolution operator $\Cal G_{\alpha}^+$ on the interval $[-Re^{i\alpha},Re^{i\alpha}]$. The following lemma implies that the difference between such solution of the truncated convolution equation and a true solution on 
$e^{i\alpha}\R$, when the latter one exist, is uniformly bounded  by  $K\frac{e^{\Lambda|\xi|+(2\Lambda-\lambda_\alpha)(R-|\xi|)}}{1+|e^{\frac{\xi\pi i}{\sqrt\epsilon} }|}$ on the interval.
One then extends analytically such the truncated solutions on small double-sectors around their interval of definition,
so that the difference of each two of them $\phi_{j-1},\phi_j$ with sufficiently close angles $\alpha_{j-1},\alpha_j$ has the same kind of uniform bound.

\begin{lemma}  Suppose that $\epsilon$ is small enough so that $\lambda_\alpha>2\Lambda$.

i) 
For $\phi,\psi:e^{i\alpha}\R\to\C$  integrable and with bounded $|\cdot|_{e^{i\alpha}\R}^\Lambda$-norm
$$|[\phi*\psi]_\alpha(\xi)-[\phi*\psi]_{\alpha}^R(\xi)|\leq
\tfrac{4}{\lambda_\alpha-2\Lambda}|\phi|_{e^{i\alpha}\R}^\Lambda  |\psi|_{e^{i\alpha}\R}^\Lambda\cdot
\frac{e^{\Lambda|\xi|+(2\Lambda-\lambda_\alpha)(R-|\xi|)}}{1+|e^{\frac{\xi\pi i}{\sqrt\epsilon} }|}$$

ii) If $\phi,\psi:[-Re^{i\alpha},Re^{i\alpha}]\to\C$ be integrable and bounded, with 
$$|\phi(\xi)|\leq K_1\frac{e^{\Lambda|\xi|}}{1+|e^{\frac{\xi\pi i}{\sqrt\epsilon} }|}, \qquad\text{and}\qquad
|\psi(\xi)|\leq K_2\frac{e^{\Lambda|\xi|+(2\Lambda-\lambda_\alpha)(R-|\xi|)}}{1+|e^{\frac{\xi\pi i}{\sqrt\epsilon} }|},$$  
then
$$|[\phi*\psi]_{\alpha}^R(\xi)|\leq
4K_1K_2R\,\frac{e^{\Lambda|\xi|+(2\Lambda-\lambda_\alpha)(R-|\xi|)}}{1+|e^{\frac{\xi\pi i}{\sqrt\epsilon} }|}$$
\end{lemma}
\begin{proof}
\textit{i)} If $\xi\in[0,Re^{i\alpha}]$ we can estimate
\begin{align*}
\int_{Re^{i\alpha}}^{+\infty e^{i\alpha}}\frac{e^{\Lambda(|\xi-s|+|s|)}}
{(1+e^{\frac{(\xi-s)\pi i}{\sqrt\epsilon}})(1+e^{\frac{s\pi i}{\sqrt\epsilon}})} ds &\leq
\int_R^{+\infty} e^{\Lambda(2s-|\xi|) +\lambda_\alpha(|\xi|-s) } ds\\ 
&\leq \frac{2}{\lambda_\alpha-2\Lambda} \cdot
\frac{e^{\Lambda|\xi|+(2\Lambda-\lambda_\alpha)(R-|\xi|)}}{1+|e^{\frac{\xi\pi i}{\sqrt\epsilon} }|},
\end{align*}
and by symmetry $s\mapsto \xi-s$ the same holds for the integral from $-\infty e^{i\alpha}$ to$\xi-Re^{i\alpha}$. Similarly for $\xi\in[-Re^{i\alpha},0]$.

\textit{ii)}  If $\xi\in[0,Re^{i\alpha}]$ we can estimate
\begin{align*}
&\int_{\xi-Re^{i\alpha}}^{R e^{i\alpha}}\!\!\!\frac{e^{\Lambda(|\xi-s|+|s|)+(2\Lambda-\lambda_\alpha)(R-|s|)}}
{(1+e^{\frac{(\xi-s)\pi i}{\sqrt\epsilon}})(1+e^{\frac{s\pi i}{\sqrt\epsilon}})} ds \,\leq\! 
\int_{|\xi|}^R \!\!e^{\Lambda(2s-|\xi|) +\lambda_\alpha(|\xi|-s)+ (2\Lambda-\lambda_\alpha)(R-s)} ds \\
&\qquad\quad\ + \int_0^{|\xi|} e^{\Lambda|\xi| + (2\Lambda-\lambda_\alpha)(R-s)} ds \ 
+  \int_{-R+|\xi|}^0 \!\! e^{-\Lambda(2s-|\xi|) +\lambda_\alpha s+ (2\Lambda-\lambda_\alpha)(R+s)} ds \\ 
&\leq \left[(R-|\xi|)+|\xi|+(R-|\xi|)\right] e^{\Lambda|\xi|+(2\Lambda-\lambda_\alpha)(R-|\xi|)} 
\,\leq\, 4R \frac{e^{\Lambda|\xi|+(2\Lambda-\lambda_\alpha)(R-|\xi|)}}{1+|e^{\frac{\xi\pi i}{\sqrt\epsilon} }|},
\end{align*}
using that $(2\Lambda-\lambda_\alpha)<0$. Similarly for $\xi\in[-Re^{i\alpha},0]$.
\qed\end{proof}

\begin{corollary}\label{corollary:BL-3}
If $\upsilon_{i}^\pm$, $i=1,2$, are two solution of the convolution equations \eqref{eq:BL-bpcm}
of Proposition~\ref{proposition:BL-3} on domains $\Omega_i$, then their difference
is exponentially flat in $\sqrt\epsilon$:
$$|\upsilon_{1}^\pm(\xi,\sqrt\epsilon)-\upsilon_{2}^\pm(\xi,\sqrt\epsilon)|
\leq K\frac{e^{\Lambda|\xi|+(2\Lambda-\lambda_\alpha)(R-|\xi|)}}{1+|e^{\frac{\xi\pi i}{\sqrt\epsilon} }|}
\quad\text{ for }\ \xi\in\Omega_{1}(\!\sqrt\epsilon)\cap\Omega_{2}(\!\sqrt\epsilon)
$$
for some $K>0$.
\end{corollary}

%

\begin{acknowledgements}
I am very grateful to Christiane Rousseau for many helpful discussions and to Reinhard Sch\"afke and Loïc Teyssier for their interest in my work. 
The paper was prepared during my doctoral studies at Université de Montreal and finalized during my stay at Université de Strasbourg -- 
I want to thank both institutions for their hospitality.
\end{acknowledgements}

\end{document}